\pdfoutput=1
\documentclass[11pt]{article}
\usepackage[a4paper,margin=3cm]{geometry}
\usepackage{amsmath, amsthm, amssymb}
\usepackage{graphicx}
\usepackage{caption}
\usepackage{subcaption}
\usepackage{txfonts}
\usepackage[pdfpagemode=UseOutlines]{hyperref}
\hypersetup{
    colorlinks=false,
    pdfborder={0 0 0},
    pdfauthor={B. de Rijk, A. Doelman, J. Rademacher},
    pdftitle={Evans function factorization via Riccati transform}
}
\usepackage{enumitem}
\makeatletter
\def\namedlabel#1#2{\begingroup
    #2%
    \def\@currentlabel{#2}%
    \phantomsection\label{#1}\endgroup
}
\makeatother
\newcommand{\R}{\mathbb{R}}\newcommand{\Z}{\mathbb{Z}}\newcommand{\F}{\mathcal{F}}\newcommand{\C}{\mathbb{C}}

\newcommand{\p}{\mathrm{p}}
\newcommand{\s}{\mathrm{s}}
\newcommand{\ho}{\mathrm{h}}
\newcommand{\K}{\mathcal{K}}
\newcommand{\El}{\mathcal{L}}
\newcommand{\h}{\mathcal{H}}
\newcommand{\ord}{\mathcal{O}}
\newcommand{\D}{\mathcal{D}}
\newcommand{\Es}{\mathcal{S}}
\newcommand{\Proj}{\mathcal{P}}

\newcommand{\U}{\mathcal{U}}
\newcommand{\E}{\mathcal{E}}
\newcommand{\A}{\mathcal{A}}
\newcommand{\J}{\mathcal{J}}
\newcommand{\M}{\mathcal{M}}
\newcommand{\N}{\mathcal{N}}
\newcommand{\B}{\mathcal{B}}
\newcommand{\G}{\mathcal{G}}
\newcommand{\Ce}{\mathcal{C}}
\newcommand{\T}{\mathcal{T}}
\newcommand{\X}{\mathcal{X}}
\newcommand{\V}{\mathcal{V}}
\newcommand{\I}{\mathcal{I}}
\newcommand{\Ef}{\mathcal{F}}
\newcommand{\W}{\mathcal{W}}
\newcommand{\Ze}{\mathcal{Z}}
\newcommand{\eps}{\epsilon}
\newcommand{\Y}{\mathcal{Y}}
\numberwithin{equation}{section}

\let\eps\epsilon
\let\epsilon\varepsilon
\let\pih\phi
\let\phi\varphi
\def\xx{{\check{x}}} 
\def\yy{{\check{y}}} 
\def\zz{{\check{z}}} 
\def\ww{{\check{w}}} 

\newtheoremstyle{theoremsty}{14pt}{14pt}{\itshape}{}{\bfseries}{.}{.5em}{}

\newtheoremstyle{definitionsty}{14pt}{14pt}{\normalfont}{}{\bfseries}{.}{.5em}{}

\theoremstyle{theoremsty}
\newtheorem{theo}{Theorem}[section]
\newtheorem{lem}[theo]{Lemma}
\newtheorem{prop}[theo]{Proposition}
\newtheorem{cor}[theo]{Corollary}

\theoremstyle{definitionsty}
\newtheorem{exa2}[theo]{Example}
\newtheorem{defi}[theo]{Definition}
\newtheorem{nota}[theo]{Notation}

\newtheorem{rem}[theo]{Remark}

\makeatletter
\renewenvironment{proof}[1][\proofname] {\par\pushQED{\qed}\normalfont\topsep6\p@\@plus6\p@\relax\trivlist\item[\hskip\labelsep\bfseries#1\@addpunct{.}]\ignorespaces}{\popQED\endtrivlist\@endpefalse} \makeatother

\title{Spectra and stability of spatially periodic pulse patterns: Evans function factorization via Riccati transformation\thanks{This work was partially supported by the Dutch science foundation (NWO) cluster NDNS+.}}
\author{Bj\"orn de Rijk\footnotemark[2] \\ \texttt{brijk@math.leidenuniv.nl}
\and Arjen Doelman\footnotemark[2] \\ \texttt{doelman@math.leidenuniv.nl}
\and Jens Rademacher\footnotemark[3] \\ \texttt{rademach@math.uni-bremen.de}}
\date{\vspace{-2em}}

\begin{document}
\maketitle

\renewcommand{\thefootnote}{\fnsymbol{footnote}}
\footnotetext[2]{Mathematisch Instituut, Universiteit Leiden, P.O. Box 9512, 2300 RA Leiden, The Netherlands.}
\footnotetext[3]{Universit\"at Bremen, Fachbereich Mathematik, Postfach 33 04 40, 28359 Bremen, Germany.}

\begin{abstract}In the spectral stability analysis of localized patterns to singular perturbed evolution problems, one often encounters that the Evans function respects the scale separation. In such cases the Evans function of the full linear stability problem can be approximated by a product of a slow and a fast reduced Evans function, which correspond to properly scaled slow and fast singular limit problems. This feature has been used in several spectral stability analyses in order to reduce the complexity of the linear stability problem. In these studies the factorization of the Evans function was established via geometric arguments that need to be customized for the specific equations and solutions under consideration. In this paper we develop an alternative factorization method. In this analytic method we use the Riccati transformation and exponential dichotomies to separate slow from fast dynamics. We employ our factorization procedure to study the spectra associated with spatially periodic pulse solutions to a general class of multi-component, singularly perturbed reaction-diffusion equations. Eventually, we obtain expressions of the slow and fast reduced Evans functions, which describe the spectrum in the singular limit. The spectral stability of localized periodic patterns has so far only been investigated in specific models such as the Gierer-Meinhardt equations. Our spectral analysis significantly extends and formalizes these existing results. Moreover, it leads to explicit instability criteria.
\end{abstract}

\thispagestyle{empty}

\section{Introduction}
Localized patterns arise frequently in evolution problems with a strong spatial scale separation, which naturally leads to the question of their dynamic stability. Several methods have been developed to study the often decisive spectral stability, especially for the paradigmatic class of parabolic semi-linear reaction-diffusion systems on the line of the form,
\begin{align}
\left\{\!\begin{array}{rcl} \partial_t u &=& D_1\partial_{\xx\xx} u - H(u,v,\epsilon)\\
\partial_t v &=& \epsilon^2 D_2\partial_{\xx\xx} v - G(u,v,\epsilon)\end{array}\right., \ \ \ u \in \R^m, v \in \R^n,\label{reac}
\end{align}
where $0 < \epsilon \ll 1$ is asymptotically small and $D_{1,2}$ are (strictly) positive diagonal matrices. The majority of these methods are built on the complex analytic Evans function $\E_\epsilon(\lambda)$, which vanishes precisely on the spectrum and thus is a tool to locate the critical spectrum -- see \cite{AGJ,EVA} and Remark \ref{applitrem}. \\
\\
The singular nature of the equations and the patterns under consideration can reduce the complexity of finding the roots of the Evans function: the slow-fast structure in the linear stability problem induces a factorization of the Evans function into a slow and a fast component,
\begin{align}
\E_\epsilon(\lambda) = \E_{s,\epsilon}(\lambda)\cdot \E_{f,\epsilon}(\lambda), \label{factorintr}
\end{align}
as was first observed by Alexander, Gardner and Jones \cite{AGJ} in the context of traveling pulses in the FitzHugh-Nagumo equations. While operating in a geometric framework, they showed that the unstable bundle formed from the projectivized linear stability equations splits into a Whitney sum of two line bundles, which are governed by slow and fast singular limit problems. This splitting can be directly linked to the factorization (\ref{factorintr}) of the Evans function in a slow and a fast component. However, although the geometric arguments behind the decomposition in \cite{AGJ} are very general, they need to be based on an analytical result that indeed controls the relevant subbundles, which has been established in \cite{AGJ,JON} in the context of the FitzHugh-Nagumo equations.  \\
\\
In subsequent work, Gardner and Jones \cite{GJO} validated the geometric argument of \cite{AGJ} and thus the factorization (\ref{factorintr}), in the context of traveling fronts in a predator-prey model. They tracked the fast subbundle analytically via the so-called `elephant trunk lemma'. Then, using the control on the fast subbundle, they approximated the slow subbundle. Further technical adaptations of the elephant trunk lemma to spectral stability problems associated with localized homoclinic or heteroclinic structures in the Gray-Scott and Fabry-P\'erot model have been developed in \cite{DGK} and \cite{RUB}, respectively. Nowadays, it is widely accepted that the elephant trunk procedure can be mimicked -- or better: adapted -- for a large class of singularly perturbed systems. However, for every application one should in principle go through one of the extensive proofs developed in the setting of the aforementioned specific systems to check whether technicalities still hold true.
\\ \\
By tracking the fast subsystem/subbundle through the elephant trunk procedure, it is possible to derive an explicit complex analytic fast reduced Evans function $\E_{f,0}(\lambda)$ whose zeros approximate those of $\E_{f,\epsilon}(\lambda)$ -- see also \cite{DGK2}. Moreover, an explicit, but meromorphic, slow reduced Evans function $\E_{s,0}(\lambda)$ can be obtained via the so-called `NLEP (= NonLocal Eigenvalue Problem) approach', which was developed in \cite{DGK} in the context of stability of homoclinic pulses in the Gray-Scott model. Thus, a combination of the elephant trunk procedure and the NLEP approach yields a \emph{reduced Evans function},
\begin{align}
\E_0(\lambda) = \E_{s,0}(\lambda)\cdot \E_{f,0}(\lambda),
\label{factorintrred}
\end{align}
whose zeros approximate those of $\E_\epsilon(\lambda)$ and whose factors $\E_{s,0}(\lambda)$ and $\E_{f,0}(\lambda)$ can be derived explicitly as $\epsilon \to 0$ through properly scaled slow and fast singular limit problems. The NLEP approach and thus the validation of the decomposition (\ref{factorintr}) and its explicit reduction (\ref{factorintrred}), was further developed in the context of localized homoclinic or heteroclinic pulses or fronts in certain classes of $2$- or $3$-component singularly perturbed reaction-diffusion equations in \cite{DGK2,DIN,VEA,HEI,VEE}. It should be remarked that in neither of these papers the elephant trunk procedure is worked out in full analytical detail. \\
\\
This leads us to the main goal of this paper. We establish the validity of both the decomposition (\ref{factorintr}) of the Evans function $\E_\epsilon(\lambda)$ and its singular limit structure (\ref{factorintrred}) for a large class of nonlinearities $H(u,v,\epsilon)$, $G(u,v,\epsilon)$ in system (\ref{reac}) and general dimensions $n,m\geq 1$. Hence, we provide a generalized analytic alternative to both the geometric elephant trunk and NLEP procedures. The method presented here is based on the Riccati transformation \cite{CHA,CH2}. This transformation, which satisfies a matrix Riccati equation, diagonalizes the linear stability problem and thus explicitly separates fast from slow dynamics. This separation yields the factorization of the Evans function (\ref{factorintr}) and provides a framework for the passage to the singular limit (\ref{factorintrred}).
\newpage
The factorization procedure can easily be outlined in a way that does neither depend on the specific structure of the system nor on the specific patterns under consideration. The first step is to write the spectral problem associated with the linearization of (\ref{reac}) about a singular, localized pattern,
\begin{align*}
 \El_\epsilon \left(\begin{array}{c} u \\ v \end{array}\right) = \left(\begin{array}{c} D_1 \partial_{\xx\xx} u \\ \epsilon^2 D_2 \partial_{\xx\xx} v\end{array}\right) - \B_\epsilon(\xx) \left(\begin{array}{c} u \\ v \end{array}\right) = \lambda \left(\begin{array}{c} u \\ v \end{array}\right),
\end{align*}
as a $2(n+m)$-dimensional slow-fast system in block matrix form,
\begin{align} \left(\begin{array}{c} \partial_x \phi \\ \partial_x \psi \end{array}\right) = \left(\begin{array}{cc} \sqrt{\epsilon} \A_{11,\epsilon}(x,\lambda) & \sqrt{\epsilon} \A_{12,\epsilon}(x,\lambda) \\ \A_{21,\epsilon}(x,\lambda) & \A_{22,\epsilon}(x,\lambda) \end{array}\right)\left(\begin{array}{c} \phi \\ \psi \end{array}\right), \label{LSPintro}\end{align}
in the rescaled spatial variable $x = \epsilon^{-1}\xx$. Subsequently, one establishes that the $2n$-dimensional fast singular limit problem,
\begin{align}\partial_x \psi = \A_{22,0}(x,\lambda) \psi, \label{fastintro}\end{align}
in which $\A_{22,0}(x,\lambda)$ represents the singular limit of $\A_{22,\epsilon}(x,\lambda)$, has exponential dichotomies on both half lines due to the localized nature of the solutions under consideration. These dichotomies can be pasted together to an exponential dichotomy on the whole real line, as long as $\lambda$ is not an eigenvalue of (\ref{fastintro}). In the third step, roughness techniques are used to carry the exponential dichotomy of (\ref{fastintro}) on $\R$ to the perturbed problem,
\begin{align}
\partial_x \psi = \A_{22,\epsilon}(x,\lambda) \psi, \label{pertintro}
\end{align}
for $0 < \epsilon \ll 1$ and $\lambda$ away from the eigenvalues of (\ref{fastintro}). This exponential dichotomy on $\R$ of (\ref{pertintro}) allows us to successfully diagonalize the linear stability problem (\ref{LSPintro}) with the Riccati transformation yielding a factorization (\ref{factorintr}) of the Evans function in a `slow' factor $\E_{s,\epsilon}(\lambda)$ and a `fast' factor $\E_{f,\epsilon}(\lambda)$. In the last step, we approximate the two blocks, in which (\ref{LSPintro}) diagonalizes, by their singular limits. Once we have understood these singular limit problems, we are able to explicitly obtain leading order expressions $\E_{s,0}(\lambda)$ and $\E_{f,0}(\lambda)$. As a consequence, the roots of the Evans function can be approximated by the roots of the reduced Evans function $\E_0(\lambda)$ given in (\ref{factorintrred}). Since $\E_{s,0}(\lambda)$ is meromorphic and $\E_{f,0}(\lambda)$ is analytic, one determines the location of the spectrum by calculating the roots and poles of $\E_{s,0}(\lambda)$ and $\E_{f,0}(\lambda)$. \\
\\
In summary, the complex task of finding the roots of the Evans function $\E_\epsilon(\lambda)$ associated with the $2(n+m)$-dimensional system (\ref{LSPintro}) reduces to determining the roots and poles of $\E_{s,0}(\lambda)$ and $\E_{f,0}(\lambda)$, which can be derived from lower-dimensional singular limit problems. More precisely, $\E_{f,0}(\lambda)$ is an Evans function associated with the $2n$-dimensional fast singular limit problem (\ref{fastintro}). In contrast, $\E_{s,0}$ is an explicit expression in terms of a particular solution to an inhomogeneous version of (\ref{fastintro}) and the leading order evolution of the $2m$-dimensional slow subsystem $\partial_x \phi = \sqrt{\epsilon} \A_{11,\epsilon}(x,\lambda)\phi$. We emphasize that some of the analytic techniques used in the factorization process can be linked to the geometric concepts developed in \cite{AGJ,GJO} in the setting of the elephant trunk procedure and in \cite{DGK,DGK2} for the NLEP approach, which is discussed at the end this paper.
\\
\\
We employ our factorization method to determine the critical spectra associated with stationary, spatially periodic pulse solutions to the general class of systems \eqref{reac}. Recall that the proof of the elephant trunk lemma has only been worked out in full analytic detail for some specific $2$-component systems \cite{DGK,ESZ,GJO,RUB}. Hence, the choice for a large class of multi-component reaction-diffusion systems illustrates the general setting to which our method applies. In addition, the choice for \emph{periodic} patterns is motivated by the fact that for certain nonlinearities $H(u,v,\epsilon)$, there is no simple modification of the elephant trunk procedure that works in this setting. We elaborate on the latter claim.
\newpage
It is a general principle that pattern solutions to singularly perturbed models as \eqref{reac} can be `built' from exponentially localized pulses (or fronts) in the fast component(s), but allow for non-localized behavior of the slow components. In other words, the solutions exhibit semi-strong interactions \cite{DOK} (of second order \cite{RAD}). Most spectral analyses, which make use of the elephant trunk and NLEP procedures, are of `slowly linear' nature, in the sense that the dynamics of the slow component in between localized fast pulses/fronts are driven by linear equations -- see Remark \ref{slowlynonlinear}. This slow (non)linearity plays a crucial role in the analysis of the Evans function and its decomposition and reduction. In fact, it is essential for an application of the elephant trunk procedure to periodic patterns that the matrix in the linear stability problem \eqref{LSPintro} is to leading order of constant coefficient type near the boundaries of the spatial domain (determined by the periodicity of the pattern) -- as is the case in \cite{PLO}. In slowly nonlinear systems, the matrix \eqref{LSPintro} explicitly varies in $x$ over the entire domain, thus obstructing an application of the elephant trunk lemma.\\
\\
Before applying our factorization method to study the spectra of these periodic pulse solutions, we consider their existence. An essential observation is that stationary solutions to \eqref{reac} satisfy an ordinary differential equation that admits a reversible symmetry $\xx \mapsto -\xx$. With the aid of geometric singular perturbation theory we establish the existence of reversible symmetric periodic pulse solutions to the $(m+n)$-component slowly nonlinear system \eqref{reac}. This by itself is a significant extension of similar results in the literature that only consider $2$-component slowly linear Gierer-Meinhardt type models \cite{PLO2}. However, we emphasize that our spectral analysis does not rely on reversible symmetry arguments. In fact, we do not assume that the underlying spatially periodic patterns are reversible. This makes an extension of our methods to models with (symmetry-breaking) convective terms or to traveling wave trains natural.\\
\\
Our general results can also be interpreted in more simple cases in which either $n=1$, $m=1$, or both $n=m=1$. In the latter case, we directly recover the expressions obtained in \cite{PLO} for the spectral stability of spatially periodic pulse patterns in the Gierer-Meinhardt equation. The outcome of our spectral analysis shows that the $n=m=1$ Gierer-Meinhardt setting represents a very special case. This restriction hides the underlying general structure of $\E_{s,0}(\lambda)$ and $\E_{f,0}(\lambda)$ in terms of the singular limit problems as obtained here. On the other hand, the restriction of \eqref{reac} to a more general slowly nonlinear $2$-component model as in \cite{VEA} yields a (relatively) simple instability criterion in terms of the signs of a number of explicit integral expressions that can be computed with only an asymptotic approximation of the underlying pattern as input. Thereby, we extend a similar result of \cite{VEA} on homoclinic pulses to spatially periodic patterns. We also refer to an companion paper in preparation \cite{RIJV}, in which the nature of the mechanisms driving the destabilization of spatially periodic patterns in general slowly nonlinear systems \eqref{reac} -- but with $n=m=1$ -- are studied in full analytical (and computational) detail. Thereby, it provides insight in the generic nature of the `Hopf and belly dances' discovered in \cite{STE} in the context of the slowly linear Gray-Scott and Gierer-Meinhardt models. \\
\\
This paper is organized as follows. In Section \ref{sec1} we introduce the class of reaction-diffusion systems under consideration. Moreover, we elaborate on the existence of periodic pulse solutions. Section \ref{sec2} contains the main results of our spectral analysis. In Section \ref{nm1} these results are expanded further in the case one of the components is scalar. The construction of the Riccati transformation is performed in Section \ref{sec5}. The actual spectral analysis via the analytic factorization method is presented in Section \ref{sec6}. Section \ref{sec7} contains some concluding remarks and future research possibilities. In Appendix \ref{A0} one can find the proof of the existence result stated in Section \ref{secexistence}. Moreover, in Appendix \ref{A1} we treat the prerequisites needed for our spectral analysis. Appendix \ref{A2} contains the proofs of some technical, but not fundamentally difficult, results in this paper.
\newpage
\begin{rem}[\textit{Slow nonlinearity}] \label{slowlynonlinear}
As mentioned earlier, most spectral analyses, using the elephant trunk and NLEP procedures, are done in models of `slowly linear' nature, which include the classical Gray-Scott and Gierer-Meinhardt models. In recent work \cite{VEA,VEE} an NLEP approach has been developed for the spectral analysis of homoclinic pulses to a general class of singularly perturbed slowly nonlinear $2$-component reaction-diffusion systems. Earlier, spectral stability of fronts was studied in a specific slowly nonlinear model in \cite{DIN}. Although the models in these works are slowly nonlinear, the elephant trunk procedure is still applicable, because eventually the dynamics of the slow component becomes linear due to homoclinic or heteroclinic nature of the patterns. However, we emphasize that there are significant adaptations necessary in the NLEP procedure for an extension of the slowly linear to the general slowly nonlinear case -- see \cite{VEA}. $\hfill \blacksquare$
\end{rem}
\begin{rem}[\textit{Spectral stability}] \label{specintr}
It should be noted that only studying the roots of $\E_{f,0}(\lambda)$ and $\E_{s,0}(\lambda)$ is not sufficient to decide upon spectral stability of the spatially periodic patterns, because there must be essential spectrum attached to zero: the leading order results of the present work do not provide information on the exact position of this `small spectrum' \cite{PLO} with respect to the imaginary axis. Therefore, an additional study of the fine structure of the spectrum around zero is necessary to determine spectral stability -- see Section \ref{secsmallspectrum}. This is subject of work in progress. $\hfill \blacksquare$
\end{rem}
\begin{rem}[\textit{The Evans function and other tools to locate the spectrum}]
\label{applitrem}
The concept of the Evans function as a method to determine the spectrum associated with a localized solution of a system of reaction-diffusion equations on the line was introduced in \cite{EVA} and was established as a general and powerful approach in \cite{AGJ,GJO,JON}. Core aspects of the NLEP approach have been developed independently in \cite{IWW,WAW}. The SLEP (= Singular Limit Eigenvalue Problem) method \cite{NI2,NI3} is an alternative method that has been linked to the Evans function approach in \cite{INS}. In \cite{DGK2}, the relation between the Evans function, the NLEP method and the SLEP method is discussed. The Evans function approach was originally developed in the context of localized homoclinic and/or heteroclinic patterns, it was first generalized to spatially periodic pulse patterns in reaction-diffusion equations in \cite{ESZ,GAR,PLO}. $\hfill \blacksquare$
\end{rem}
\section{Setting} \label{sec1}
\subsection{Reaction-diffusion equations with semi-strong interaction}\label{S2.1}
In this section we introduce the class of systems under consideration in this paper. Take $m, n \in \Z_{>0}$ and consider a general reaction-diffusion system in one space dimension with a scale separation in the diffusion lengths \eqref{reac}. Following \cite{VEA}, we write
\begin{align*}
H(u,v,\epsilon) &= H(u,0,\epsilon) + \tilde{H}_2(u,v,\epsilon),
\end{align*}
where $\tilde{H}_2(u,v,\epsilon) := H(u,v,\epsilon) - H(u,0,\epsilon)$, so that $\tilde{H}_2$ vanishes at $v = 0$. To sustain stable localized patterns in semi-strong interaction (of second order \cite{RAD}) in system (\ref{reac}), we allow $\tilde{H}_2(u,v,\epsilon)$ to scale with $\epsilon^{-1}$ and define
\begin{align*}
H_2(u,v) &:= \lim_{\epsilon \downarrow 0} \epsilon \tilde{H}_2(u,v,\epsilon).
\end{align*}
Finally, we write
\begin{align}
H(u,v,\epsilon) &= H_1(u,v,\epsilon) + \epsilon^{-1}H_2(u,v),\label{reac2}
\end{align}
with $H_1(u,v,\epsilon) := H(u,0,\epsilon) + [\tilde{H}_2(u,v,\epsilon) - \epsilon^{-1} H_2(u,v)]$. By construction $H_2(u,v)$ vanishes at $v = 0$. Moreover, we assume that $H_1(u,v,\epsilon)$ and $G(u,v,\epsilon)$ are smooth functions of $\epsilon$ at $\epsilon = 0$. Note that we allow for the possibility that $H_2(u,v) \equiv 0$ in the upcoming analysis. We emphasize that, when both $H_2(u,v) \equiv 0$ and $n = 1$, all patterns are unstable -- see Remark \ref{weakint}. This confirms the scalings used for classical systems as the Gray-Scott and Gierer-Meinhardt models \cite{DGK,DGK2,IWW,WAW} -- see also \cite{VEA}. For the benefit of our spectral analysis, we need one extra condition on $G$. That is, $G$ vanishes at $v = 0$. We postpone the discussion of this extra condition to Remark \ref{condG1}. In summary, the model class we consider is of the form
\begin{align} \left\{\!\begin{array}{rcl} \partial_t u &=& D_1\partial_{\xx\xx} u - H_1(u,v,\epsilon) - \epsilon^{-1}H_2(u,v)\\ \partial_t v &=& \epsilon^2 D_2\partial_{\xx\xx} v - G(u,v,\epsilon)\end{array}\right., \ \ \ u \in \R^m, v \in \R^n,\label{diff2} \end{align}
or, in the `small' spatial scale $x = \epsilon^{-1} \xx$,
\begin{align}\left\{\!\begin{array}{rcl} \epsilon^{2} \partial_t u &=& D_1 \partial_{xx} u - \epsilon^{2} H_1(u,v,\epsilon) - \epsilon H_2(u,v) \\
  \partial_t v &=& D_2 \partial_{xx} v - G(u,v,\epsilon)\end{array}\right., \ \ \ u \in \R^m, v \in \R^n, \label{diff}\end{align}
in which we will usually work. The aforementioned conditions read:
\smallskip
\begin{enumerate}
\item[\namedlabel{assS1}{\textbf{(S1)}}] \textit{\textbf{Conditions on the interaction terms}}\\
There exists open, connected sets $U \subset \R^m, V \subset \R^n$ and $I \subset \R$ with $0 \in V$ and $0 \in I$ such that $H_1,G$ and $H_2$ are $C^2$ on their domains $U \times V \times I$ and $U \times V$, respectively. Moreover, we have $H_2(u,0) = 0$ and $G(u,0,\epsilon) = 0$ for all $u \in U$ and $\epsilon \in I$.
\end{enumerate}
\begin{rem}[\textit{The quantities $D_{1,2}$ in the cases $m = 1$ and $n = 1$}] \label{quantD12}
If we have $n = 1$, we can without loss of generality assume $D_2 = 1$ in (\ref{diff2}) by rescaling the spatial variable $\xx$. Similarly, in the case $m = 1$, we can without loss of generality assume $D_1 = 1$ by rescaling the small parameter $\epsilon$. $\hfill \blacksquare$
\end{rem}
\subsection{Stationary, spatially periodic pulse solutions}
In this paper we are interested in the spectra associated with stationary, spatially periodic pulse solutions satisfying the reaction-diffusion system (\ref{diff}). Stationary solutions to (\ref{diff}) satisfy the singularly perturbed ordinary differential equation,
\begin{align}\left\{\!
\begin{array}{rcl} D_1 \partial_x u &=& \epsilon p \\ \partial_x p &=& \epsilon H_1(u,v,\epsilon) + H_2(u,v) \\ D_2 \partial_x v &=& q \\ \partial_x q &=& G(u,v,\epsilon)\label{ODE}\end{array}\right., \ \ \ u \in U, p \in \R^m, v \in V, q \in \R^n.\end{align}
If we take $\epsilon = 0$ in (\ref{ODE}), the dynamics is given by the so-called \emph{fast reduced system},
\begin{align}
\left\{\!
\begin{array}{rcl}
\partial_x u &=& 0\\
\partial_x p &=& H_2(u,v)\\
D_2 \partial_x v &=& q\\
\partial_x q &=& G(u,v,0)\end{array}\right., \ \ \ u \in U, p \in \R^m, v \in V, q \in \R^n. \label{relp}
\end{align}
We observe that $\M = \{(u,p,0,0) : u \in U, p \in \R^m\}$ consists entirely of equilibria of (\ref{relp}) by assumption \ref{assS1}. We require $\M$ to be normally hyperbolic.
\smallskip
\begin{enumerate}
\item[\namedlabel{assS2}{\textbf{(S2)}}] \textit{\textbf{Normal hyperbolicity}}\\
For each $u \in U$ the real part $\mathrm{Re}(\G(u)) = \frac{1}{2}\left(\G(u) + \G(u)^*\right)$ of $\G(u) := \partial_v G(u,0,0)$ is positive definite.
\end{enumerate}
\smallskip
When $\epsilon > 0$, the manifold $\M$ consists no longer of equilibria, but remains invariant. The flow restricted to $\M$ is to leading order governed by the so-called \emph{slow reduced system},
\begin{align} \left\{\!
\begin{array}{rcl}
D_1 \partial_\xx u &=& p\\
\partial_\xx p &=& H_1(u,0,0) \end{array}\right., \ \ \ u \in U, p \in \R^m.\label{slowp}
\end{align}
It is well-known that the dynamics around such a normally hyperbolic manifold $\M$ is captured by Fenichel's geometric singular perturbation theory \cite{FEN2,JOE}. Fenichel's theory can reduce the complexity of finding periodic orbits in the following way. Suppose we have obtained a so-called singular periodic orbit by piecing together orbit segments of the fast and slow reduced systems in such a way that they form a closed loop. Although this singular orbit is not a solution to the full system, one can prove with the aid of Fenichel's theory that an actual periodic orbit lies in the vicinity of the singular one, provided $\epsilon > 0$ is sufficiently small. This construction is performed in Section \ref{secexistence}. \\
\\
In this paper we are interested in periodic solutions to (\ref{ODE}) that are close to singular periodic orbits, which consists of two orbit segments, to wit: a pulse satisfying the fast reduced system (\ref{relp}) and a segment on the invariant manifold $\M$, satisfying the slow reduced system (\ref{slowp}). We make the hypotheses on the solutions precise in the following assumption.
\smallskip
\begin{enumerate}
\item[\namedlabel{assS3}{\textbf{(S3)}}] \textit{\textbf{Existence of a periodic pulse solution}}\\
For $\epsilon > 0$ sufficiently small, there exists a $2L_\epsilon$-periodic solution $\pih_{\p,\epsilon} \colon \R \to U \times \R^m \times V \times \R^n$ to system (\ref{ODE}) with the following three properties.
\begin{enumerate}
\item[\textbf{1.}] \textbf{Order of the period}\\
It holds $L_\epsilon = \epsilon^{-1}\check{L}_\epsilon$ and $\check{L}_\epsilon \to \check{L}_0>0$ as $\epsilon\to 0$.
\item[\textbf{2.}] \textbf{Existence of a pulse}\\
There exists $u_0 \in U, p_0 \in \R^m$ and a solution $(v_{\ho}(x),q_{\ho}(x))$ to system,
\begin{align}
\left\{\!
\begin{array}{rcl}
D_2 \partial_x v &=& q\\ \partial_x q &=& G(u_0,v,0)\end{array}\right., \ \ \ v \in V, q \in \R^{n}, \label{fastintr}
\end{align}
which is homoclinic to $(0,0)$, such that
\begin{align}\|\pih_{\p,\epsilon}(0) - (u_0,p_0,v_\ho(0),q_\ho(0))\| = \ord(\epsilon).  \label{fastestS3}\end{align}
\item[\textbf{3.}] \textbf{Exponential decay to a solution in $\M$}\\
There exists $L_s, \beta_s > 0$ and a solution $(u_{\s}(\xx),p_{\s}(\xx))$ to (\ref{slowp}) such that it holds for $x \in [0,2L_\epsilon]$
\begin{align} d(\pih_{\p,\epsilon}(x),\M) &\leq L_s e^{-\beta_s \min\{x,2L_\epsilon - x\}}, \label{slowestS31}\\
\|\pih_{\p,\epsilon}(x) - (u_{\s}(\epsilon x),p_\s(\epsilon x),0,0)\| &\leq L_s\max\left\{\epsilon,e^{-\beta_s \min\{x,2L_\epsilon - x\}}\right\}. \label{slowestS32}\end{align}
\end{enumerate}
\end{enumerate}
\begin{rem}[\textit{Reversible symmetry}] \label{reversiblerem}
Note that (\ref{ODE}) is $R$-reversible, where $R \colon \R^{2(m+n)} \to \R^{2(m+n)}$ is the reflection in the space $p = q = 0$. Indeed, if $\pih$ is a solution to (\ref{ODE}), then $x \mapsto R \pih(-x)$ is also a solution to (\ref{ODE}). Similarly, denote by $R_s \colon \R^{2m} \to \R^{2m}$ and $R_f \colon \R^{2n} \to \R^{2n}$ the reflections in $p = 0$ and $q = 0$, respectively. Systems (\ref{slowp}) and (\ref{fastintr}) are $R_s$- and $R_f$-reversible, respectively.\\
 \\
Reversible symmetry arguments significantly simplify our existence analysis in Section \ref{secexistence}. As a consequence, the constructed periodic orbits are symmetric about the space $p = q = 0$. However, reversible symmetry arguments do not essentially simplify our spectral analysis. In fact, these arguments could obscure an extension of our results to models with convective terms, which do not obey the reversible symmetry. Therefore, we do not impose a reversible symmetry assumption on the solutions $\pih_{\p,\epsilon}$ in assumption \ref{assS3}. $\hfill \blacksquare$
\end{rem}
\begin{rem}[\textit{Embedding in an $\epsilon$-independent compactum}] \label{S3.4} An immediate and important consequence of \ref{assS3}-3 is that there exists a compact set $\K \subset U \times \R^m \times V \times \R^n$, independent of $\epsilon$, such that $\pih_{\p,\epsilon}(x) \in \K$ for all $x \in \R$. Note that we can choose $\K$ in such a way that its projection $\K_U \subset U$ on the $u$-component contains $u_0$. This will be convenient for later purposes. $\hfill \blacksquare$
\end{rem}
\begin{rem}[\textit{Exponential decay}] \label{homoclin} By assumption \ref{assS2} $(0,0)$ is a hyperbolic saddle in system (\ref{fastintr}). Hence, the homoclinic $(v_\ho,q_\ho)$ is exponentially localized, i.e. there exists constants $K,\mu > 0$ such that $\|(v_\ho(x),q_\ho(x))\| \leq Ke^{-\mu|x|}$ for $x \in \R$. $\hfill \blacksquare$
\end{rem}
\begin{rem}[\textit{Extending estimate (\ref{fastestS3})}] \label{introbetaf} Consider the solution $\pih_\ho(x)$ to system (\ref{relp}) given by
\begin{align*} \pih_{\ho}(x) = \left(u_0,p_0 + \int_0^x H_2(u_0,v_{\ho}(z))dz,v_\ho(x),q_\ho(x)\right).\end{align*}
System (\ref{ODE}) can be written as $\partial_x \psi = f(\psi,\epsilon)$, where $f$ is locally Lipschitz by \ref{assS1}. Denote by $L_f$ the Lipschitz constant of $f$ on $\K \times [0,\epsilon_0]$ for some $0 < \epsilon_0 \in I$. We infer via Gr\"onwall's inequality and assumption \ref{assS3}-2 that there exists a constant $K > 0$, independent of $\epsilon$, such that
\begin{align} \|\pih_{\p,\epsilon}(x) - \pih_\ho(x)\| \leq \epsilon(K + L_f|x|)e^{L_f|x|}, \label{fastestS31} \end{align}
for $x \in \R$. First, we employ (\ref{fastestS31}) on $J_\epsilon := [x_1\log(\epsilon),-x_1\log(\epsilon)]$ with $x_1 > 0$ $\epsilon$-independent. Second, we use Remark \ref{homoclin} and \ref{assS3}-3 to derive estimates of $d(\pih_{\p,\epsilon}(x),\M)$ and $d(\pih_{\ho}(x),\M)$ on $[-L_\epsilon,L_\epsilon] \setminus J_\epsilon$. Putting these items together yields $\beta_f, \rho > 0$ such that
\begin{align*} \|\pih_{\p,\epsilon}(x) - \pih_\ho(x)\| = \ord(\epsilon^{\beta_f}),\end{align*}
for $x \in I_\epsilon := [-\epsilon^{-\rho},\epsilon^{-\rho}]$. We will refer to $I_\epsilon$ as the \emph{pulse region}. $\hfill \blacksquare$
\end{rem}
\begin{rem}[\textit{The condition that $G$ vanishes at $v = 0$}] \label{condG1} As mentioned in Section \ref{S2.1}, our model (\ref{diff2}) is a general reaction-diffusion system that allows for semi-strong interaction \eqref{reac2}, with the extra condition that the term $G$ vanishes at $v = 0$. For general $G$, consider a $2n$-dimensional compact submanifold $\M_0$ of $\{(u,p,v,0) \colon G(u,v,0) = 0\} \subset U \times \R^m \times V \times \R^n$. By Fenichel's theory $\M_0$ perturbs, for $\epsilon > 0$ sufficiently small, to a locally invariant manifold $\M_\epsilon$ to (\ref{ODE}). This manifold $\M_\epsilon$ is diffeomorphic to $\M_0$ and lies at Hausdorff distance $\ord(\epsilon)$ from $\M_0$. When $\M_0$ can be given as a graph over $(u,p) \in U \times \R^m$, the same holds for $\M_\epsilon$. Thus, in that case one can change coordinates in (\ref{ODE}) relative to $\M_\epsilon$ and we obtain $\M_\epsilon = \M_0 \subset \M$. Therefore, on the existence level, the condition that $G$ vanishes at $v = 0$, corresponds to an a priori coordinate change in (\ref{ODE}).\\
\\
However, on the stability level, one introduces more than additional technical difficulties when $G$ does not vanish at $v = 0$. Indeed, without relative coordinates, we do not achieve estimate (\ref{slowestS31}), which is essential in our spectral stability analysis. However, applying the coordinate change to equation \eqref{diff2} changes its structure fundamentally. In the new coordinates \eqref{diff2} is not even of reaction-diffusion type. Hence, we expect that the spectral analysis differs essentially, when $G$ does not vanish at $v = 0$. This is an interesting subject of future research, especially since it includes the possibility of localized patterns with oscillating tails \cite{CAS,VEA}, but is outside the scope of this paper. $\hfill \blacksquare$
\end{rem}
\subsection{Existence of stationary, spatially periodic pulse solutions} \label{secexistence}
To motivate assumption \ref{assS3}, we elaborate on the existence of reversible periodic pulse solutions to system (\ref{ODE}) satisfying assertions \ref{assS3}-1,2,3. For a general class of slow-fast systems the following result is established in \cite{SOT}. Suppose a singular periodic orbit is given, obtained by piecing together orbit segments of the fast and slow reduced systems in such a way that they form a closed loop. Then, under certain conditions, an actual periodic orbit lies close to the singular one, provided $\epsilon > 0$ is sufficiently small.\\
 \\
An essential condition for the result in \cite{SOT} is that the slow components are constant in the fast reduced system.  In our case the slow $p$-component is non-constant along orbits in \eqref{relp}. Therefore, the result in \cite{SOT} is not applicable. However, we will obtain a similar result. Consider a singular periodic orbit (Figure \ref{fig:sub2}) obtained by gluing a pulse solution to (\ref{relp}) to a solution on $\M$ governed by (\ref{slowp}). For small $\epsilon > 0$, we will show that there exists a periodic orbit in (\ref{ODE}) in the vicinity of this singular periodic orbit.\\
 \\
Concerning the existence of the singular periodic orbit, we need some assumptions. We start with the assumption on the fast reduced system (\ref{relp}).
\smallskip
\begin{itemize}
\item[\namedlabel{assE1}{\textbf{(E1)}}] \textit{\textbf{Existence of a pulse solution to the fast reduced system}}\\
There exists $u_0 \in U$ such that (\ref{fastintr}) has a solution $\psi_{\ho}(x,u_0) = (v_\ho(x,u_0),q_\ho(x,u_0))$ homoclinic to $0$. The stable manifold $W_{u_0}^s(0)$ to the hyperbolic saddle $0$ in system (\ref{fastintr}) intersects the space $\ker(I - R_f)$ transversely in the point $\psi_{\ho}(0,u_0)$, where $R_f$ is as in Remark \ref{reversiblerem}.
\end{itemize}
\begin{rem}[\textit{Elementary homoclinic orbits}]
In the terminology of \cite{VAN} homoclinics that cross a transverse intersection of $W_{u_0}^s(0)$ and $\ker(I - R_f)$ are called \emph{elementary}. In particular, any non-degenerate homoclinic solution is elementary by \cite[Lemma 4]{VAN}.$\hfill \blacksquare$
\end{rem}
\begin{rem}[\textit{Extension of assumption \ref{assE1}}] \label{extE1}
Since transverse intersection are robust under perturbations, assumption \ref{assE1} implies the existence of an open neighborhood $U_\ho \subset U$ of $u_0$ such that for every $u \in U_\ho$ there exists a homoclinic solution $\psi_{\ho}(x,u)$ to
\begin{align*}\left\{\begin{array}{rcl} D_2 \partial_x v &=& q\\ \partial_x q &=& G(u,v,0)\end{array}\right., \ \ \ v \in V, q \in \R^n,\end{align*} crossing $W_u^s(0) \pitchfork \ker(I-R_f) = \{\psi_{\ho}(0,u)\}$. $\hfill \blacksquare$
\end{rem}
\begin{rem}[\textit{Take-off and touch-down manifolds}] \label{TOTD}
The homoclinics $\psi_{\ho}(x,u)$ for $u \in U_\ho$ yield solutions,
\begin{align*}\pih_{\ho}(x,u) := \left(u,\int_0^x H_2(u,v_{\ho}(z,u))dz,v_{\ho}(x,u),q_{\ho}(x,u)\right),\end{align*}
to (\ref{relp}), which are homoclinic to $\M$. The limits $\lim_{x \to \pm \infty} \pih_{\ho}(x,u)$ give rise to the so-called \emph{take-off} and \emph{touch-down manifolds}. For that reason, we define the mapping $\J \colon U_\ho \to \R^m$ by
\begin{align*} \J(u) = \int_{-\infty}^0 H_2(u,v_{\ho}(z,u))dz. \end{align*}
The $m$-dimensional graphs $\T_\pm := \{(u,\pm \J(u)) : u \in U_\ho\}$ on $\M$ are the take-off and touch-down manifolds. Indeed, it holds $\lim_{x \to \pm \infty} \pih_{\ho}(x,u) = (u,\pm \J(u),0,0)$. $\hfill \blacksquare$
\end{rem}
The manifolds $\T_\pm$ of Remark \ref{TOTD} allow us to piece the pulse solutions $\pih_\ho$ to solutions that lie in $\M$ in order to obtain a singular periodic orbit (Figure \ref{fig:sub2}). Therefore, we shift our attention to the slow reduced system (\ref{slowp}). Recall that system (\ref{slowp}) is $R_s$-reversible by Remark \ref{reversiblerem}. It holds $R_s[\T_+] = \T_-$. Hence, to establish a connection between the take-off and touch-down manifolds $\T_\pm$, it is sufficient to find a solution that starts in $\ker(I - R_s)$ and crosses the take-off manifold $\T_-$ at some point. This is the content of our next assumption.
\smallskip
\begin{itemize}
\item[\namedlabel{assE2}{\textbf{(E2)}}] \textit{\textbf{Existence of connecting orbit in slow reduced system}}\\
Let $R_s, \T_{\pm}, \J$ be as in Remarks \ref{reversiblerem} and \ref{TOTD}. There exists a solution $\psi_{\s}(\xx) = (u_\s(\xx),p_\s(\xx))$ to system (\ref{slowp}) such that $\psi_{\s}(0) \in \ker(I - R_s)$ and $\psi_{\s}(\check{L}_0) \in \T_-$ for some $\check{L}_0 > 0$. Moreover, let $\Phi_\s(\xx,\yy)$ be the evolution of the associated variational equation,
\begin{align*} \partial_\xx \phi = \left(\begin{array}{cc} 0 & D_1^{-1} \\ \partial_u H_1(u_\s(\xx),0,0) & 0 \end{array}\right)\phi.\end{align*}
Denote for $i,j \in \{1,\ldots,m\}$ by $h_j$ the coordinates of $H_1(u_\s(0),0,0)$ and by $A_{ij}$ the $(m \times m)$-submatrix of
\begin{align*} \Phi_\s(0,\check{L}_0)\left(\begin{array}{c} I \\ -\partial_u \J(u_0)\end{array}\right),\end{align*}
containing rows $\{i,m+1,\ldots,2m\}\setminus\{m+j\}$. There exists $i_* \in \{1,\ldots,m\}$ such that
\begin{align} \sum_{j = 1}^m (-1)^j h_j \det(A_{i_*j}) \neq 0. \label{transvers} \end{align}
\end{itemize}
\smallskip
To obtain persistence under small perturbations, one could require that the intersection between the take-off manifold $\T_-$ and the trajectory $\psi_\s$ is transversal. However, this is impossible for $m > 1$. Therefore, the technical condition \eqref{transvers} is employed to generate a `good' set of initial conditions in $\ker(I-R_s)$. This set becomes under the forward flow of the slow reduced system \eqref{slowp} an $m$-dimensional manifold, which contains the solution $\psi_\s$ and does intersect $\T_-$ transversally (within $\M$).
\begin{rem}[\textit{Assumption \ref{assE2} in the case $m = 1$}] \label{E2m1} In the case $m = 1$ system (\ref{slowp}) is Hamiltonian. Therefore, assumption \ref{assE2} is satisfied if and only if there exists $u_1 \in U$ such that the following four conditions hold true,
\begin{align*} \int_{u_1}^{u_0} H_1(\hat{u},0,0)d\hat{u} = \frac{1}{2}\J(u_0)^2,&  \ \ \  (u_1-u_0)\J(u_0) > 0,\\
 H_1(u_1,0,0) \neq 0,& \ \ \  H_1(u_0,0,0) \neq \J(u_0)\partial_u \J(u_0).  \end{align*}
The first three conditions can be related to the existence of a trajectory from $(u_1,0) \in \ker(I-R_s)$ to $(u_0,-\J(u_0))$. The last condition is associated with the transversality of the intersection between this trajectory and the take-off curve $\T_-$. $\hfill \blacksquare$
\end{rem}
\begin{figure}[t]
\centering
\begin{subfigure}[b]{.54\textwidth}
  \centering
  \includegraphics[width=\linewidth]{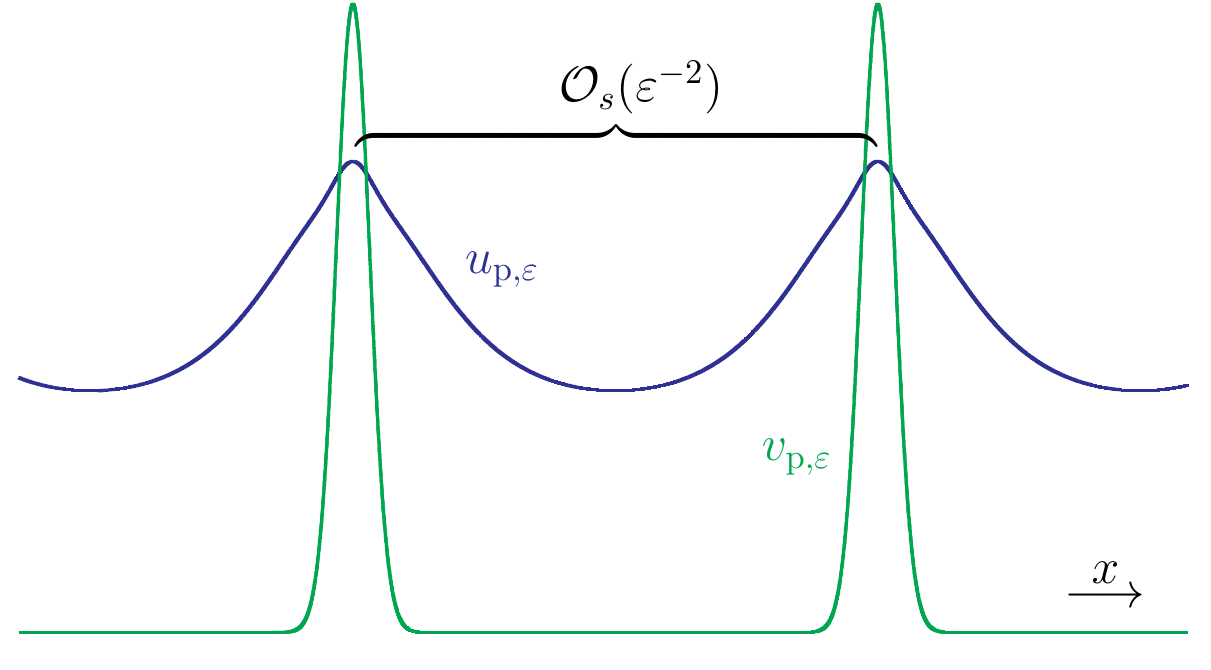}
  \caption{Slow-fast dynamics: the $v$-component exhibits localized pulse behavior and the $u$-component evolves slowly.}
\end{subfigure}%
\hspace{.03\textwidth}
\begin{subfigure}[b]{.34\textwidth}
  \centering
  \includegraphics[width=\linewidth]{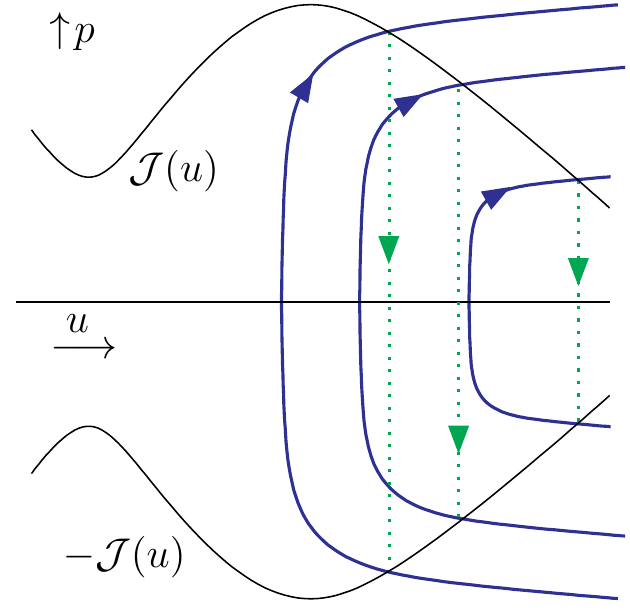}
  \caption{Orthogonal projection of the singular periodic orbit on the slow manifold $\M$.}
\label{fig:sub2}
\end{subfigure}
\caption{Periodic pulse solutions in the case $n = m = 1$.}
\label{fig:test}
\end{figure}
We are now able to state the existence result. The proof uses techniques developed in \cite{PLO2}, where the existence of stationary, spatially periodic pulse solutions in the Gierer-Meinhardt equation is considered. However, we emphasize that the framework in \cite{PLO2} differs from ours due to a difference in scaling in the $p$-component. This is explained in more detail in Remark \ref{scaling}.
\newpage
We extend and adapt the techniques in \cite{PLO2} to the more general existence problem (\ref{ODE}). This existence result, which focuses on periodic solutions in a large class of singularly perturbed systems beyond the type of slow-fast systems considered in \cite{SOT}, is to our knowledge new -- we are not aware of a similar general result in the literature. To avoid too much diversion, we present the proof in Appendix \ref{A0}. The proof has a geometric nature. It relies on Fenichel's theory, reversible symmetry arguments and an appropriate version of the Exchange Lemma.
\begin{theo} \label{maintheorem}
Assume \ref{assS1}-\ref{assS2} and \ref{assE1}-\ref{assE2} hold true. Then, for $\epsilon > 0$ sufficiently small, there exists a $2L_\epsilon$-periodic solution $\pih_{\p,\epsilon}$ to  (\ref{ODE}) satisfying assertions 1-3 in \ref{assS3}. Moreover, $\pih_{\p,\epsilon}$ is symmetric about the space $p = q = 0$. It holds $\pih_{\p,\epsilon}(x) = R\pih_{\p,\epsilon}(-x)$ and $\pih_{\p,\epsilon}(L_\epsilon - x) = R\pih_{\p,\epsilon}(L_\epsilon + x)$ for $x \in \R$, where $R$ is as in Remark \ref{reversiblerem}.
\end{theo}
\begin{rem}[\textit{An expression for $\check{L}_0$ in the case $m = 1$}] \label{L0m1}
Let $m = 1$ and $u_1 \in U$ as in Remark \ref{E2m1}. Using the Hamiltonian nature of system (\ref{slowp}), we can express $\check{L}_0$, as defined in \ref{assS3}-1, in terms of $u_0$ and $u_1$ as
\begin{align*}\check{L}_0 = \int_{u_1}^{u_0} \frac{d\hat{u}}{\sqrt{2 F(\hat{u}) + \J(u_0)^2}},\end{align*}
where $F(u) := \int_{u_0}^u H_1(\hat{u},0,0)d\hat{u}$ for $u \in U$. $\hfill \blacksquare$
\end{rem}
\begin{exa2}[\textit{Generalized Gierer-Meinhardt Equations with slow nonlinearity}]
An instance of (\ref{diff}), which satisfies assumptions \ref{assS1}-\ref{assS2} and \ref{assE1}-\ref{assE2}, is the generalized Gierer-Meinhardt equation with slow nonlinearity,
\begin{align}
\left\{\!
\begin{array}{rcl}
\partial_t u &=& \epsilon \partial_{xx} u - \epsilon^2 f(u) + \epsilon u^{\alpha_1}v^{\beta_1} \\
\partial_t v &=& \partial_{xx} v - v + u^{\alpha_2}v^{\beta_2}\end{array}\right., \ \ \ u \in \R_{> 0}, v \in \R, \label{GM}
\end{align}
with parameters $\alpha_{1,2} \in \R, \beta_{1,2} \in \Z_{> 1}$ and $f \in C^2(0,\infty)$. One can take $f(u) = \mu u$ for some $\mu > 0$ to obtain the slowly linear Gierer-Meinhardt equations considered in \cite{PLO}. In the model considered in \cite{VEE} one takes $f(u) = c_1 u - c_2 u^d$ with $c_{1,2} > 0$ and $d > 1$. Also, the choice $f(u) = \sin(u)$ is possible. Then, one obtains the sinusoidal Gierer-Meinhardt model considered as a guiding example in the companion paper \cite{RIJV}. $\hfill \blacksquare$
\end{exa2}
\section{Main results} \label{sec2}
Throughout the rest of this paper, assume \ref{assS1}-\ref{assS3} and regard $\pih_{\p,\epsilon}(x) = (u_{\p,\epsilon}(x),v_{\p,\epsilon}(x))$, for $\epsilon > 0$ sufficiently small, as a stationary, spatially $2L_\epsilon$-periodic pulse solution to (\ref{diff}) as described in \ref{assS3}. Let $\check{\pih}_{\p,\epsilon}(\xx) = (\check{u}_{\p,\epsilon}(\xx),\check{v}_{\p,\epsilon}(\xx))$ be the corresponding solution to the rescaled system (\ref{diff2}). We linearize system (\ref{diff2}) about $\check{\pih}_{\p,\epsilon}$ and obtain the differential operator $\El_\epsilon \colon C_{ub}^2(\R,\R^{m+n}) \subset C_{ub}(\R,\R^{m+n}) \to C_{ub}(\R,\R^{m+n})$ given by:
\begin{align} \El_\epsilon \left(\begin{array}{c} u \\ v \end{array}\right) = \left(\begin{array}{c} D_1\partial_{\xx\xx} u \\ \epsilon^2 D_2\partial_{\xx\xx} v \end{array}\right) - \B_\epsilon(\xx) \left(\begin{array}{c} u \\ v \end{array}\right), \label{defL} \end{align}
with
\begin{align}\B_\epsilon(\xx) := &\left(\begin{array}{cc} \partial_u H_1 (\check{\pih}_{\p,\epsilon},\epsilon) + \epsilon^{-1} \partial_u H_2 (\check{\pih}_{\p,\epsilon}) & \partial_v H_1(\check{\pih}_{\p,\epsilon},\epsilon) +  \epsilon^{-1} \partial_v H_2 (\check{\pih}_{\p,\epsilon}) \\ \partial_u G (\check{\pih}_{\p,\epsilon},\epsilon) &  \partial_v G (\check{\pih}_{\p,\epsilon},\epsilon)\end{array}\right), \label{defB}
\end{align}
suppressing the $\xx$-dependence of $\check{\pih}_{\p,\epsilon}$ in the right hand side of \eqref{defB}. Here, $C_{ub}^k(\R,\R^{m+n})$ denotes the Banach space of $k$ times continuously differentiable functions, with derivatives up to order $k$ bounded and uniformly continuous. It is endowed with the norm,
\begin{align*} \|f\| = \sum_{i = 0}^k \|(\partial_\xx)^i f\|_\infty.\end{align*}
Note that by assumption \ref{assS3}-1, $\El_\epsilon$ is a periodic differential operator with period $\check{L}_\epsilon$. \\
\\
In this section we collect the main outcomes of our spectral analysis. We formulate a leading order approximation result for the spectrum $\sigma(\El_\epsilon)$. This leads to explicit (spectral) instability criteria. As mentioned in Remark \ref{specintr} the leading order approximation results are not sufficient to decide upon spectral stability. However, we will show that the question of spectral stability can be reduced to a study of the fine  structure of the spectrum around $0$.
\subsection{Formulation of the Evans function}
Before stating the main results, we focus on the structure of the spectrum $\sigma(\El_\epsilon)$ and introduce the analytic Evans function that vanishes precisely on $\sigma(\El_\epsilon)$. \\
\\
Regard $\El_{\C,\epsilon}$ as the extension of $\El_\epsilon$ to $C_{ub}(\R,\C^{m+n})$ given by \begin{align*}\El_{\C,\epsilon} \phi = \El_\epsilon[\mathrm{Re}(\phi)] + i\El_\epsilon[\mathrm{Im}(\phi)].\end{align*} Fundamental in our spectral analysis is that by \cite[Proposition 2.1]{GAR} a point $\lambda \in \C$ is in the spectrum $\sigma(\El_\epsilon)$ if and only if there exists $\phi \in C_{ub}^2(\R,\C^{m+n}) \!\setminus\! \{0\}$ such that $\El_{\C,\epsilon} \phi = \lambda \phi$. In the small spatial scale $x = \epsilon^{-1}\xx$ the eigenvalue problem $\El_{\C,\epsilon}\phi = \lambda \phi$ can be written as the linear system,
\begin{align}
\partial_x \phi = \left(\begin{array}{cc} \sqrt{\epsilon} \A_{11,\epsilon}(x,\lambda) & \sqrt{\epsilon} \A_{12,\epsilon}(x) \\ \A_{21,\epsilon}(x) & \A_{22,\epsilon}(x,\lambda)\end{array}\right)\phi, \ \ \ \phi = (u,p,v,q) \in \C^{2(m+n)}, \label{fullstab}
\end{align}
with periodic coefficient matrices,
\begin{align*}
\A_{11,\epsilon}(x,\lambda) &:=\left(\begin{array}{cc} 0 & D_1^{-1} \\ \epsilon \left(\partial_u H_1 (u_{\p,\epsilon}(x),v_{\p,\epsilon}(x),\epsilon) + \lambda\right) + \partial_u H_2 (u_{\p,\epsilon}(x),v_{\p,\epsilon}(x)) & 0 \end{array}\right),\\
\A_{12,\epsilon}(x) &:= \left(\begin{array}{cc} 0 & 0 \\ \epsilon \partial_v H_1 (u_{\p,\epsilon}(x),v_{\p,\epsilon}(x),\epsilon) + \partial_v H_2 (u_{\p,\epsilon}(x),v_{\p,\epsilon}(x)) & 0 \end{array}\right),
\end{align*}
and
\begin{align*}
\A_{21,\epsilon}(x) := \left(\begin{array}{cc} 0 & 0 \\ \partial_u G(u_{\p,\epsilon}(x),v_{\p,\epsilon}(x),\epsilon) & 0 \end{array}\right), \ \ \A_{22,\epsilon}(x,\lambda) := \left(\begin{array}{cc} 0 & D_2^{-1} \\ \partial_v G(u_{\p,\epsilon}(x),v_{\p,\epsilon}(x),\epsilon) + \lambda & 0 \end{array}\right).
\end{align*}
We will refer to (\ref{fullstab}) as the \emph{linear stability problem}.
\begin{rem}[\textit{Scaling of the $p$-component}] \label{scaling} We emphasize that there is a difference in scaling in the $p$-component between the linear stability problem (\ref{fullstab}) and the existence problem (\ref{ODE}). The scaling regime in (\ref{fullstab}) brings the system in slow-fast form, which has the advantage of making important tools available such as the Riccati transform (Theorem \ref{ric}). In principle, the existence problem could also be put in this form by introducing $p = \epsilon^{-1/2} D_1 \partial_x u$ instead -- as is done in existence analysis of periodic pulse solutions in the Gierer-Meinhardt equations in \cite{PLO2}. The equation for the $p$-component in the slow reduced system (\ref{slowp}) would read $\partial_\xx p = 0$ in that case. This makes the construction of the desired singular periodic orbit, performed in Section \ref{secexistence}, impossible. Therefore, the scaling regime in (\ref{ODE}) is the most natural for our existence analysis. In \cite{PLO2} one avoids setting $\epsilon = 0$ in the existence analysis and makes a distinction between slow and `super-slow' behavior. $\hfill \blacksquare$
\end{rem}
We are looking for solutions $\phi \in C_{ub}(\R,\C^{m+n})$ to (\ref{fullstab}). It is well-known by Floquet Theory \cite[Chapter 1]{BES} that bounded solutions to (\ref{fullstab}) must satisfy $\phi(-L_\epsilon) = \gamma \phi(L_\epsilon)$ for some $\gamma$ in the unit circle $S^1$. This fact gives rise to the following definition, which provides a tool to locate the spectrum of $\El_\epsilon$.
\begin{defi}
Denote by $\T_\epsilon(x,z,\lambda)$ the evolution operator of system (\ref{fullstab}) with initial condition $z$. The \emph{Evans function} $\E_\epsilon \colon \C \times \C \to \C$ is given by
\begin{align*} \E_\epsilon(\lambda,\gamma) := \det(\T_\epsilon(0,-L_\epsilon,\lambda) - \gamma \T_\epsilon(0,L_\epsilon,\lambda)).\end{align*}
\end{defi}
\begin{prop} \label{Evansprop} The Evans function has the following properties.
\begin{enumerate}
 \item[1.] The Evans function is analytic in both $\lambda$ and $\gamma$.
 \item[2.] We have $\lambda \in \sigma(\El_\epsilon)$ if and only if there exists $\gamma \in S^1$ such that the dispersion relation $\E_\epsilon(\lambda,\gamma) = 0$ is satisfied.
 \item[3.] It holds $\overline{\E_\epsilon(\lambda,\gamma)} = \E_\epsilon(\overline{\lambda},\overline{\gamma})$ for $\lambda,\gamma \in \C$. Thus, the spectrum $\sigma(\El_\epsilon)$ is closed under complex conjugation.
\end{enumerate}
\end{prop}
\begin{proof}
The first two properties are the content of \cite[Lemma 2.4]{GAR} and \cite[Proposition 2.1]{GAR}, respectively. Since (\ref{fullstab}) is a real valued problem for $\lambda \in \R$, the third property follows by the reflection principle.
\end{proof}
Proposition \ref{Evansprop} shows that the spectrum of $\El_\epsilon$ is the union of the discrete zero sets $\N(\E_\epsilon(\cdot,\gamma))$, where $\gamma$ ranges over $S^1$. Hence the spectrum $\sigma(\El_\epsilon)$ is an at most countable union of curves, each of which is parameterized over $S^1$. This gives rise to the definition of $\gamma$-eigenvalues, which is due to Gardner \cite{GAR}.
\begin{defi}
Let $\gamma \in S^1$. A zero $\lambda$ of the Evans function $\E_\epsilon(\cdot,\gamma)$ is called a \emph{$\gamma$-eigenvalue} of $\El_\epsilon$.
\end{defi}
\begin{rem}[\textit{Reversible symmetry}] \label{revsym}
Recall that system (\ref{ODE}) is $R$-reversible by Remark \ref{reversiblerem}. Suppose $\pih_{\p,\epsilon}$ obeys this reversible symmetry, which is the case for the solutions obtained in existence result Theorem \ref{maintheorem}. Then, the linear stability problem (\ref{fullstab}) is $R$-reversible at $0$, i.e. it holds $R\T_\epsilon(x,y,\lambda)R = \T_\epsilon(-x,-y,\lambda)$ for $x,y \in \R$. One readily deduces $\E_\epsilon(\lambda,\gamma) = \E_\epsilon(\lambda,\gamma^{-1})\gamma^{2(m+n)}$ for $\lambda \in \C$ and $\gamma \in \C \!\setminus\! \{0\}$. So, if $\lambda \in \C$ is a $\gamma$-eigenvalue for some $\gamma \in S^1$, then $\lambda$ is also a $\overline{\gamma}$-eigenvalue, since we have $\overline{\gamma} = \gamma^{-1}$ for $\gamma \in S^1$. Therefore, the image of $S^1$ covers each curve of spectrum twice, which yields degenerate spectrum (Figure \ref{fig2:test}). $\hfill \blacksquare$
\end{rem}
\subsection{Approximation of the spectrum} \label{mainresult3}
The main outcome of our spectral analysis is an explicit \emph{reduced Evans function} $\E_0(\lambda,\gamma)$, whose zeros, for $\gamma$ restricted to $S^1$, approximate the zeros of the Evans function $\E_\epsilon(\lambda,\gamma)$, provided that $\epsilon > 0$ is sufficiently small. Therefore, the reduced Evans function $\E_0$ proves to be a powerful tool to study the shape of the spectrum $\sigma(\El_\epsilon)$ and decide upon spectral (in)stability questions. We emphasize that the approximation of $\E_\epsilon$ by $\E_0$ is only valid on certain half planes $\Ce_\Lambda$, which contain the critical spectrum of $\El_\epsilon$.
\begin{nota} \label{notCEL}
For every $\Lambda < 0$ we denote by $\Ce_\Lambda$ the open half plane,
\begin{align*} \Ce_\Lambda := \{\lambda \in \C : \mathrm{Re}(\lambda) > \Lambda\}.\end{align*}
\end{nota}
The slow-fast structure of the linear stability problem (\ref{fullstab}) is reflected by the fact that $\E_0$ can be written as a product
\begin{align}\E_0(\lambda,\gamma) = (-\gamma)^n\E_{f,0}(\lambda)\E_{s,0}(\lambda,\gamma). \label{e:evprod}\end{align}
Here, the analytic map $\E_{f,0} \colon \Ce_\Lambda \to \C$ is called the \emph{fast reduced Evans functions}. It is associated with the \emph{homogeneous fast limit problem},
\begin{align}
\partial_x \phi = \A_{22,0}(x,\lambda)\phi, \ \ \ \phi \in \C^{2n}, \ \ \ \ \ \ \A_{22,0}(x,\lambda) := \left(\begin{array}{cc} 0 & D_2^{-1} \\ \partial_v G (u_0,v_{\ho}(x),0) + \lambda & 0 \end{array}\right). \label{redstab}
\end{align}
System \eqref{redstab} corresponds to the homogeneous variant of the fast layer problem in the singular limit of \eqref{fullstab}. We establish the existence of the fast reduced Evans function.
\begin{prop} \label{fastredEv}
There exists $\Lambda < 0$ and an analytic map $\E_{f,0} \colon \Ce_\Lambda \to \C$, which has a zero if and only if problem (\ref{redstab}) admits a non-trivial bounded solution. In particular, $\lambda \in \Ce_{\Lambda}$ is a simple zero of $\E_{f,0}$ if and only if there exists a unique non-trivial bounded solution to (\ref{redstab}) up to scalar multiples.
\end{prop}
The slow reduced Evans function $\E_{s,0} \colon [\Ce_\Lambda \setminus \E_{f,0}^{-1}(0)] \times \C \to \C$ is determined by two problems. The first is the \emph{inhomogeneous fast limit problem},
\begin{align}
\partial_x \X = \A_{22,0}(x,\lambda)\X + \A_{21,0}(x), \ \ \X \in \mathrm{Mat}_{2n \times 2m}(\C), \ \ \A_{21,0}(x) := \left(\begin{array}{cc} 0 & 0 \\ \partial_u G (u_0,v_{\ho}(x),0) & 0 \end{array}\right). \label{fastinhom}
\end{align}
The matrix system (\ref{fastinhom}) can be seen as a family of inhomogeneous fast layer problems in the singular limit of (\ref{fullstab}). The second is the \emph{slow limit problem},
\begin{align}
\left\{\! \begin{array}{rcl} D_1\partial_\xx u &=& p \\ \partial_\xx p &=& \left(\partial_u H_1 (u_\s(\xx),0,0) + \lambda\right)u\end{array}\right., \ \ \ u,p \in \C^{m}, \label{slowintr}
\end{align}
which correspond to the slow subsystem of \eqref{fullstab}.
\begin{defi} \label{slowEv}
The \emph{slow reduced Evans function} $\E_{s,0} \colon [\Ce_\Lambda \setminus \E_{f,0}^{-1}(0)] \times \C \to \C$ is defined by
\begin{align*} \E_{s,0}(\lambda,\gamma) = \det\left(\Upsilon(\lambda)\T_{s}(2\check{L}_0,0,\lambda) - \gamma I\right),  \end{align*}
where $\check{L}_0$ is given by assumption \ref{assS3}-1, $\T_s(\xx,\yy,\lambda)$ is the evolution of the slow limit problem (\ref{slowintr}) and $\Upsilon(\lambda)$ is given by
\begin{align} \begin{split}\Upsilon(\lambda) &=  \left(\begin{array}{cc} I & 0 \\  \G(\lambda) & I\end{array}\right),\\
\G(\lambda) &= \int_{-\infty}^\infty [\partial_u H_2 (u_0,v_{\ho}(x)) + \partial_v H_2(u_0,v_{\ho}(x))\V_{in}(x,\lambda)]dx,\end{split} \label{defupsilon}\end{align}
where $\V_{in}(x,\lambda)$ denotes the upper-left $(n \times m)$-block of the unique $(2n \times 2m)$-matrix solution $\X_{in}(x,\lambda)$ to the inhomogeneous fast limit problem (\ref{fastinhom}).
\end{defi}
We collect some properties of the slow reduced Evans functions $\E_{s,0}$.
\begin{prop} \label{slowredEv}
The slow reduced Evans function $\E_{s,0}$ is analytic in both $\lambda$ and $\gamma$. Moreover, $\E_{s,0}(\cdot,\gamma)$ is meromorphic on $\Ce_\Lambda$ in such a way that the reduced Evans function $\E_0(\cdot,\gamma)$ is analytic on $\Ce_\Lambda$ for each $\gamma \in \C$.
\end{prop}
Concerning zeros of $\E_0$, consider a simply closed curve $\Gamma\subset\Ce_\Lambda$ that avoids the roots of $\E_{s,0}$ and $\E_{f,0}$
\begin{align*}
\N_0 := \left[\N(\E_{f,0}) \cup \bigcup_{\gamma \in S^1} \N(\E_{s,0}(\cdot,\gamma))\right].
\end{align*}
Let $\gamma \in S^1$. Denote by $N_\Gamma$, $N_{\Gamma,\gamma}$ and $P_{\Gamma,\gamma}$ the number of zeros of $\E_{f,0}$, $\E_{s,0}(\cdot,\gamma)$ and poles of $\E_{s,0}(\cdot,\gamma)$ interior to $\Gamma$ counted with multiplicity, respectively. Given the analyticity and product form \eqref{e:evprod} of the reduced Evans function $\E_0(\cdot,\gamma)$, its number of roots interior to $\Gamma$ counted with multiplicity is given by
\begin{align}
N_\Gamma + N_{\Gamma,\gamma} - P_{\Gamma,\gamma}.
 \label{Eulerlike}\end{align}
Our first approximation result confirms this and its persistence for $\E_\epsilon$ with small $\epsilon>0$.
\begin{theo} \label{mainresult}
Take a simple closed curve $\Gamma$ contained in $\Ce_\Lambda \!\setminus\! \N_0$. For $\epsilon > 0$ sufficiently small, the number of $\gamma$-eigenvalues (including multiplicity) of $\El_\epsilon$ interior to $\Gamma$ equals \eqref{Eulerlike}.
\end{theo}
\begin{figure}[t]
\centering
\begin{subfigure}[b]{.45\textwidth}
  \centering
  \includegraphics[width=\linewidth]{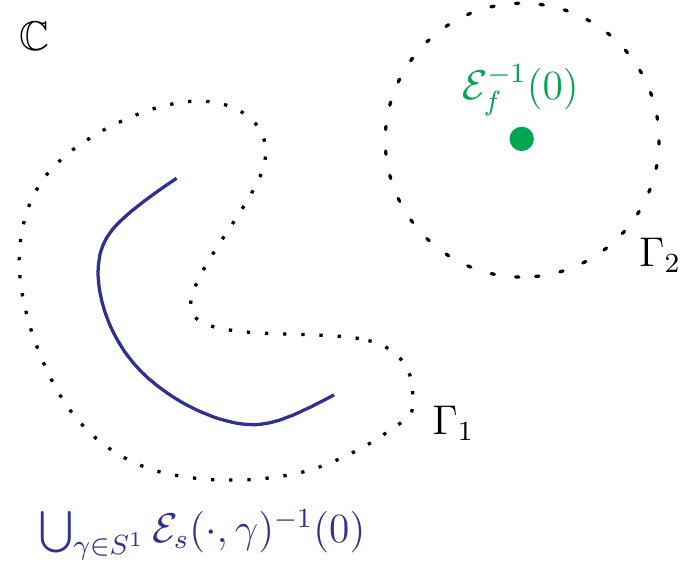}
  \caption{Spectrum in the singular limit with disjoint contours $\Gamma_1$ and $\Gamma_2$.}
\end{subfigure}%
\hspace{.05\textwidth}
\begin{subfigure}[b]{.45\textwidth}
  \centering
  \includegraphics[width=\linewidth]{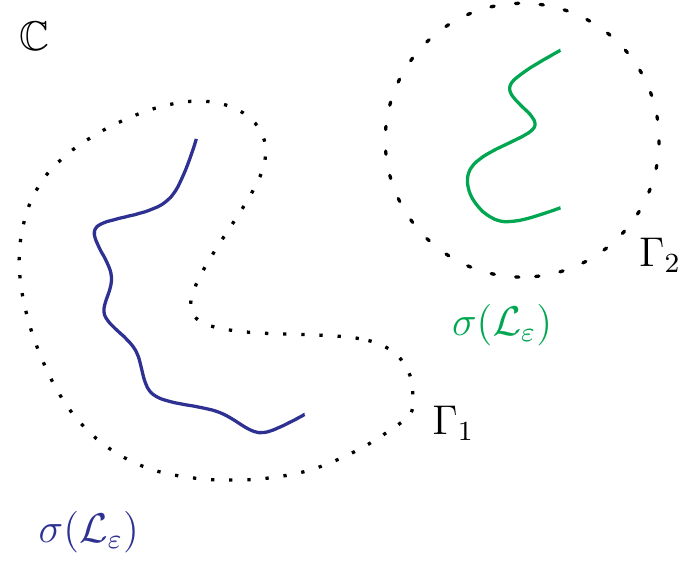}
  \caption{The true spectrum remains in the interior of $\Gamma_1$ and $\Gamma_2$ for $\epsilon > 0$ sufficiently small.}
\end{subfigure}
\caption{approximation of the spectrum $\sigma(\El_\epsilon)$ in the case of reversible symmetry (Remark \ref{revsym}).}
\label{fig2:test}
\end{figure}
Theorem \ref{mainresult} shows that the set,
\begin{align*} \sigma_0 := \bigcup_{\gamma \in S^1} \N(\E_0(\cdot,\gamma)),\end{align*}
can be seen as the spectrum in the singular limit of the operator $\El_\epsilon$ in $\Ce_\Lambda$. Indeed, choose contours close enough to and disjoint from the connected components of $\sigma_0$, with, say, Hausdorff distance $\delta$. This results in an $\epsilon_\delta > 0$ such that, if $\epsilon \in (0,\epsilon_\delta)$, then the spectrum $\sigma(\El_\epsilon) \cap \Ce_\Lambda$ is contained in a $\delta$-neighborhood of $\sigma_0$. In that sense the reduced Evans function $\E_0$ determines the shape of the spectrum $\sigma(\El_\epsilon)$ up to connected components of the set $\sigma_0$. However, this approximation could still be too rough, since it does not concern pointwise distances of spectra. For instance, when a connected component of $\sigma_0$ transversally intersects the imaginary axis, one can not decide with Theorem \ref{mainresult} upon spectral (in)stability. In that sense, our second approximation result is more refined.
\begin{theo} \label{fixgamma}
Fix $\gamma \in S^1$. There exists $\delta_\gamma > 0$, such that for every $0 < \delta < \delta_\gamma$, provided $\epsilon > 0$ is sufficiently small (with bound dependent on $\delta$), the $\gamma$-eigenvalues of $\El_\epsilon$ in $\Ce_\Lambda$ are contained in a $\delta$-neighborhood of $\N(\E_0(\cdot,\gamma))$. Moreover, for every $\lambda_\gamma \in \N(\E_0(\cdot,\gamma))$ the number of $\gamma$-eigenvalues (including multiplicity) inside $B(\lambda_\gamma,\delta)$ is equal to the multiplicity $\mathrm{ord}_{\lambda_\gamma}(\E_0(\cdot,\gamma))$ of $\lambda_\gamma$ as a root of $\E_0(\cdot,\gamma)$.
\end{theo}
\begin{cor} \label{fixgammacor}
The solution $\pih_{\p,\epsilon}$ is spectrally unstable for $\epsilon > 0$ sufficiently small, if one of the following is true:
\begin{enumerate}
 \item There exists $\gamma_\diamond \in S^1$ and $\lambda_\diamond \in \C$ with $\mathrm{Re}(\lambda_\diamond) > 0$ satisfying $\E_{s,0}(\lambda_\diamond,\gamma_\diamond) = 0$;
 \item There exists $\gamma_\diamond \in S^1$ and a zero $\lambda_\diamond$ of $\E_{f,0}$ with $\mathrm{Re}(\lambda_\diamond) > 0$ such that $\mathrm{ord}_{\lambda_\diamond}(\E_{s,0}(\cdot,\gamma_\diamond)) + \mathrm{ord}_{\lambda_\diamond}(\E_{f,0}) > 0$.
\end{enumerate}
\end{cor}
When a connected component of $\sigma_0$ transversally intersects the imaginary axis, Corollary \ref{fixgammacor} shows that the periodic pulse solution $\pih_{\p,\epsilon}$ is, for $\epsilon > 0$ sufficiently small, spectrally unstable. Indeed, there must be in that case a point, which has positive real part, in one of the discrete sets $\N(\E_{s,0}(\cdot,\gamma))$ for some $\gamma \in S^1$. The fact that this point can be isolated, is the strength of Theorem \ref{fixgamma}. The trade-off is that the bound on the small parameter $\epsilon$ is $\gamma$-dependent in Theorem \ref{fixgamma}, whereas the bound on $\epsilon$ is uniform in $\gamma$ in Theorem \ref{mainresult}.\\
 \\
The proofs of our main approximation results Theorems \ref{mainresult} and \ref{fixgamma} can be found in Section \ref{sec6}.
\begin{rem}[\textit{Weak coupling}] \label{weakint}
As mentioned in section \ref{S2.1}, weak coupling $H_2(u,v) \equiv 0$ is allowed in our spectral analysis. In that case the integral term $\G(\lambda)$ in \eqref{defupsilon} is identically 0, which implies that $\E_{s,0}$ is only determined by the slow limit problem \eqref{slowintr}. Therefore, $\E_{s,0}$ is analytic on $\Ce_{\Lambda}$ and zeros of $\E_{f,0}$ can not be canceled by poles of $\E_{s,0}$. We conclude that the spectral stability problem fully splits into slow and fast subproblems with no interaction between them. As a consequence, zeros of $\E_{f,0}$ of positive real part yield spectral instability immediately by Proposition \ref{fixgammacor}-2. In particular, in the case $n = 1$ the fast reduced Evans function $\E_{f,0}$ always has a zero in the right half plane, as we will show in Proposition \ref{fastm1}. Thus, all patterns $\pih_{\p,\epsilon}$ are spectrally unstable in the case $H_2(u,v) \equiv 0$ and $n = 1$. This motivates the scaling in \eqref{reac2}. $\hfill \blacksquare$
\end{rem}
\begin{rem}[\textit{Translation invariance and spectral stability}] \label{translationinvariance}
If there is a zero $\lambda_\gamma$ of $\E_0(\cdot,\gamma)$ satisfying $\mathrm{Re}(\lambda_\gamma) > 0$ for some $\gamma \in S^1$, we know that the periodic pulse solution $\pih_{\p,\epsilon}$ is spectrally unstable by Corollary \ref{fixgammacor}. However, to determine whether $\pih_{\p,\epsilon}$ is spectrally stable, it is not sufficient to know the zeros of $\E_0$. This is due to translational invariance, i.e. the derivative of $\pih_{\p,\epsilon}$ is always an eigenfunction of $\El_\epsilon$. Hence, $0$ will always be a zero of $\E_\epsilon(\cdot,1)$. So, to decide whether $\pih_{\p,\epsilon}$ is spectrally stable, a study of the so-called \emph{small spectrum} around $0$ together with the fact that for all $\gamma \in S^1$ there are no zeros $\lambda_\gamma \neq 0$ of $\E_0(\cdot,\gamma)$ satisfying $\mathrm{Re}(\lambda_\gamma) \geq 0$ affirms us that $\pih_{\p,\epsilon}$ is spectrally stable. We refer to Corollary \ref{stability} for a more rigorous statement. $\hfill \blacksquare$
\end{rem}
\begin{rem}[\textit{Bloch-wave transform}]
It is also possible to consider the spectrum of the operator $\El_\epsilon$ in the Sobolev space of $\gamma$-twisted periodic functions for fixed $\gamma \in S^1$. The spectrum is in that case given by the discrete set $\N(\E_\epsilon(\cdot,\gamma))$ of $\gamma$-eigenvalues of $\El_\epsilon$. The decomposition of the spectrum of $\El_\epsilon$ in parts $\N(\E_\epsilon(\cdot,\gamma))$, where $\gamma$ ranges over $S^1$, corresponds to the Bloch-wave transformation of the operator $\El_\epsilon$. $\hfill \blacksquare$
\end{rem}
\subsection{Zero-pole cancelation} \label{zeropolesec}
Take $\gamma \in S^1$. By Proposition \ref{slowredEv} a pole of the slow reduced Evans function $\E_{s,0}(\cdot,\gamma)$ can only occur at a zero of the fast reduced Evans function $\E_{f,0}$, because the reduced Evans function $\E_0(\lambda,\gamma) = (-\gamma)^n\E_{f,0}(\lambda)\E_{s,0}(\lambda,\gamma)$ is analytic. In particular, by formula (\ref{Eulerlike}) in Theorem \ref{mainresult} it can occur that $\E_{f,0}$ has a zero $\lambda \in \Ce_\Lambda$ with positive real part, but the periodic pulse solution is spectrally stable. In this case the zero $\lambda$ of $\E_{f,0}$ cancels the pole of $\E_{s,0}$ at $\lambda$.\\
\\
So, besides the position of the zeros of $\E_{f,0}$ and $\E_{s,0}$ with respect to the imaginary axis, one should determine whether zero-pole cancelation occurs at the points in $\N(\E_{f,0})$ with positive real part in order to decide upon spectral (in)stability questions. The next proposition focuses on this issue.
\begin{prop} \label{zeropole} Let $\lambda_\diamond$ be a simple zero of $\E_{f,0}$. Suppose there exists $\delta_0 > 0$ such that
\begin{align*}B(\lambda_\diamond,\delta_0) \cap \bigcup_{\gamma \in S^1} \N(\E_{s,0}(\cdot,\gamma)) = \emptyset.\end{align*}
Then, there exists $\delta_1 > 0$ such that for every $0 < \delta < \delta_1$, provided $\epsilon > 0$ is sufficiently small (with bound dependent on $\delta$), either (i.) $B(\lambda_\diamond,\delta)$ contains precisely one $\gamma$-eigenvalue for every $\gamma \in S^1$ or (ii.) $B(\lambda_\diamond,\delta)$ contains no spectrum of $\El_\epsilon$. The prior (i.) is the case if
\begin{align} \int_{-\infty}^\infty \partial_v H_2(u_0,v_{\ho}(z))\phi_{\lambda_\diamond,1}(z)dz = 0, \label{idpoles1}\end{align}
where $\phi_{\lambda_\diamond}(x) = (\phi_{\lambda_\diamond,1}(x),\phi_{\lambda_\diamond,2}(x))$ is a non-trivial bounded solution to (\ref{redstab}) at $\lambda = \lambda_\diamond$, or
\begin{align} \int_{-\infty}^\infty \psi_{\lambda_\diamond,2}(z)\partial_u G(u_0,v_{\ho}(z),0)dz = 0, \label{idpoles2}\end{align}
where $\psi_{\lambda_\diamond}(x) = (\psi_{\lambda_\diamond,1}(x),\psi_{\lambda_\diamond,2}(x))$ denotes a non-trivial bounded solution to the adjoint equation of (\ref{redstab}) at $\lambda = \lambda_\diamond$.
\end{prop}
The orthogonality relations (\ref{idpoles1}) and (\ref{idpoles2}) imply that there is no zero-pole cancelation. However, the converse is not true as pointed out in Section \ref{failure}. One can show that the integrals in the right hand sides of (\ref{idpoles1}) and (\ref{idpoles2}) appear as one of multiple factors in the principle part of the Laurent expansion of $\E_{s,0}(\cdot,\gamma)$. Although it is possible to write down the singular part of the Laurent series of $\E_{s,0}(\cdot,\gamma)$ explicitly at a zero $\lambda \in \N(\E_{f,0})$ (of higher multiplicity), we decide to postpone this to Section \ref{slows} for the benefit of exposition, since the involved expressions are rather complex (except in the case $m = 1$ -- see Proposition \ref{singm1}). Eventually, these singular parts provide a tool to determine precisely whether zero-pole cancelation occurs or not. Therefore, Proposition \ref{zeropole} is weaker -- but better digestible -- than the statements in Section \ref{slows}.
\subsection{Small spectrum} \label{secsmallspectrum}
The derivative of the homoclinic $(v_{\ho}(x),q_{\ho}(x))$ in (\ref{fastintr}) is a non-trivial bounded solution of (\ref{redstab}) for $\lambda = 0$ due to translation invariance. So, by Proposition \ref{fastredEv} it holds $\E_{f,0}(0) = 0$. Since, we have
\begin{align*} \int_{-\infty}^\infty \partial_v H_2(u_0,v_{\ho}(z)) v_{\ho}'(z)dz = 0, \end{align*}
there occurs no zero-pole cancelation at $\lambda = 0$ by Proposition \ref{zeropole}. This has yet been predicted in Remark \ref{translationinvariance} by the fact $\E_\epsilon(0,1) = 0$. Hence, a separate study of the spectrum close to $0$ -- the so-called \emph{small spectrum} -- is necessary in order to decide upon spectral (in)stability. The following Corollary of Theorems \ref{mainresult} and \ref{fixgamma} makes this statement rigorous.
\begin{cor}[Spectral stability is determined by the fine structure of the small spectrum] \label{stability}
Suppose the following conditions are met:
\begin{enumerate}
 \item There exists $\delta_0 > 0$ such that
 \begin{align*} \bigcup_{\gamma \in S^1} \N(\E_{s,0}(\cdot,\gamma)) \subset \{\lambda \in \C : \mathrm{Re}(\lambda) < -\delta_0\};\end{align*}
 \item For all zeros $\lambda_\diamond \in \N(\E_{f,0}) \!\setminus\! \{0\}$ of $\E_{f,0}$ with $\mathrm{Re}(\lambda_\diamond) \geq 0$, $\E_{s,0}(\cdot,\gamma)$ has a pole of order $\mathrm{ord}_{\lambda_\diamond}(\E_{f,0})$ at $\lambda_\diamond$ for every $\gamma \in S^1$.
\end{enumerate}
Then, there exists $\delta_1 > 0$, such that for all $0 < \delta < \delta_1$, provided $\epsilon > 0$ is sufficiently small (with bound dependent on $\delta$), it holds
\begin{align*} \sigma(\El_\epsilon) \cap \{\lambda \in \C : \mathrm{Re}(\lambda) \geq 0\} \subset B(0,\delta).\end{align*}
Moreover, there are for every $\gamma \in S^1$ at most $\mathrm{ord}_0(\E_{f,0})$ $\gamma$-eigenvalues inside $B(0,\delta)$.
\end{cor}
\begin{rem}[\textit{Determining the fine structure of the small spectrum}] \label{finestructure}
In \cite{SAS} the fine structure of the small spectrum is unraveled for nearly homoclinic wave trains. An expansion of the circle of critical $\gamma$-eigenvalues around $0$ is given in terms of their period. In our case only the $(v,q)$-components of the solution $\pih_{\p,\epsilon}$ are close to the homoclinic $(v_\ho(x),q_\ho(x))$ by assumption \ref{assS3} and Remark \ref{introbetaf}. Therefore, the results in \cite{SAS} are not directly applicable to our situation. Furthermore, in the Gierer-Meinhardt regime the structure of the small spectrum is known \cite[Section 5]{PLO}. However, the construction in \cite{PLO} depends essentially on the fact that the slow limit problem is linear and the model has two components. Hence, a direct extension to our (slowly nonlinear) situation is not immediately clear. This is subject of future research (Section \ref{future}). $\hfill \blacksquare$
\end{rem}
\begin{rem}[\textit{Nonlinear stability}] \label{nonlinearstab}
Since the spectrum $\sigma(\El_\epsilon)$ always touches the origin due to translational invariance, nonlinear stability of $\pih_{\p,\epsilon}$ is a delicate issue. On the one hand, one can `factor out' the translational invariance by looking at orbital (in)stability \cite{MEY}. Then one measures the distance from the perturbation to the family of all translates of the solution rather than to the solution itself. On the other hand, one could look at diffusive nonlinear stability \cite{SAN3,SCH}. In that case the perturbations are spatially localized, i.e. they are in certain Sobolev spaces, where the norms are adjusted by an appropriate weight $\rho(x) = (1-x^2)^d$ for some $d \geq 1/2$. This allows for purely diffusive behavior of the perturbed solution.\\
\\
The presence of spectrum $\sigma(\El_\epsilon)$ in the positive half-plane $\{\lambda \in \C : \mathrm{Re}(\lambda) > 0\}$ implies orbital instability by \cite[Theorem 4.3]{MEY}. Therefore, the spectral instability criteria in Corollary \ref{fixgammacor} imply nonlinear orbital instability immediately. On the other hand, if one assumes, in addition to spectral stability, that $0$ is a simple zero of $\E_{f,0}$ and that the spectrum $\sigma(\El_\epsilon)$ touches the origin in a quadratic tangency, then we achieve nonlinear diffusive stability by \cite[Theorem 1]{SAN3} or \cite[Theorem 1.1]{SCH}. Therefore, also nonlinear stability is eventually determined by the fine structure of the small spectrum. $\hfill \blacksquare$
\end{rem}
\section{The cases \texorpdfstring{$n=1$}{n=1} and \texorpdfstring{$m=1$}{m=1}} \label{nm1}
Most specific models of the form \eqref{reac} in literature have $n=1$ or $m=1$ so that the limit problems (\ref{redstab}), (\ref{fastinhom}) or (\ref{slowintr}) become 2-dimensional. In either case a more refined analysis of the fast and slow reduced Evans functions $\E_{f,0}$ and $\E_{s,0}$ is possible and allows for further insight into zero-pole cancelation and instability criteria. In this section we assume without loss of generality $D_1 = 1$ in the case $m = 1$ and $D_2 = 1$ in the case $n = 1$ -- see Remark \ref{quantD12}. We start with an analysis of the fast reduced Evans function for $n = 1$.
\begin{prop}[The case $n = 1$] \label{fastm1}
Suppose $n = 1$. All zeros of the fast reduced Evans function $\E_{f,0} \colon \Ce_\Lambda \to \C$ are real and simple. There is precisely one positive zero $\lambda_0$ of $\E_{f,0}$. The $v$-component of any non-trivial, bounded solution to (\ref{redstab}) at $\lambda = \lambda_0$ has no zeros. Moreover, $0$ is a zero of $\E_{f,0}$ and $(v_{\ho}'(x),q_\ho'(x))$ spans the space of bounded solutions to (\ref{redstab}) at $\lambda = 0$.
\end{prop}
\begin{proof}
We refer to Theorem \ref{Sturm} for the fact that every non-trivial, bounded solution to (\ref{redstab}) for $\lambda \in \Ce_\Lambda$ is a solution to the Sturm-Liouville eigenvalue problem,
\begin{align} \partial_{xx} v = \left(\frac{\partial G}{\partial v}(u_0,v_{\ho}(x),0) + \lambda\right)v, \ \ \ v \in H^2(\R,\C) \label{SLP}.\end{align}
By \cite[Theorem 2.3.3]{KAP} the set of eigenvalues $\lambda_N < \ldots < \lambda_0$ of (\ref{SLP}) is finite, real and strictly decreasing. Moreover, the eigenfunction corresponding to $\lambda_i, i = 0,\ldots,N$ is unique up to scalar multiples and has precisely $i$ zeros. Hence, all zeros of $\E_{f,0}$ are real and simple by Proposition \ref{fastredEv}. Furthermore, the derivative $v_{\ho}'(x)$ is a solution for $\lambda = 0$ of (\ref{SLP}). The function $v_{\ho}'(x)$ has precisely one zero. So, we derive $\lambda_1 = 0$ and $\lambda_0 > 0$.
\end{proof}
We shift our focus to the case $m = 1$. Then, the slow reduced Evans function $\E_{s,0}(\lambda,\gamma)$ is a quadratic polynomial in $\gamma$, which gives rise to the following result.
\begin{prop}[The case $m = 1$] \label{slowm1} Suppose $m = 1$. The slow reduced Evans function $\E_{s,0} \colon [\Ce_\Lambda \!\setminus\! \N(\E_{f,0})] \times \C \to \C$ is given by
\begin{align*} \E_{s,0}(\lambda,\gamma) = \gamma^2 - \mathrm{Tr}(\Upsilon(\lambda)\T_s(2\check{L}_0,0,\lambda)) \gamma + 1. \end{align*}
In particular, $\lambda \in \Ce_\Lambda \!\setminus\! \N(\E_{f,0})$ is a zero of $\E_{s,0}(\cdot,\gamma)$ for some $\gamma \in S^1$ if and only if
\begin{align} t(\lambda) := \mathrm{Tr}(\Upsilon(\lambda)\T_s(2\check{L}_0,0,\lambda)) \in [-2,2]. \label{traces} \end{align}
\end{prop}
\begin{proof}
This follows readily by expanding $\E_{s,0}(\lambda,\gamma)$ as quadratic polynomial in $\gamma$.
\end{proof}
\begin{exa2}[\textit{Generalized Gierer-Meinhardt equation}]
In \cite{PLO} the spectral stability of spatially periodic pulse patterns is studied, where (\ref{diff}) is the slowly linear generalized Gierer-Meinhardt equation (\ref{GM}) with $f(u) = \mu u$ for some $\mu > 0$. Thus, system (\ref{slowintr}) is autonomous. Condition (\ref{traces}) simplifies in that case to
\begin{align*}
2\cosh(2\check{L}_0\sqrt{\mu + \lambda}) + \frac{\G(\lambda)\sinh(2\check{L}_0\sqrt{\mu+\lambda})}{\sqrt{\mu+\lambda}} \in [-2,2],
\end{align*}
where $\G(\lambda)$ is defined in (\ref{defupsilon}). Although derived with a different method, this result agrees with \cite[Theorem 1.1.I]{PLO}.
$\hfill \blacksquare$
\end{exa2}
Recall that (\ref{ODE}) is $R$-reversible by Remark \ref{reversiblerem}. Suppose $\pih_{\p,\epsilon}(L_\epsilon)$ is contained in the space $\ker(I-R)$, which is the case for the solutions considered in existence result Theorem \ref{maintheorem}.  Combining this with \ref{assS3}-3 and the Hamiltonian nature of system (\ref{slowp}) yields that $u_\s$ is symmetric about $\check{L}_0$, i.e. it holds $u_\s(\check{L}_0 + \xx) = u_\s(\check{L}_0 - \xx)$ for all $\xx \in \R$. Hence, system (\ref{slowintr}) obeys a reversible symmetry. Thus, if $\phi(\xx,\lambda)$ is a solution to (\ref{slowintr}), then so is $\xx \mapsto R_s\phi(2\check{L}_0 - \xx,\lambda)$, where $R_s$ is defined in Remark \ref{reversiblerem}. Hence, there exists non-trivial solutions $u_+(\xx,\lambda)$ and $u_-(\xx,\lambda)$ to
\begin{align} \partial_{\xx\xx} u = \left(\partial_u H_1 (u_\s(\xx),0,0) + \lambda\right)u, \label{initvalue} \end{align}
which are symmetric and antisymmetric about $\check{L}_0$, respectively. This allows for further simplification of condition (\ref{traces}) in Proposition \ref{slowm1}.
\begin{prop}[The case $m = 1$ with reversible symmetry] \label{slowm1reversible} Suppose $m = 1$ and $u_\s$ is symmetric about $\check{L}_0$. Take $\lambda \in \Ce_\Lambda \!\setminus\! \N(\E_{f,0})$. Let $u_+(\xx,\lambda)$ and $u_-(\xx,\lambda)$ be non-trivial solutions to (\ref{initvalue}), which are symmetric and antisymmetric about $\check{L}_0$, respectively. Condition (\ref{traces}) in Proposition \ref{slowm1} simplifies to
\begin{align}
\mathrm{Tr}(\Upsilon(\lambda)\T_s(2\check{L}_0,0,\lambda)) =  \frac{2}{\W(\lambda)}\left[\frac{d}{d\xx}[u_+(\xx,\lambda)u_-(\xx,\lambda)](0) - \G(\lambda)u_+(0,\lambda)u_-(0,\lambda)\right] \in [-2,2] \label{cond4}
\end{align}
where $\W(\lambda)$ denotes the Wronskian of $u_+(\cdot,\lambda)$ and $u_-(\cdot,\lambda)$ and
\begin{align*}
\G(\lambda) = \int_{-\infty}^\infty \left[\frac{\partial H_2}{\partial u} (u_0,v_{\ho}(x)) + \partial_v H_2(u_0,v_{\ho}(x))v_{in}(x,\lambda)\right]dx,
\end{align*}
with $v_{in}(x,\lambda)$ the unique (meromorphic) solution to the inhomogeneous problem,
\begin{align} D_2 \partial_{xx} v = \left(\partial_v G(u_0,v_{\ho}(x),0) + \lambda\right)v + \frac{\partial G}{\partial u}(u_0,v_{\ho}(x),0), \ \ \ v \in \C^n. \label{fastihom3}\end{align}
\end{prop}
\begin{proof}
Express $\T_s(2\check{L}_0,0,\lambda)$ in terms of the symmetric and antisymmetric solutions and invoke Proposition \ref{slowm1}.
\end{proof}
\begin{exa2}[\textit{Sinusoidal Gierer-Meinhardt model}]
In the companion paper \cite{RIJV} a sinusoidal Gierer-Meinhardt equation (\ref{GM}) with $f(u) = \sin(u)$ is studied. In this special case (\ref{initvalue}) is known as Lam\'e's equation. The symmetric and antisymmetric solutions $u_\pm$ to this equation are known in terms of Jacobi Theta and Zeta functions. Moreover, the homogeneous fast limit problem (\ref{redstab}) is solvable in terms of hypergeometric functions. Applying the variation of constants formula gives the solution $v_{in}(x,\lambda)$ to problem (\ref{fastinhom}). Therefore, all terms in the right hand side of (\ref{cond4}) are explicitly known. This example shows that also in the slowly nonlinear situation the spectral stability analysis can be very explicit. The sinusoidal Gierer-Meinhardt equation serves as a guiding example in \cite{RIJV} to study the spectral phenomena occurring when the periodic patterns approach the homoclinic limit -- see also Section \ref{future}. $\hfill \blacksquare$
\end{exa2}
As mentioned in Section \ref{zeropolesec}, it is possible to obtain explicit expressions of the principle part of the Laurent series of $\E_{s,0}(\cdot,\gamma)$ at a zero $\lambda \in \N(\E_{f,0})$. Because of their complexity the expansions are treated separately in Section \ref{slows}. However, in the case $m = 1$, the expressions simplify significantly. Therefore, it is worthwhile to devote a separate Proposition to this case.
\begin{prop}[Singular part in the case $m = 1$] \label{singm1} Suppose $m = 1$. Let $\lambda_\diamond$ be a simple zero of $\E_{f,0}$ and $\gamma \in S^1$. The singular part of the Laurent expansion of $\E_{s,0}(\lambda,\gamma)$ at $\lambda = \lambda_\diamond$ is given by
\begin{align*} \frac{-\gamma u(2\check{L}_0,\lambda_\diamond)}{\lambda - \lambda_\diamond} \int_{-\infty}^\infty \partial_v H_2(u_0,v_{\ho}(x)) v_{\lambda_\diamond}(x)dx \int_{-\infty}^\infty v_{\lambda_\diamond}(x)^* \frac{\partial G}{\partial u} (u_0,v_{\ho}(x),0) dx, \end{align*}
where $u(\xx,\lambda_\diamond)$ is the solution to (\ref{initvalue}) at $\lambda = \lambda_\diamond$ with initial values $u(0) = 0, u'(0) = 1$ and $v_{\lambda_\diamond}$ is a normalized bounded solution (having $L^2$-norm $1$) to
\begin{align} D_2 \partial_{xx} v = \left(\partial_v G(u_0,v_{\ho}(x),0) + \lambda_\diamond\right)v, \ \ \  v \in \C^n. \label{sechom}\end{align}
\end{prop}
\begin{proof}
The statement is proven in a more general setting in Theorem \ref{Sturminhom}-4.
\end{proof}
\subsection{Robustness of zero-pole cancelation and an example for non-cancelation} \label{failure}
Proposition \ref{singm1} shows that for $m = 1$ the slow reduced Evans function $\E_{s,0}(\cdot,\gamma)$ has a removable singularity at a simple zero $\lambda_\diamond$ of $\E_{f,0}$ if and only if one of the identities (\ref{idpoles1}), (\ref{idpoles2}) holds true or there exists a non-trivial solution to (\ref{initvalue}) at $\lambda = \lambda_\diamond$ with boundary values $u(0) = 0 = u(2\check{L}_0)$. The set of $\lambda_\diamond \in \C$ for which (\ref{idpoles1}) or (\ref{idpoles2}) holds true will in general be discrete, since the involved expressions are analytic in $\lambda_\diamond$. Moreover, \cite[Theorem 4.3.1-6]{ZET} shows that this is also the case for the set of $\lambda_\diamond \in \C$ for which the boundary value problem \eqref{initvalue}, $u(0)=0=u(2\check{L}_0)$ admits a non-trivial solution. Hence, zero-pole cancelation is a robust phenomenon in the absence of additional structure (such as the translational invariance at $\lambda = 0$ mentioned in Remark \ref{translationinvariance}). \\
\\
Being robust, zero-pole cancelation can still fail in one-parameter families. Consider the Gierer-Meinhardt equation (\ref{GM}), where $\alpha_2 \neq 0$ and $f(u) = -\mu u$ with $\mu > 0$. We emphasize that in this case the slow reduced system \eqref{slowp} is linear of center type. This differs from the `standard' Gierer-Meinhardt setting considered in \cite{DGK2,IWW,PLO2,PLO,WAW}, where $f(u) = +\mu u$ and the slow reduced system is linear of saddle type.\\
\\
Let $u_0 > 0$. Note that (\ref{GM}) satisfies \ref{assS1}-\ref{assS2} and \ref{assE1} with $v_{\ho}(x,u_0) > 0$ for all $x \in \R$. Take $u_1 < 0$ such that $\mu u_1^2 = \J(u_0)^2 + \mu u_0^2$. One easily checks that the four conditions in Remark \ref{E2m1} holds true. Therefore, assumption \ref{assE2} is satisfied. Hence, Theorem \ref{maintheorem} implies that, for $\epsilon > 0$ sufficiently small, there exists a $2L_\epsilon$-periodic solution $\pih_{\p,\epsilon}$ to (\ref{GM}) satisfying assertions \ref{assS3}-1,2,3. Moreover, it holds by Remark \ref{L0m1}
\begin{align*}\check{L}_0(\mu) = \frac{\pi}{2} + \sin^{-1} \frac{u_0}{\sqrt{\tfrac{\mathcal{J}(u_0)^2}{\mu} + u_0^2}}.\end{align*}
In \cite[Lemma 3.3]{PLO} it is shown that $\lambda_0 = 1/4(\beta_2 + 1)^2 - 1 > 0$ is the positive zero of the fast reduced Evans function $\E_{f,0}$. Note that both $\partial_v H_2(u_0,v_{\ho}(x))$ and $\partial_u G(u_0,v_{\ho}(x),0)$ are strictly negative for all $x \in \R$. Moreover, the $v$-component of any non-trivial solution to (\ref{redstab}) at $\lambda = \lambda_0$ has no zeros by Proposition \ref{fastm1}. Therefore, identities (\ref{idpoles1}) and (\ref{idpoles2}) are not satisfied. Now assume $\mu > \lambda_0$. The solution to (\ref{initvalue}) at $\lambda = \lambda_0$ with initial values $u(0) = 0, u'(0) = 1$ is given by,
\begin{align*}u(\xx,\lambda_0) = \frac{\sin(\sqrt{\mu - \lambda_0} \xx)}{\sqrt{\mu - \lambda_0}}.\end{align*}
Clearly, it holds $u(2\check{L}_0,\lambda_0) = 0$ if and only if
\begin{align}\mu = \lambda_0 + \left(\frac{k\pi}{2\check{L}_0(\mu)}\right)^2,\label{zeropolefail}\end{align}
for some $k \in \Z_{\geq 1}$. Since $\check{L}_0(\mu) \in (\pi/2,\pi)$ for every $\mu > 0$, equation (\ref{zeropolefail}) will have a solution $\mu = \mu_k > 0$ for every $k \in \Z_{\geq 1}$. We conclude with the aid of Propositions \ref{zeropole} and \ref{singm1} that for every $\mu > 0$ there exists $\delta_\mu > 0$ such that the disc $B(\lambda_0,\delta_\mu)$ contains, provided $\epsilon > 0$ is sufficiently small, either no spectrum of $\El_\epsilon$ in the case $\mu \notin \{\mu_k : k \in \Z_{\geq 1}\}$ or a $\gamma$-eigenvalue of $\El_\epsilon$ for every $\gamma \in S^1$ in the case $\mu = \mu_k$ for some $k \in \Z_{\geq 1}$.\\
\\
The transition of $\mu$ through a point $\mu_k$ may seem like a blue sky catastrophe, which makes the pulse solution $\pih_{\p,\epsilon}$ `suddenly' spectrally unstable. However, such a transition from cancelation to non-cancelation is caused by unstable spectrum moving through the point $\lambda_0$. This can be seen by taking $\mu$ sufficiently close to $\mu_k$ and comparing the values of the trace $\mathrm{Tr}(\Upsilon(\lambda)\T_s(2\check{L}_0,0,\lambda))$ at $\lambda = \mu - [k\pi/(2\check{L}_0(\mu))]^2$ and when $\lambda$ approaches $\lambda_0$. The intermediate value Theorem and Proposition \ref{singm1} imply that there must be a zero of $\E_{s,0}(\cdot,\gamma)$ close to $\lambda_0$ for some $\gamma \in S^1$.
\subsection{An instability criterion}
As a further application of our analysis, we show how to test for spectral instability using a parity argument. More precisely, one can gain insight in the values of $t(\lambda)$, defined in \eqref{traces}, at a pole using Proposition \ref{singm1}, at $\lambda = 0$ due to translational invariance and by taking the limit for $\lambda \in \R$ to infinity. Subsequently, one determines with the mean value theorem, if the graph of $t(\lambda)$ crosses $[-2,2]$. The use of such a parity argument is common in Evans function based stability analyses. See \cite[Chapter 4.2]{SAN2} for an overview of the literature on this topic.
\begin{cor}
Suppose $m = n = 1$. Let $\lambda_0$ be as in Proposition \ref{fastm1}. Denote by $u(\xx,\lambda_0)$ the solution to (\ref{initvalue}) at $\lambda = \lambda_0$ with initial values $u(0) = 0, u'(0) = 1$. Moreover, let $v_{\lambda_0}$ be the normalized bounded solution (having $L^2$-norm $1$) to (\ref{sechom}) at $\lambda = \lambda_0$. Finally, let $v_{\ho}(x,u)$ for $u \in U_\ho$ be as in Remark \ref{extE1}. Consider the quantities,
\begin{align*} \I_1 &:= u(2\check{L}_0,\lambda_0)\int_{-\infty}^\infty \frac{\partial H_2}{\partial v}(u_0,v_{\ho}(x)) v_{\lambda_0}(x)dx \int_{-\infty}^\infty \frac{\partial G}{\partial u} (u_0,v_{\ho}(x),0) v_{\lambda_0}(x)dx,\\
\I_2 &:= u_\s'(0)(\G_0 u_\s'(0) - 2H_1(u_\s(0),0,0)) \int_{0}^{\check{L}_0} \frac{(\partial_u H_1 (u_\s(\xx),0,0) + 1)[(u_\s'(\xx))^2 - (H_1(u_\s(\xx),0,0))^2]}{[(u_\s'(\xx))^2 + (H_1(u_\s(\xx),0,0))^2]^2}d\xx \\
& \ \ \ \ \ \ \ \ \ \ \ \ + \frac{(u_\s'(0))^2 + \G_0 H_1(u_\s(0),0,0)u_\s'(0) - (H_1(u_\s(0),0,0))^2}{(u_\s'(0))^2 + (H_1(u_\s(0),0,0))^2},
\end{align*}
where
\begin{align*}
 \G_0 = \int_{-\infty}^\infty \left.\frac{\partial H_2}{\partial u} (u,v_{\ho}(x,u))\right|_{u = u_0} dx.
\end{align*}
Then, the periodic pulse solution $\pih_{\p,\epsilon}$ is spectrally unstable, for $\epsilon > 0$ sufficiently small, if:
\begin{center}
(i.) $\I_1 \leq 0$ or (ii.) $u_\s$ is symmetric about $\check{L}_0$ and it holds $\I_2 > -1$.
\end{center}
\end{cor}
\begin{proof} By Proposition \ref{slowm1} the slow reduced Evans function $\E_{s,0}(\lambda,\gamma)$ has a zero for some $\gamma \in S^1$ if and only if \eqref{traces} holds true. Our approach is to gain information about the trace $t(\lambda)$ at $\lambda = 0$, $\lambda = \lambda_0 > 0$ and when $\lambda \in \R$ tends to infinity. An application of the intermediate value Theorem will eventually yield the result. We distinguish three cases: (1.) $\I_1 < 0$, (2.) $\I_1 = 0$ and (3.) $\I_0 > 0, \I_2 > - 1$ and $u_\s$ symmetric about $\check{L}_0$.
\newpage
\textit{The case $\I_1 < 0$}\\
By combining Propositions \ref{slowm1} and \ref{singm1} the singular part of the trace $t(\lambda)$ at $\lambda = \lambda_0 \in \N(\E_{f,0})$ equals $\I_1$. On the other hand, it is shown in Lemma \ref{C2} that $t(\lambda)$ tends to infinity as $\lambda \to \infty$. Therefore, Proposition \ref{fastm1} and the intermediate value Theorem imply that there exists $\lambda_* \in (\lambda_0,\infty)$ such that $t(\lambda_*) \in [-2,2]$, if $\I_1 < 0$ holds true. Hence, the periodic pulse solution $\pih_{\p,\epsilon}$ is, provided $\epsilon > 0$ is sufficiently small, spectrally unstable by Proposition \ref{slowm1} and Corollary \ref{fixgammacor}-1.\\
\\
\textit{The case $\I_1 = 0$}\\
If we have $\I_1 = 0$, then $\E_{s,0}(\cdot,\gamma)$ has no pole at $\lambda_0 \in \N(\E_{f,0})$ for each $\gamma \in S^1$ by Proposition \ref{singm1}, which implies $\pih_{\p,\epsilon}$ is spectrally unstable, for $\epsilon > 0$ sufficiently small, by Corollary \ref{fixgammacor}-2.\\
\\
\textit{The case $\I_1 > 0$, $\I_2 > -1$ and $u_\s$ symmetric about $\check{L}_0$}\\
At $\lambda = 0$ the derivative $u_\s'(x)$ is due to translational invariance a non-trivial solution to (\ref{initvalue}), which is anti-symmetric about $\check{L}_0$. By Rofe-Beketov's formula \cite[Chapter 1.9]{BES} a linear independent solution to (\ref{initvalue}) is given by
  \begin{align*} z(\xx) &:= u_\s'(\xx)\int_{\check{L}_0}^\xx \frac{(\partial_u H_1 (u_\s(\yy),0,0) + 1)[(u_\s'(\yy))^2 - (H_1(u_\s(\yy),0,0))^2]}{[(u_\s'(\yy))^2 + (H_1(u_\s(\yy),0,0))^2]^2}d\yy\\
   &\ \ \ \ \ \ \ \ \ \ \ \ -\frac{H_1(u_\s(\xx),0,0)}{(u_\s'(\xx))^2 + (H_1(u_\s(\xx),0,0))^2}.\end{align*}
Note that $z(x)$ is symmetric about $\check{L}_0$ and that the Wronskian of $z$ and $u_\s'$ is constant with value $1$. Furthermore, there exists $C \in \R$ such that the solution $v_{in}(x,0)$ to (\ref{fastihom3}) is given by
\begin{align*} v_{in}(x,0) = \left.\frac{\partial v_{\ho}}{\partial u}(x,u)\right|_{u = u_0} + Cv_{\ho}'(x).\end{align*}
Substituting the obtained identities in the expression for the trace in Proposition \ref{slowm1reversible} yields $t(0) = 2\I_2$. Again, the intermediate value Theorem implies that there exists a $\lambda_* \in (0,\lambda_0)$ such that $t(\lambda_*) \in [-2,2]$ if $\I_1 > 0$ and $\I_2 > -1$ holds true. Hence, the periodic pulse solution $\pih_{\p,\epsilon}$ is, provided $\epsilon > 0$ is sufficiently small, spectrally unstable by combining Proposition \ref{slowm1} and Corollary \ref{fixgammacor}-1.
\end{proof}
\section{The Riccati transformation} \label{sec5}
The Riccati transformation provides a tool for the decoupling of singularly perturbed linear systems by block diagonalization. The decoupling takes place via a linear non-autonomous transformation, which is determined by two matrix Riccati equations. Specifically, for bounded matrix valued $A_{ij}(x,\eps)$ the coupled system,
\begin{align}
\left\{\!\begin{array}{rcl}
 \partial_x \phi &=& \eps (A_{11}(x,\eps) \phi + A_{12}(x,\eps)\psi)\\
 \partial_x \psi &=& A_{21}(x,\eps)\phi + A_{22}(x,\eps)\psi
\end{array}\right., \ \ \ (\phi,\psi) \in \C^{n_1+n_2} \label{ric1}
\end{align}
with $0 < \eps \ll 1$, decouples via a coordinate change (the Riccati transformation)  to
\begin{align}
\left\{\!\begin{array}{rcl}
\partial_x \chi &=& \eps\left[A_{11}(x,\eps) + A_{12}(x,\eps)U_\eps(x)\right]\chi\\
\partial_x \omega &=& \left[A_{22}(x,\eps) - \eps U_\eps(x)A_{12}(x,\eps)\right]\omega
\end{array}\right., \ \ \ (\chi,\omega) \in \C^{n_1+n_2}, \label{ric8}
\end{align}
where  $U_\eps(x)$ is a family of matrices satisfying a certain matrix Riccati equation as detailed below. Applying this decoupling to the linear stability problem (\ref{fullstab}) is a major step in reducing our stability problem to a `fast' and a `slow' stability problem yielding the factorization \eqref{factorintr} of the Evans function.\\
\\
Although the construction of the transformation is based on two results of Chang \cite[Theorem 1]{CHA} and \cite[Lemma 1]{CH2}, the assumptions on the coefficient matrices in \cite{CH2} are too restrictive. Therefore, we need a refinement of his statements. For this reason and the fact that the Riccati transformation lies at the core of our analytic factorization method, we present the full construction of the transformation in the next theorem. Moreover, we also prove that a periodic coefficient matrix implies periodicity of the Riccati transform, which appears to be a new result -- see Remark \ref{newperiod}.
\begin{theo} \label{ric}
Let $n_1,n_2 \in \Z_{> 0}, \eps_0 \in \R_{>0}$ and $A_{ij} \in C_b(\R \times (0,\eps_0),\mathrm{Mat}_{n_i \times n_j}(\C)), i,j = 1,2$. Assume that
\begin{align} \partial_x \psi = A_{22}(x,\eps)\psi, \ \ \ \psi \in \C^{n_2},\label{ric2} \end{align}
admits an exponential dichotomy on $\R$ with constants $K_0,\mu_0 > 0$, independent of $\eps$. Then, for $\eps > 0$ sufficiently small, there exists continuously differentiable matrix functions $U_\eps(x)$ and $S_\eps(x)$ satisfying the matrix Riccati equations,
\begin{align}
\left\{\!\begin{array}{rclr}
\partial_x U &=& A_{22}U - \eps U A_{11} - \eps U A_{12} U + A_{21}, \ \ \ &U \in \mathrm{Mat}_{n_2 \times n_1}(\C),\\
\partial_x S &=& \eps (A_{11} + A_{12}U)S - S(A_{22} - \eps U A_{12}) - A_{12}, \ \ \ &S \in \mathrm{Mat}_{n_1 \times n_2}(\C),
\end{array}\right.  \label{ric5}
\end{align}
with the following properties:
\begin{enumerate}
 \item[1.] $U_\eps$ and $S_\eps$ have $\eps$-independent bounds:
 \begin{align} \|U_\eps\| \leq 8K_0 \mu_0^{-1}\|A_{21}\|, \ \ \ \|S_\eps\| \leq 8C\mu_0^{-1}\|A_{12}\|, \label{ric12} \end{align}
 where $C > 0$ is a constant depending on $K_0$ only.
 \item[2.] The coordinate transform,
\begin{align}
\left(\begin{array}{c} \phi \\ \psi\end{array}\right) = H_\eps(x) \left(\begin{array}{c} \chi \\ \omega\end{array}\right), \ \ \ H_\eps(x) := \left(\begin{array}{cc} I & -\eps S_\eps(x)\\ U_\eps(x) & I - \eps U_\eps(x)S_\eps(x)\end{array}\right), \label{ric7}
\end{align}
diagonalizes system (\ref{ric1}) into (\ref{ric8}).
\item[3.] Let $\theta \in (0,1)$. The unique bounded solution $\Omega_\eps$ to the inhomogeneous matrix problem,
\begin{align}\partial_x \Omega = A_{22}(x,\eps)\Omega + A_{21}(x,\eps), \ \ \ \Omega \in \mathrm{Mat}_{n_2 \times n_1}(\R,\C), \label{ricinhom}\end{align}
satisfies for $x \in \R$
\begin{align*}\|U_\eps(x) - \Omega_\eps(x)\| \leq C\eps^{1-\theta},\end{align*}
where $C > 0$ is a constant depending on $K_0,\mu_0,\|A_{11}\|,\|A_{12}\|$ and $\|A_{21}\|$ only.
\item[4.] Let $a > 0$. For each $x \in \R$, we have the approximation,
\begin{align}
\|U_\eps(x)\| \leq \frac{4K_0}{\mu_0}\left[\sup_{y \in [x-a,x+a]}\left(\eps\|A_{12}\| \|U_\eps(y)\|^2 + \|A_{21}(y,\eps)\|\right) + C e^{-a\mu_0/2}\right], \label{ric10}
\end{align}
where $C > 0$ is a constant depending on $K_0,\mu_0,\|A_{12}\|$ and $\|A_{21}\|$ only.
\item[5.] If the matrices $A_{ij}(\cdot,\eps)$ are $L$-periodic for $1 \leq i,j \leq 2$, then the coordinate transform $H_\eps$ is also $L$-periodic.
\end{enumerate}
\end{theo}
\begin{proof} First, we set up an integral equation for $U_\eps$ and prove global existence via a contraction argument. Since $U_\eps$ triangulizes the system, an integral equation for $S_\eps$ can be derived from the variation of constants formula. The first four properties of $U_\eps$ and $S_\eps$ follow readily from the integral equations they satisfy. Finally, periodicity of the transform is proven by exponential separation.\\
\\
\textit{Existence of $U_\eps(x)$.}\\
Since $A_{11}$ is uniformly bounded in $\eps$ on $\R$, system,
\begin{align}
 \partial_x \phi = \eps A_{11}(x,\eps)\phi, \ \ \ \phi \in \C^{n_1},\label{ric4}
\end{align}
has bounded growth on $\R$ with constants $K_1,\eps \mu_1 > 0$, where $K_1 = 1$ and $\mu_1$ is independent of $\eps$. Denote by $T_{1,\eps}(x,y)$ and $T_{2,\eps}(x,y)$ the evolution operators of systems (\ref{ric4}) and (\ref{ric2}), respectively. Take $\rho = 8K_0\mu_0^{-1}\|A_{21}\|$. The ball $B(0,\rho) \subset C_b(\R,\mathrm{Mat}_{n_2 \times n_1}(\C))$ is a metric space endowed with the supremum norm. We want to show that the map $\A_\eps \colon B(0,\rho) \to B(0,\rho)$ given by
\begin{align*}
(\A_\eps U)(x) =& \int_{-\infty}^x T_{2,\eps}^s(x,y)\left[-\eps U(y)A_{12}(y,\eps)U(y) + A_{21}(y,\eps)\right]T_{1,\eps}(y,x)dy\\ & - \int_x^\infty T_{2,\eps}^u(x,y)\left[-\eps U(y)A_{12}(y,\eps)U(y) + A_{21}(y,\eps)\right]T_{1,\eps}(y,x)dy,
\end{align*}
is a well-defined contraction. If $\eps > 0$ is sufficiently small, it holds for all $U \in B(0,\rho)$
\begin{align*}\|\A_\eps U\| \leq \frac{2K_0}{\mu_0 - \eps \mu_1}\left[\eps \rho^2 \|A_{12}\| + \|A_{21}\|\right] < \rho.\end{align*}
Therefore, $\A_\eps$ is well-defined. Similarly, provided $\eps > 0$ is sufficiently small, we estimate for $U_1,U_2 \in B(0,\rho)$
\begin{align*} \|\A_\eps U_1 - \A_\eps U_2\| \leq \frac{4\eps K_0\rho \|A_{12}\|}{\mu_2 - \eps \mu_1}\|U_1 - U_2\| < \|U_1 - U_2\|.\end{align*}
Hence, $\A_\eps$ is a contraction mapping. By the Banach fixed point Theorem the integral equation $\A_\eps U = U$ has a unique solution $U_\eps(x)$ in $B(0,\rho)$. It is readily seen by differentiating this integral equation that $U_\eps$ satisfies the matrix Riccati equation (\ref{ric5}). Moreover, we derive the bound on $U_\eps$ in (\ref{ric12}). Since (\ref{ric2}) has an exponential dichotomy on $\R$, (\ref{ricinhom}) admits a unique bounded solution $\Omega_\eps$ by Proposition \ref{inhomexpdi}. By Proposition \ref{Palmer}, it holds for $|x-y| \leq \eps^{-\theta}$
\begin{align}\|T_{1,\eps}(x,y) - I\| \leq C' \eps^{1-\theta}, \label{ric9}\end{align}
where $C' > 0$ depends on $\|A_{11}\|$ only. Using (\ref{ric9}) and the bound on $U_\eps$ in (\ref{ric12}), we estimate for $x \in \R$
\begin{align*} \|U_\eps(x) - \Omega_\eps(x)\| &= \left\|(\A_\eps U_\eps)(x)
- \int_{-\infty}^x T_{2,\eps}^s(x,y)A_{21}(y,\eps)dy + \int_x^\infty T_{2,\eps}^u(x,y)A_{21}(y,\eps)dy\right\| \\
&\leq 4K_0\mu_0^{-1}\left(\eps \rho^2\|A_{12}\| +  \eps^{1-\theta} C'\|A_{21}\| + e^{-\mu_0\eps^{-\theta}}\|A_{21}\|\right).
\end{align*}
This shows the third property. The fourth property follows by splitting the interval of integration of the two integrals in the right hand side of the identity $U_\eps(x) = (\A_\eps U_\eps)(x)$. We obtain four integrals over $(-\infty,x-a)$, $(x-a,x)$, $(x,x+a)$ and $(x+a,\infty)$, respectively. We estimate each integral separately and obtain approximation (\ref{ric10}) with $C = \rho^2\|A_{12}\| + \|A_{21}\|$.\\
\\
\textit{Existence of $S_\eps(x)$.}\\
Since $A_{11}, A_{12}$ and $U_\eps$ are uniformly bounded in $\eps$ on $\R$, system,
\begin{align} \partial_x \chi = \eps\left[A_{11}(x,\eps) + A_{12}(x,\eps)U_\eps(x)\right]\chi, \ \ \ \chi \in \C^{n_1}, \label{ric6}\end{align} has bounded growth on $\R$ with constants $K_3,\eps\mu_3 > 0$, where $K_3 = 1$ and $\mu_3$ is independent of $\eps$. On the other hand, equation,
\begin{align}\partial_x \omega = (A_{22}(x,\eps) - \eps U_\eps(x)A_{12}(x,\eps))\omega, \ \ \ \omega \in \C^{n_2}, \label{ric3} \end{align}
can be seen as a perturbation of (\ref{ric2}). By roughness (Proposition \ref{RoughnessR}) it therefore possesses an exponential dichotomy on $\R$ with constants $K_4,\mu_4 > 0$. Here, $K_4$ depends on $K_0$ only and we choose $\mu_4 = \mu_0/2$. Denote by $T_{3,\eps}(x,y)$ and $T_{4,\eps}(x,y)$ the evolution operators of systems (\ref{ric6}) and (\ref{ric3}), respectively. We define $S_\eps(x)$ via the variation of constants formula,
\begin{align*} S_\eps(x) = -\int_{-\infty}^x T_{3,\eps}(x,y)A_{12}(y,\eps)T_{4,\eps}^u(y,x)dy + \int_x^\infty T_{3,\eps}(x,y)A_{12}(y,\eps)T_{4,\eps}^s(y,x)dy, \end{align*}
We readily derive the bound on $S_\eps$ in (\ref{ric12}). This proves the first property. It is easily verified by differentiation that $S_\eps$ satisfies the matrix Riccati equation (\ref{ric5}). Finally, using $S_\eps$ and $U_\eps$ satisfy equations (\ref{ric5}), it is a straightforward calculation to see the change of variables (\ref{ric7}) transforms system (\ref{ric1}) into (\ref{ric8}). This proves the second property.\\
\\
\textit{Exponential separation and periodicity}\\
Only the fifth property remains to be proven. Our plan is to show that system (\ref{ric1}) is exponentially separated in the sense of \cite{PALE}. Subsequently, we make use of the fact that exponential separation preserves periodicity. Therefore, denote by $P_\eps(x),x \in \R$ the projections corresponding to the exponential dichotomy of system (\ref{ric3}) on $\R$, established in the latter paragraph. We define the following projections,
\begin{align*}
P_{1,\eps} = \left(\begin{array}{cc} 0 & 0 \\ 0 & P_\eps(0)\end{array}\right), \ \ \ P_{2,\eps} = \left(\begin{array}{cc} I & 0 \\ 0 & 0\end{array}\right), \ \ \ P_{3,\eps} = \left(\begin{array}{cc} 0 & 0 \\ 0 & I-P_\eps(0)\end{array}\right).
\end{align*}
Denote by $V_{i,\eps} \subset \C^{n_1+n_2}$ the range of the projection $P_{i,\eps}$ for $i = 1,2,3$. Let $m_2$ be the rank of $P_\eps(0)$. Using both the bounded growth of (\ref{ric6}) and the exponential dichotomy of (\ref{ric3}), we conclude system that (\ref{ric8}) is, for $\eps > 0$ sufficiently small, $(m_2,n_1,n_2-m_2)$-exponentially separated with respect to the decomposition $V_{1,\eps} \oplus V_{2,\eps} \oplus V_{3,\eps}$. As a result, system (\ref{ric1}) is also $(m_2,n_1,n_2-m_2)$-exponentially separated with respect to the decomposition $W_{1,\eps} \oplus W_{2,\eps} \oplus W_{3,\eps}$, where $W_{i,\eps}$ is the range of the projection $Q_{i,\eps} := H_\eps(0)P_{i,\eps}H_\eps^{-1}(0)$ for $i = 1,2,3$.\\
\\
Now, suppose $A_{ij}(\cdot,\eps)$ are $L$-periodic for $1 \leq i,j \leq 2$. Let $X_\eps(x)$ be the principle fundamental matrix of system (\ref{ric1}) with $X_\eps(0) = I$. Invoking \cite[Corollary 4]{BYV} gives that $X_\eps(\cdot)Q_{2,\eps}X_\eps(\cdot)^{-1}$ is $L$-periodic. Denote by $T_\eps(x,y)$ the evolution of the diagonal system (\ref{ric8}). We calculate for $x \in \R$
\begin{align*}
X_\eps(x)Q_{2,\eps}X_\eps(x)^{-1} &= H_\eps(x)T_\eps(x,0)P_{2,\eps}T_\eps(0,x)H_\eps(x)^{-1} = H_\eps(x)P_{2,\eps}H_\eps(x)^{-1}\\
&= \left(\begin{array}{cc} I - \eps S_\eps(x) U_\eps(x) & \eps S_\eps(x) \\ U_\eps(x) + \eps U_\eps(x)S_\eps(x)U_\eps(x) & \eps U_\eps(x)S_\eps(x)\end{array}\right).
\end{align*}
Hence, $S_\eps, U_\eps S_\eps, S_\eps  U_\eps$ and $U_\eps + \eps U_\eps S_\eps U_\eps$ are $L$-periodic.
So, $U_\eps S_\eps U_\eps$ is also $L$-periodic. Combining this with the $L$-periodicity of $U_\eps + \eps U_\eps S_\eps U_\eps$, we conclude that $U_\eps$ is $L$-periodic. This implies that $H_\eps$ is $L$-periodic, which concludes the proof of the fifth statement.
\end{proof}
\begin{rem}[\textit{Periodicity of the transform}] \label{newperiod}
The periodicity of the transform in Theorem \ref{ric}-5 is a new discovery to the authors' knowledge. It is natural to ask whether there always exists a periodic choice for a coordinate change, which transforms a periodic system into diagonal form. However, it is shown in \cite[Chapter 5]{PALO} that this is not the case. It seems that the periodicity of the coordinate change $H_\eps$ is due to the special (slow-fast) structure of system (\ref{ric1}) used in the (non-generic) proof of Theorem \ref{ric}-5. $\hfill \blacksquare$
\end{rem}
\begin{rem}[\textit{Exponential trichotomies}]
The concept of exponential dichotomies can be generalized to exponential trichotomies \cite{SEL}. In addition to a stable and unstable direction, a center direction is considered. In fact, the $(m_2,n_1,n_2-m_2)$-exponential separation of systems (\ref{ric1}) and (\ref{ric8}) obtained in Theorem \ref{ric} shows that these systems admit an exponential trichotomy on $\R$. Here, the stable and unstable directions are related to the fast subsystem and the center direction is related to the slow subsystem. $\hfill \blacksquare$
\end{rem}
\begin{rem}[\textit{The Riccati transform in a broader setting}] Formally, the Riccati transform can be employed to diagonalize general linear equations as pointed out in \cite[Remark 4.7]{BEM}. However, the Riccati solutions can become singular in finite time. We use both the slow-fast structure of (\ref{ric1}) and the exponential dichotomy of (\ref{ric2}) to achieve global boundedness of the transformation functions $U_\eps$ and $S_\eps$. $\hfill \blacksquare$
\end{rem}
\section{Proofs of the main results via the analytic factorization method} \label{sec6}
\subsection{Approach} \label{Approach}
As mentioned in the introduction we employ an analytic factorization method to prove our main Theorems \ref{mainresult}, \ref{fixgamma} and \ref{zeropole}. We start by showing that the spectrum of $\El_\epsilon$ is contained in an $\epsilon$-independent sector. This provides an important a priori bound on the magnitude of the spectrum. The construction of the $\epsilon$-independent sector is carried out in Section \ref{4.1}.\\
\\
Subsequently, we start constructing the reduced Evans function $\E_0 \colon \Ce_\Lambda \times \C \to \C$, whose zero set approximates the zero set of the Evans function $\E_\epsilon$. As stated in section \ref{mainresult3}, $\E_0$ is defined in terms of the three singular limit problems (\ref{redstab}), (\ref{fastinhom}) and (\ref{slowintr}) of the linear stability problem (\ref{fullstab}). We study these three singular limit problems first separately in Section \ref{4.2} before drawing the connection between $\E_0$ and the Evans function $\E_\epsilon$. \\
\\
The reduced Evans function $\E_0$ is a product (\ref{e:evprod}) of an analytic fast reduced Evans function $\E_{f,0}$ and a meromorphic slow reduced Evans function $\E_{s,0}$. This factorization of $\E_0$ corresponds to the fact that system (\ref{fullstab}) can be diagonalized using the Riccati transform. Indeed system (\ref{fullstab}) is clearly of the form (\ref{ric1}). However, the application of the Riccati transformation is only legitimate in the case system,
\begin{align}
\partial_x \psi = \A_{22,\epsilon}(x,\lambda)\psi, \ \ \ \psi \in \C^{2n}, \label{faststab}
\end{align}
has an exponential dichotomy on $\R$. Using roughness techniques we can show that the latter is the case, whenever $\lambda$ is away from the discrete set of eigenvalues of the homogeneous fast limit problem (\ref{redstab}), which correspond to the zeros of the fast reduced Evans function $\E_{f,0}$. This result is the content of Section \ref{4.3}. Consequently, using the periodicity of the coefficient matrices of system (\ref{fullstab}), it is possible to factorize the Evans function $\E_\epsilon$ into two factor $\E_{f,\epsilon}$ and $\E_{s,\epsilon}$. This is performed in Section \ref{4.4}. The two factors $\E_{f,\epsilon}$ and $\E_{s,\epsilon}$ can be linked to $\E_{f,0}$ and $\E_{s,0}$, respectively. The latter is the content of Section \ref{4.rela}. Finally, an application of Rouch\'e's Theorem yields the proofs of Theorems  \ref{mainresult}, \ref{fixgamma} and \ref{zeropole}. This can be found in Section \ref{4.5}.
\subsection{A priori bounds on the spectrum} \label{4.1}
In this section we show that the spectrum of $\El_\epsilon$ is contained in an $\epsilon$-independent sector.
\begin{prop} \label{structurespectrum}
The operator $\El_\epsilon \colon C_{ub}^2(\R,\R^{m+n}) \subset C_{ub}(\R,\R^{m+n}) \to C_{ub}(\R,\R^{m+n})$ given by (\ref{defL}) is sectorial and densely defined. For $\epsilon > 0$ sufficiently small, there exists constants $\omega \in \R_{> 0}$ and $\varpi \in (\pi/2,\pi)$, both independent of $\epsilon$, such that the sector $\Sigma := \{\lambda \in \C : \lambda \neq \omega, |\mathrm{arg}(\lambda - \omega)| \leq \varpi\} \cup \{\omega\}$ is contained in the resolvent set $\rho(\El_\epsilon)$.
\end{prop}
\begin{proof} It follows by \cite[Corollary 3.1.9.ii]{LUN} that $\El_\epsilon$ is densely defined. An application of \cite[Theorem 1.3.2]{HEN} shows that $\El_\epsilon$ is sectorial as sum of a sectorial and a bounded operator. At first sight this seems to be sufficient. However, the matrix $\B_\epsilon$ defined in (\ref{defB}) will in general have a norm of order $\ord(\epsilon^{-1})$. Therefore, we need to follow a different path to prove that the sector corresponding to $\El_\epsilon$ may be chosen independent of $\epsilon$. \\
\\
Our approach is to decompose $\El_\epsilon$ in more elementary building blocks to control the $\epsilon^{-1}$-terms in $\B_\epsilon$. First, we show that the operator $\El_{1,\epsilon} \colon C_{ub}^2(\R,\R^m) \subset C_{ub}(\R,\R^m) \to C_{ub}(\R,\R^m)$ given by
\begin{align*} \El_{1,\epsilon}(u) = (D_1 \partial_{\xx\xx} + \epsilon^{-1} \partial_u H_2 (\check{u}_{\p,\epsilon}(\cdot),\check{v}_{\p,\epsilon}(\cdot)))u,\end{align*}
is sectorial with an $\epsilon$-independent sector. Subsequently, we prove this for $\widehat{\El}_\epsilon \colon C_{ub}^2(\R,\R^{m+n}) \subset C_{ub}(\R,\R^{m+n}) \to C_{ub}(\R,\R^{m+n})$ given by
\begin{align*}\left(\begin{array}{c} u \\ v \end{array}\right) \mapsto \left(\begin{array}{c} \El_{1,\epsilon}(u) + \epsilon^{-1}\partial_v H_2 (\check{u}_{\p,\epsilon}(\cdot),\check{v}_{\p,\epsilon}(\cdot)) v \\ \epsilon^2 D_2 \partial_{\xx\xx} v\end{array}\right).\end{align*}
Finally, $\El_\epsilon$ can be seen as a perturbation of $\widehat{\El}_\epsilon$ by a bounded operator with $\ord(1)$-norm.\\
\\
\textit{The operator $\El_{1,\epsilon}$}\\
Our goal is to show that the spectrum of $\El_{1,\epsilon}$ is contained in an $\epsilon$-independent sector. By \cite[Proposition 2.1]{GAR} it is sufficient to show that the associated eigenvalue problem,
\begin{align}
\left\{\!
\begin{array}{rcl} \sqrt{D_1} \partial_\xx u &=& \sqrt{\lambda} p \\ \sqrt{D_1} \partial_\xx p &=& \left(\sqrt{\lambda} + \frac{\partial_u H_2 (\check{u}_{\p,\epsilon}(\xx),\check{v}_{\p,\epsilon}(\xx))}{\sqrt{\lambda}\epsilon}\right) u \end{array}\right., \ \ \ u,p \in \C^m, \label{linsys7}
\end{align}
has no non-trivial bounded solutions for $\lambda$ in some $\epsilon$-independent sector. Denote by $\T_{1,\epsilon}(\xx,\yy,\lambda)$ the evolution of system (\ref{linsys7}). Let $\T_1(\xx,\yy,\lambda)$ be the evolution of
\begin{align}
\left\{\!
\begin{array}{rcl} \sqrt{D_1} \partial_\xx u &=& \sqrt{\lambda} p \\ \sqrt{D_1} \partial_\xx p &=& \sqrt{\lambda} u \end{array}\right., \ \ \ u,p \in \C^m. \label{linsys6}
\end{align}
By solving system (\ref{linsys6}) explicitly, one observes that, whenever $\lambda \in \C \!\setminus\! \R_{\leq 0}$, system (\ref{linsys6}) has both bounded growth with constants $K, \mu_{\lambda,+} > 0$ and an exponential dichotomy with constants $K, \mu_{\lambda,-} > 0$. Here, we have $\mu_{\lambda,+} = \|D_1^{-1/2}\| |\mathrm{Re}(\sqrt{\lambda})|$,  $\mu_{\lambda,-} = \|D_1\|^{-1/2} |\mathrm{Re}(\sqrt{\lambda})|$ and $K = 1$. Since we have $\lim_{\epsilon \downarrow 0} \check{L}_\epsilon = \check{L}_0 > 0$, as stated in \ref{assS3}-1, there exists a constant $C_1 > 0$, independent of $\epsilon$, such that for all $\lambda \in S_1 := \{\mu \in \C : |\mathrm{Re}(\!\sqrt{\mu})| \geq C_1\}$ it holds $h_\lambda := \mu_{\lambda,-}^{-1}\sinh^{-1}(4) \leq \check{L}_\epsilon$.\\
\\
Let $\lambda \in S_1$. Note that $H_2$ vanishes at $v = 0$ by \ref{assS1}. Combining this with \ref{assS3}-3 yields $K_1,\mu_1 > 0$ such that
\begin{align*}\|\partial_u H_2 (\check{u}_{\p,\epsilon}(\xx),\check{v}_{\p,\epsilon}(\xx))\| \leq K_1e^{-\epsilon^{-1}\mu_1 |\xx|},\end{align*}
for $x \in [-\check{L}_\epsilon,\check{L}_\epsilon]$. Let $\ww,\zz \in \R$ such that $0 \leq \zz-\ww \leq 2h_\lambda \leq 2\check{L}_\epsilon$. Taking into account the $2\check{L}_\epsilon$-periodicity of $\check{u}_{\p,\epsilon}$ and $\check{v}_{\p,\epsilon}$, we have
\begin{align*} \int_\ww^\zz \frac{\|\partial_u H_2 (\check{u}_{\p,\epsilon}(\xx),\check{v}_{\p,\epsilon}(\xx))\|}{\sqrt{|\lambda|}\epsilon} d\xx \leq \frac{2K_1}{\mu_1\sqrt{|\lambda|}}.\end{align*}
By Lemma \ref{Palmer}, we conclude for $|\ww-\zz| \leq 2h_\lambda$
\begin{align*}\|\T_1(\zz,\ww,\lambda) - \T_{1,\epsilon}(\zz,\ww,\lambda)\| &\leq \frac{2K_1}{\mu_1\sqrt{|\lambda|}}\exp\left(2\mu_{\lambda,+} h_\lambda + \frac{2K_1}{\mu_1\sqrt{|\lambda|}}\right).
\end{align*}
Since $\mu_{\lambda,+}h_\lambda$ is independent of $\lambda$, we derive that there exists a constant $C_2 > 0$ such that, whenever $\lambda \in S_1$ satisfies $|\lambda| > C_2$, then it holds for all $|\ww-\zz| \leq 2h_\lambda$
\begin{align}\|\T_1(\zz,\ww,\lambda) - \T_{1,\epsilon}(\zz,\ww,\lambda)\| < 1. \label{evolest}\end{align}
Now let $\Sigma_1$ be an $\epsilon$-independent sector disjoint from $B(0,C_2) \cup [\C \! \setminus \! S_1]$ -- see Figure \ref{fig4:sub1}. For all $\lambda \in \Sigma_1$, there are no non-trivial bounded solutions to (\ref{linsys7}) by combining (\ref{evolest}) with Proposition \ref{roughnessPALI}. So, by \cite[Proposition 2.1]{GAR} the resolvent set $\rho(\El_{1,\epsilon})$ contains the $\epsilon$-independent sector $\Sigma_1$.\\
\\
\textit{The operator $\widehat{\El}_\epsilon$}\\
Consider the elliptic operator $\El_2 \colon C_{ub}^2(\R,\R^n) \subset C_{ub}(\R,\R^n) \to C_{ub}(\R,\R^n)$ given by $\El_2(v) = D_2 \partial_{\xx\xx} v$. Clearly, we have $\rho(\El_2) = \C \!\setminus\! \R_{\leq 0} \supset \Sigma_1$. For $\lambda \in \Sigma_1$ the operator on $C_{ub}(\R,\R^{m+n})$ defined by
\begin{align*} \left(\begin{array}{c} u \\ v \end{array}\right) \mapsto \left(\begin{array}{c} (\El_{1,\epsilon} - \lambda)^{-1}(u - \epsilon^{-1} \partial_v H_2 (\check{u}_{\p,\epsilon}(\cdot),\check{v}_{\p,\epsilon}(\cdot)) (\epsilon^2{\El}_2 - \lambda)^{-1}(v)) \\ (\epsilon^2{\El}_2 - \lambda)^{-1}(v) \end{array}\right),\end{align*}
is an inverse of $\widehat{\El}_\epsilon - \lambda$. Therefore, the sector $\Sigma_1$ is contained in the resolvent set $\rho(\widehat{\El}_\epsilon)$. \\
 \\
\textit{Conclusion}\\
Define
\begin{align*} \B_{b,\epsilon}(\xx) = \left(\begin{array}{cc} \partial_u H_1 (\check{u}_{\p,\epsilon}(\xx),\check{v}_{\p,\epsilon}(\xx),\epsilon) & \partial_v H_1 (\check{u}_{\p,\epsilon}(\xx),\check{v}_{\p,\epsilon}(\xx),\epsilon) \\  \partial_u G (\check{u}_{\p,\epsilon}(\xx),\check{v}_{\p,\epsilon}(\xx),\epsilon) & \partial_v G (\check{u}_{\p,\epsilon}(\xx),\check{v}_{\p,\epsilon}(\xx),\epsilon)\end{array}\right),
\end{align*}
Let $\El_{b,\epsilon} \colon C_{ub}(\R,\R^{m+n}) \to C_{ub}(\R,\R^{m+n})$ be the multiplication operator $[\El_{b,\epsilon} \phi](\xx) = \B_{b,\epsilon}(\xx)\phi$. By Remark \ref{S3.4} the norm of $\El_{b,\epsilon}$ satisfies $\|\El_{b,\epsilon}\| = \ord(1)$.\\
\\
Invoking \cite[Theorem 1.3.2]{HEN} and its proof yields the conclusion: the sum $\El_\epsilon = \widehat{\El}_\epsilon + \El_{b,\epsilon}$ with domain $C_{ub}^2(\R,\R^{m+n})$ is densely defined and sectorial with an $\epsilon$-independent sector $\Sigma \subset \rho(\El_\epsilon)$. Here, we have used $\|\El_{b,\epsilon}\| = \ord(1)$ and that $\Sigma_1$ is independent of $\epsilon$.
\end{proof}
Proposition \ref{structurespectrum} shows that the spectral stability of the periodic pulse solution $\pih_{\p,\epsilon}$ is fully determined by the $\gamma$-eigenvalues of $\El_\epsilon$ in one of the following family of regions in $\C$ -- see Figure \ref{fig4:sub2}.
\begin{nota}
Let $\omega\in\R_{> 0}$ and $\varpi \in (\pi/2,\pi)$ be as in Proposition \ref{structurespectrum}. For $\Lambda \in \R_{< 0}$ we denote
\begin{align*}\Sigma_{\Lambda,0} &:= \{\lambda \in \C : \mathrm{Re}(\lambda) \in (\Lambda,\omega), |\mathrm{arg}(\lambda - \omega)| < \pi - \varpi\}.\end{align*}
\end{nota}
\begin{figure}[t]
\centering
\begin{subfigure}[b]{.43\textwidth}
  \centering
  \includegraphics[width=\linewidth]{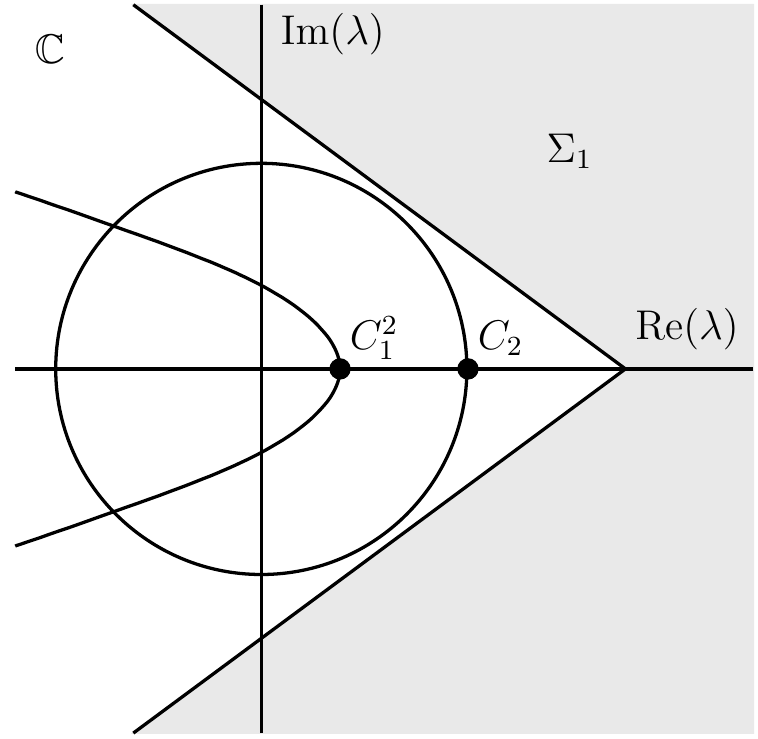}
  \caption{Construction of the sector $\Sigma_1$ in the proof of Proposition \ref{structurespectrum}.}
  \label{fig4:sub1}
\end{subfigure}%
\hspace{.04\textwidth}
\begin{subfigure}[b]{.43\textwidth}
  \centering
  \includegraphics[width=0.9\linewidth]{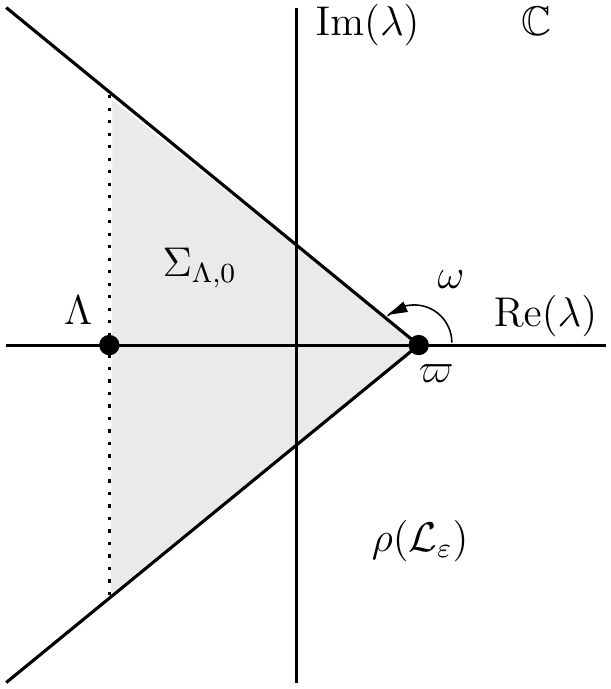}
  \caption{The region $\Sigma_{\Lambda,0}$ contains the spectrum $\sigma(\El_\epsilon)$ in $\Ce_\Lambda$.}
  \label{fig4:sub2}
\end{subfigure}
\caption{a priori bounds on the critical spectrum.}
\end{figure}
\subsection{The three singular limit problems} \label{4.2}
As a prerequisite to properly define and analyze the reduced Evans function $\E_0$ and its factors $\E_{f,0}$, $\E_{s,0}$, we treat the homogeneous fast limit problem (\ref{redstab}), the inhomogeneous fast limit problem (\ref{fastinhom}) and the slow limit problem (\ref{slowintr}), consecutively.
\subsubsection{Homogeneous fast limit problem}
The homogeneous fast limit problem (\ref{redstab}) can be seen as the singular limit of our linear stability problem (\ref{fullstab}), where the $u$-component is taken constantly $0$. To track solutions to (\ref{redstab}) properly, we establish an exponential dichotomies of system (\ref{redstab}) on both half lines. A sufficient condition for this is that the coefficient matrix of (\ref{redstab}) is asymptotically hyperbolic. With the tools that have become available when (\ref{redstab}) has exponential dichotomies on both half-lines, we are able to construct an analytic function $\E_{f,0}$ that detects the values of $\lambda \in \C$ for which (\ref{redstab}) has bounded solutions. This will be the fast reduced Evans function. The above is the content of the following lemma and theorem.
\begin{lem} \label{lemlambda0}
There exists $\Lambda_0 > 0$ such that for $\Lambda \in (-\Lambda_0,0)$ the spectrum of the matrix,
\begin{align*} A(u,\lambda) := \left(\begin{array}{cc} 0 & D_2^{-1} \\ \partial_v G(u,0,0) + \lambda & 0\end{array}\right),\end{align*}
is uniformly bounded away from the imaginary axis on $\K_U \times \Ce_\Lambda$, where $\K_U$ is as in Remark \ref{S3.4} and $\Ce_\Lambda$ as in Notation \ref{notCEL}.
\end{lem}
\begin{proof}
For a matrix $A \in \mathrm{Mat}_{n \times n}(\C)$ denote by $\F(A) = \{v^*Av : v \in \C^n, \|v\| = 1\}$ its field of values. Since $\partial_v G(u,0,0)$ has positive definite real part by \ref{assS2}, the field of values $\F(\partial_v G(u,0,0))$ is for every $u \in \K_U$ contained in the positive half plane by \cite[Property 1.2.5a]{HOR}. In fact, by compactness of $\K_U$ there exists $\Lambda_0 > 0$ such that we have $\F(\partial_v G(u,0,0)) \subset \Ce_{-\Lambda_0}$ for every $u \in \K_U$. Let $\Lambda \in (-\Lambda_0,0)$. For $u \in \K_U$ and $\lambda \in \Ce_\Lambda$ we establish using \cite[Property 1.2.3]{HOR} and \cite[Corollary 1.7.7]{HOR}
\begin{align*}\sigma((\partial_v G(u,0,0) + \lambda)D_2^{-1}) \subset (\F(\partial_v G(u,0,0)) + \lambda)\F(D_2^{-1}) \subset \left\{\mu \in \C : \mathrm{Re}(\mu) \geq d_{\max}^{-1}(\Lambda_0 + \Lambda)\right\},\end{align*}
where $d_{\max}$ is the largest diagonal value of $D_2$. The eigenvalues of $A(u,\lambda)$ are given by the square roots of the eigenvalues of $(\partial_v G(u,0,0) + \lambda)D_2^{-1}$. Therefore, we obtain for $u \in \K_U$ and $\lambda \in \Ce_\Lambda$ that any eigenvalue $\mu \in \sigma(A(u,\lambda))$ satisfies $|\mathrm{Re}(\mu)| \geq \cos(\pi/4)\sqrt{(\Lambda_0 + \Lambda)/d_{\max}}$, which concludes the proof.
\end{proof}
\begin{theo}\label{Sturm}
Let $\Lambda \in (-\Lambda_0,0)$ with $\Lambda_0 > 0$ as in Lemma \ref{lemlambda0}. System (\ref{redstab}) has exponential dichotomies on $[0,\infty)$ and $(-\infty,0]$ for every $\lambda \in \Ce_\Lambda$ with constants $K_r(\lambda),\mu_r(\lambda) > 0$, chosen continuously dependent on $\lambda$. Moreover, one can choose the corresponding rank $n$ projections $\Proj_{r\pm}(\pm x,\lambda), x \in [0,\infty)$ such that $\Proj_{r\pm}(x,\cdot)$ is analytic on $\Ce_\Lambda$.\\
\\
Let $B_r^{u,s} \colon \Ce_\Lambda \to \mathrm{Mat}_{2n \times n}(\C)$ be analytic bases of $\Proj_{r+}(0,\lambda)[\C^{2n}] = B_r^u(\lambda)[\C^n]$ and $\ker(\Proj_{r-}(0,\lambda)) = B_r^s(\lambda)[\C^n]$ for $\lambda \in \Ce_\Lambda$. The analytic function $\E_{f,0} \colon \Ce_\Lambda \to \C$ given by $\E_{f,0}(\lambda) = \det(B_r^u(\lambda),B_r^s(\lambda))$ has the following properties:
\begin{enumerate}
\item $\E_{f,0}(\lambda) = 0$ if and only if (\ref{redstab}) has a non-trivial (exponentially localized) solution in $C_b(\R,\C^{2n})$.
\item $\E_{f,0}(\lambda) \neq 0$ if and only if (\ref{redstab}) has an exponential dichotomy on $\R$.
\item The zero set of $\N(\E_{f,0}) \subset \Ce_\Lambda$ is discrete and independent of the choice of bases $B_r^{u,s}$.
\item A zero $\lambda$ of $\E_{f,0}$ is simple if and only if (\ref{redstab}) admits a unique non-trivial solution in $C_b(\R,\C^{2n})$ up to scalar multiples.
\end{enumerate}
\end{theo}
\begin{proof}
Let $\lambda \in \Ce_\Lambda$. The matrix $\A_{22,0}^\infty(\lambda) := A(u_0,\lambda)$ is hyperbolic by Lemma \ref{lemlambda0}. Recall that $\A_{22,0}(x,\lambda)$ is the coefficient matrix of (\ref{redstab}). By Remark \ref{homoclin} there exists $K_1,\mu_1 > 0$ such that $\|\A_{22,0}(x,\lambda) - \A_{22,0}^\infty(\lambda)\| \leq K_1e^{-\mu_1|x|}$ for $x \in \R$. Following \cite[Lemma 3.4]{PAL} and \cite[Theorem 1]{SAN} system (\ref{redstab}) has exponential dichotomies on both $[0,\infty)$ and $(-\infty,0]$ with the desired properties.\\
\\
By \cite[Proposition 2.1]{PAL} we have $\E_{f,0}(\lambda) \neq 0$ if and only if (\ref{redstab}) has an exponential dichotomy on $\R$. On the other hand, every solution $\phi \in C_b(\R,\C^{2n})$ to (\ref{redstab}) must satisfy $\phi(0) \in B_r^u(\lambda)[\C^n] \cap B_r^s(\lambda)[\C^n]$. This settles the first two properties. To prove the third property, we consider the operator $\El_r \colon C_b^2(\R,\R^n) \subset C_b(\R,\R^n) \to C_b(\R,\R^n)$ given by
\begin{align*}\El_r(v) = (D_2\partial_{xx}  - \partial_v G(u_0,v_{\ho}(\cdot),0))v.\end{align*}
By \cite[Theorem 1.3.2]{HEN} and \cite[Corollary 3.1.9.i]{LUN} $\El_r$ is sectorial. The eigenvalues of the sectorial operator $\El_r$ correspond to the zeros of the analytic map $\E_{f,0}$. We deduce that $\E_{f,0}$ has a finite number of zeros in $\Ce_\Lambda$. The third property has been proven. The fourth property follows immediately by \cite[Theorem 3.19]{BEY}.
\end{proof}
Although the following corollary is an easy consequence of the boundedness of $\Sigma_{\Lambda,0}$, it will be of great importance to obtain $\lambda$-uniform estimates later.
\begin{cor}\label{Sturmcor}
Let $\Lambda \in (-\Lambda_0,0)$ with $\Lambda_0 > 0$ as in Lemma \ref{lemlambda0}. System (\ref{redstab}) has exponential dichotomies on $[0,\infty)$ and $(-\infty,0]$ for every $\lambda \in \Sigma_{\Lambda,0}$ with constants $K_r,\mu_r > 0$, independent of $\lambda$.
\end{cor}
Another important consequence of the exponential dichotomies established in Theorem \ref{Sturm} is that the differential operator associated with (\ref{redstab}) is Fredholm.
\begin{cor} \label{Fredholmcor}
Let $\Lambda \in (-\Lambda_0,0)$ with $\Lambda_0 > 0$ as in Lemma \ref{lemlambda0}. For each $\lambda \in \Ce_{\Lambda}$ the bounded operator $\El_\lambda \colon C_b^1(\R,\C^{2n}) \to C_b(\R,\C^{2n})$ given by $\El_\lambda(\phi) = (\partial_x - \A_{22,0}(\cdot,\lambda))\phi$ is Fredholm of index 0. Moreover, $\El_\Lambda$ is invertible if and only if $\lambda \in \Ce_\Lambda \!\setminus\! \N(\E_{f,0})$.
\end{cor}
\begin{proof}
This is the content of \cite[Lemma 4.2]{PAL}.
\end{proof}
Finally, Theorem \ref{Sturm} provides us with the fast reduced Evans function and thereby proves Proposition \ref{fastredEv}.
\begin{defi}
Let $\Lambda \in (-\Lambda_0,0)$ with $\Lambda_0 > 0$ as in Lemma \ref{lemlambda0}. The map $\E_{f,0} \colon \Ce_\Lambda \to \C$ given by $\E_{f,0}(\lambda) = \det(B_r^u(\lambda),B_r^s(\lambda))$ obtained in Theorem \ref{Sturm} is called the \emph{fast reduced Evans function}.
\end{defi}
\subsubsection{Inhomogeneous fast limit problem}
The inhomogeneous fast limit problem (\ref{fastinhom}) can be seen as a family of layer problems associated with the singular limit of the linear stability problem (\ref{fullstab}). The matrix solution $\X_{in}(x,\lambda)$ to the inhomogeneous problem (\ref{fastinhom}) forms one of the key ingredients of the slow reduced Evans function $\E_{s,0}$. We will prove that $\X_{in}(x,\cdot)$ is meromorphic for each $x \in \R$. Singularities of $\X_{in}(x,\cdot)$ can only occur when $\lambda$ is one of the zeros of $\E_{f,0}$, which are obtained in Theorem \ref{Sturm}.
\begin{theo} \label{Sturminhom}
Let $\Lambda \in (-\Lambda_0,0)$ with $\Lambda_0 > 0$ as in Lemma \ref{lemlambda0}. There exists a unique continuous solution $\X_{in} \colon \R \times \Ce_\Lambda \!\setminus\! \N(\E_{f,0}) \to \mathrm{Mat}_{2n \times 2m}(\C)$ to equation (\ref{fastinhom}) with the following properties:
\begin{enumerate}
 \item $\X_{in}(x,\cdot)$ is meromorphic on $\Ce_{\Lambda}$ for all $x \in \R$.
 \item $\X_{in}(x,\cdot)$ is analytic on $\Ce_{\Lambda} \!\setminus\! \N(\E_{f,0})$ for all $x \in \R$.
 \item $\X_{in}(\cdot,\lambda)$ is exponentially localized for each $\lambda \in \Ce_{\Lambda} \!\setminus\! \N(\E_{f,0})$, i.e. there exists constants $K_{in}(\lambda),\mu_{in}(\lambda) > 0$, chosen continuously dependent on $\lambda$, such that for all $x \in \R$
 \begin{align*} \|\X_{in}(x,\lambda)\| \leq K_{in}(\lambda)e^{-\mu_{in}(\lambda)|x|}.\end{align*}
 \item Let $\lambda_\diamond \in \Ce_{\Lambda}$ be a simple zero of $\E_{f,0}$. Denote by $\phi_{\lambda_\diamond}(x)$ and $\psi_{\lambda_\diamond}(x)$ non-trivial bounded (exponentially localized) solutions to (\ref{redstab}) and its adjoint equation,
\begin{align}\partial_x \psi = -\psi\A_{22,0}(x,\lambda)^*, \ \ \  \psi \in \mathrm{Mat}_{1 \times n}(\C), \label{adjoint}\end{align}
 respectively, such that
\begin{align*} \int_{-\infty}^\infty \psi_{\lambda_\diamond}(z) \left(\begin{array}{cc} 0 & 0 \\ I & 0\end{array}\right) \phi_{\lambda_\diamond}(z)dz = 1.\end{align*}
There exists a neighborhood $B_{\lambda_\diamond} \subset \Ce_\Lambda$ of $\lambda_\diamond$ and a mapping $\X_{\lambda_\diamond,in} \colon \R \times B_{\lambda_\diamond} \to \mathrm{Mat}_{2n \times 2m}(\C)$, such that on $\R \times B_{\lambda_\diamond}$ we expand
 \begin{align*} \X_{in}(x,\lambda) = \frac{\phi_{\lambda_\diamond}(x)}{\lambda - \lambda_\diamond} \int_{-\infty}^\infty \psi_{\lambda_\diamond}(z)\A_{21,0}(z)dz + \X_{\lambda_\diamond,in}(x,\lambda).\end{align*}
Here, $\X_{\lambda_\diamond,in}(x,\cdot)$ is analytic on $B_{\lambda_\diamond}$ for every $x \in \R$. Moreover, $\X_{\lambda_\diamond,in}(\cdot,\lambda)$ is exponentially localized for every $\lambda \in B_{\lambda_\diamond}$.
\end{enumerate}
\end{theo}
\begin{proof}
For $\lambda \in \Ce_\Lambda$ the operator $\El_\lambda$ defined in Corollary \ref{Fredholmcor} is Fredholm of index 0 and $\El_\lambda$ is invertible if and only if $\lambda \in \Ce_\Lambda \!\setminus\! \N(\E_{f,0})$. Therefore, $\El_\lambda^{-1}$ is analytic on $\Ce_\Lambda \!\setminus\! \N(\E_{f,0})$ and meromorphic on $\Ce_{\Lambda}$ by \cite[Theorem 1.3.1]{MEN}. This settles the first two properties. \\
\\
First, (\ref{redstab}) has by Theorem \ref{Sturm} an exponential dichotomy on $\R$ for $\lambda \in \Ce_{\Lambda} \!\setminus\! \N(\E_{f,0})$ with constants chosen continuously dependent on $\lambda$. Second, recall that $G$ vanishes at $v = 0$ by \ref{assS1}. So, by Remark \ref{homoclin} $\A_{21,0}$ is exponentially localized. Combining these items with Proposition \ref{inhomexpdi} establishes the third property. By \cite[Theorem 3.19]{BEY} $\lambda_\diamond$ is a simple eigenvalue of the operator pencil $\lambda \mapsto \El_{\lambda}$ if and only if $\lambda_\diamond$ is a simple zero of $\E_{f,0}$. Hence, the fourth property follows by an application of Keldysh formula for $\El_\lambda^{-1}$ given in \cite[Theorem 1.6.5]{MEN}.
\end{proof}
\begin{rem}[\textit{Connection with Fredholm alternative}]
Let $f \in C_b(\R,\C^{2n})$. The Fredholm alternative in \cite[Lemma 4.2]{PAL} states that the inhomogeneous equation,
\begin{align*} \partial_x \phi = \A_{22,0}(x,\lambda)\phi + f(x), \ \ \ \phi \in \C^{2n},\end{align*}
has a bounded solution if and only if the solvability condition,
\begin{align*} \int_{-\infty}^\infty \psi(x)f(x) dx = 0,\end{align*}
is satisfied for all bounded solutions $\psi$ to the adjoint equation (\ref{adjoint}). This corresponds directly with the fact that $\X_{in}(x,\cdot)$ has a removable singularity at a simple zero $\lambda_\diamond$ of $\E_{f,0}$ if and only if we have
\begin{align*} \int_{-\infty}^\infty \psi_{\lambda_\diamond}(z)\A_{21,0}(z)dz = 0,\end{align*}
by Theorem \ref{Sturminhom}-4. $\hfill \blacksquare$
\end{rem}
It is possible to obtain expressions for the singular part of the Laurent series of $\X_{in}$ at a zero of $\E_{f,0}$ of higher multiplicity by looking at a canonical system of generalized eigenvectors. However, for simplicity of exposition we treat only one special case in the Proposition below, which requires significantly less notation. The interested reader is referred to \cite[Chapter 1]{MEN} for the general set-up.
\begin{prop} \label{Sturminhomspecial}
Let $\Lambda \in (-\Lambda_0,0)$ with $\Lambda_0 > 0$ as in Lemma \ref{lemlambda0}. Suppose $\partial_v G(u_0,v,0)$ is symmetric for all $v \in V$. Let $\lambda_\diamond \in \Ce_{\Lambda}$ be a zero of $\E_{f,0}$ of multiplicity $k_{\lambda_\diamond}$ and let $\phi_{1,\lambda_\diamond}(x),$ $\ldots,\phi_{k_{\lambda_\diamond},\lambda_\diamond}(x)$ form a basis of the space of bounded (exponentially localized) solutions to (\ref{redstab}) satisfying the orthogonality relations,
\begin{align*} \int_{-\infty}^\infty \phi_{i,\lambda_\diamond}(z)^*\left(\begin{array}{cc} I & 0 \\ 0 & 0\end{array}\right)\phi_{j,\lambda_\diamond}(z)dz = \delta_{ij}, \ \ \ 1 \leq i,j \leq k_{\lambda_\diamond}. \end{align*}
There exists a neighborhood $B_{\lambda_\diamond} \subset \Ce_\Lambda$ of $\lambda_\diamond$ and a mapping $\X_{\lambda_\diamond,in} \colon \R \times B_{\lambda_\diamond} \to \mathrm{Mat}_{2n \times 2m}(\C)$, such that for each $\lambda \in B_{\lambda_\diamond}$ we expand
\begin{align*} \X_{in}(x,\lambda) = \X_{\lambda_\diamond,in}(x,\lambda) + \frac{1}{\lambda - \lambda_\diamond} \sum_{j = 1}^{k_{\lambda_\diamond}} \phi_{j,\lambda_\diamond}(x) \int_{-\infty}^\infty \phi_{j,\lambda_\diamond}(z)^{*}\left(\begin{array}{cc} 0 & -I \\ I & 0\end{array}\right)\A_{21,0}(z)dz.\end{align*}
Here, $\X_{\lambda_\diamond,in}(x,\cdot)$ is analytic on $B_{\lambda_\diamond}$ for every $x \in \R$. Moreover, $\X_{\lambda_\diamond,in}(\cdot,\lambda)$ is exponentially localized.
\end{prop}
\begin{proof}
Let $\El_\lambda$ be as defined in Corollary \ref{Fredholmcor}. Since $\partial_v G(u_0,v_{\ho}(x),0)$ is symmetric for all $x \in \R$, $\phi(x,\lambda)$ is a (bounded) solution to (\ref{redstab}) if and only if $\phi(x,\overline{\lambda})^{\mathrm{T}}\left(\begin{smallmatrix} 0 & -I \\ I & 0\end{smallmatrix}\right)$ is a (bounded) solution to (\ref{adjoint}). Moreover, because all solutions $\phi(x,\lambda)$ of (\ref{redstab}) are real for $\lambda \in \R$, it follows $\phi(x,\overline{\lambda})^{\mathrm{T}} = \phi(x,\lambda)^*$ for each $x \in \R$ and $\lambda \in \C$. On the other hand, we have by \cite[Lemma 4.2]{PAL} that $f \in \mathrm{ran}(\El_\lambda)$ if and only if
\begin{align*} \int_{-\infty}^\infty \psi(z)f(z)dz = 0,\end{align*}
for all bounded solutions $\psi(x)$ to (\ref{adjoint}). So, for a bounded solution $\phi(x,\lambda) = (v(x,\lambda),q(x,\lambda))$ to (\ref{redstab}) we have the implication,
\begin{align*} \frac{\partial \El_\lambda}{\partial \lambda}\phi(x,\lambda) = \left(\begin{array}{c}0\\ v(x,\lambda)\end{array}\right) \in \mathrm{ran}(\El_\lambda) \Longrightarrow \int_{-\infty}^\infty \|v(z,\lambda)\|^2 dz = 0\end{align*}
by pairing $\frac{\partial}{\partial \lambda}\El_\lambda \phi(x,\lambda)$ with the solution $\phi(x,\lambda)^* \left(\begin{smallmatrix} 0 & -I \\ I & 0\end{smallmatrix}\right)$ to the adjoint equation (\ref{adjoint}). Hence, $\frac{\partial \El_\lambda}{\partial \lambda} \phi(x,\lambda)$ is contained in $\mathrm{ran}(\El_\lambda)$ if and only if $\phi(x,\lambda)$ is trivial. So, the generalized eigenspace of $\El_{\lambda}$ equals the eigenspace $\ker(\El_{\lambda})$. With this knowledge the result follows from an application of Keldysh formula for $\El_\lambda^{-1}$ given in \cite[Theorem 1.6.5]{MEN}.
\end{proof}
\subsubsection{Slow limit problem} \label{slows}
Consider the slow reduced Evans function $\E_{s,0}$ -- see Definition \ref{slowEv}. In the next result we establish the properties mentioned in Proposition \ref{slowredEv}. In addition, we obtain explicit expressions of the singular part of the Laurent expansion of $\E_{s,0}$ close to a pole using Theorem \ref{Sturminhom}.
\begin{prop}\label{slowEvans}
Let $\Lambda \in (-\Lambda_0,0)$ with $\Lambda_0 > 0$ as in Lemma \ref{lemlambda0}. The slow reduced Evans function $\E_{s,0}$ is well-defined and has the following properties:
\begin{enumerate}
 \item $\E_{s,0}$ is analytic in both $\lambda$ and $\gamma$ on $\Ce_\Lambda \!\setminus\! \N(\E_{f,0}) \times \C$ and $\E_{s,0}(\cdot,\gamma)$ is meromorphic on $\Ce_\Lambda$ for each $\gamma \in \C$.
 \item For every $\gamma \in S^1$ the map $\E_{s,0}(\cdot,\gamma)$ is non-trivial.
 \item Suppose $\lambda_\diamond$ is a simple zero of $\E_{f,0}$. Let $\phi_{\lambda_\diamond} = (\phi_{\lambda_\diamond,1},\phi_{\lambda_\diamond,2}), \psi_{\lambda_\diamond} = (\psi_{\lambda_\diamond,1},\psi_{\lambda_\diamond,2}), B_{\lambda_\diamond}$ and $\X_{\lambda_\diamond,in}$ as in Theorem \ref{Sturminhom}-4. Define for $\lambda \in B_{\lambda_\diamond}$
     \begin{align} \phi &:= \int_{-\infty}^\infty \partial_v H_2(u_0,v_{\ho}(z))\phi_{\lambda_\diamond,1}(z)dz \in \C^m, \nonumber \\
     \psi &:= \int_{-\infty}^\infty \psi_{\lambda_\diamond,2}(z)\partial_u G(u_0,v_{\ho}(z),0)dz \in \mathrm{Mat}_{1 \times m}(\C), \label{slowpole}\\
     \G_a(\lambda) &:= \int_{-\infty}^\infty \left[\partial_u H_2 (u_0,v_{\ho}(z)) + \partial_v H_2(u_0,v_{\ho}(z))\V_{\lambda_\diamond,in}(z,\lambda)\right]dz \in \mathrm{Mat}_{m \times m}(\C)\nonumber,\end{align}
     where $\V_{\lambda_\diamond,in}$ denotes the upper-left $(n \times m)$-block of the analytic $(2n \times 2m)$-matrix $\X_{\lambda_\diamond,in}$. Moreover, let $(u_i(x,\lambda),p_i(x,\lambda))$, $i = 1,\ldots,2m$ be a fundamental set of solutions to the slow limit system (\ref{slowintr}). Finally, let $\Ce(\lambda,\gamma)$ be the cofactor matrix of
     \begin{align} \U_a(\lambda,\gamma) := \left(\begin{array}{cc} I & 0 \\ \G_a(\lambda) & I\end{array}\right)\T_s(2\check{L}_0,0,\lambda) - \gamma I \in \mathrm{Mat}_{2m \times 2m}(\C). \label{Laurentmatrix}\end{align}
     For all $\gamma \in S^1$, the singular part of the Laurent series of $\E_{s,0}(\cdot,\gamma)$ at $\lambda_\diamond$ is given by
     \begin{align*} \frac{1}{\lambda - \lambda_\diamond} \sum_{i = 1}^{2m} \left( \psi u_i(2\check{L}_0,\lambda_\diamond)\right) \left( \phi^\mathrm{T} \left[\Ce_{ji}(\lambda_\diamond,\gamma)\right]_{j = m + 1}^{2m} \right). \end{align*}
     \end{enumerate}
\end{prop}
\begin{proof} First we show that $\E_{s,0}$ is well-defined. Note that $H_2$ vanishes at $v = 0$ by \ref{assS1}. So, Remark \ref{homoclin} yields that $\partial_u H_2 (u_0,v_{\ho}(x))$ is exponentially localized. Combining the latter fact with Theorem \ref{Sturminhom}-3 implies that the integral $\G(\lambda)$ converges and thus $\E_{s,0}$ is well-defined. It is well-known \cite[Lemma 2.1.4]{KAP} that, when the coefficient matrix depends analytically on a parameter, then the evolution is analytic in this parameter too. Combining this with Theorem \ref{Sturminhom}-1,2 yields the first property.
\newpage
\textit{A sketch of the proof of the second property}\\
Proving the second property is quite laborious, but not conceptually difficult. Therefore, we choose to give a sketch of the proof. The full proof can be found in the Appendix \ref{A2}. It is sufficient to establish that, for $\gamma \in S^1$, $\E_{s,0}(\lambda,\gamma)$ is bounded away from $0$ for $\lambda > 0$ is sufficiently large. First, we prove that $\G(\lambda)$ remains bounded as $\lambda \to \infty$. Second, we show that a rescaled version of (\ref{slowintr}) has an exponential dichotomy on $\R$ for $\lambda > 0$ sufficiently large. This enables us to construct suitable bases of the stable and unstable subspaces at $0$, which will form the column vectors of an invertible matrix $\h(\lambda)$. Eventually, we estimate the product $\E_{s,0}(\lambda,\gamma)\det(\h(\lambda))$ by a $\lambda$-independent non-zero determinant for large $\lambda > 0$.\\
\\
\textit{The third property}\\
Assume $\lambda_\diamond$ is a simple zero of $\E_{f,0}$. Theorem \ref{Sturminhom}-4 allows us to split off the singular part of $\G(\lambda)$ at $\lambda_\diamond$. Indeed, we have for $\lambda \in B_{\lambda_\diamond}$
\begin{align*} \G(\lambda) = \frac{1}{\lambda - \lambda_\diamond} \phi \psi + \G_a(\lambda),\end{align*}
with $\phi,\psi$ and $\G_a(\lambda)$ as in (\ref{slowpole}). Using the multi-linearity of the determinant, we expand for $\lambda \in B_{\lambda_\diamond}$ and $\gamma \in S^1$
\begin{align*} \E_{s,0}(\lambda,\gamma) &= \det\left[\U_a(\lambda,\gamma) + \frac{1}{\lambda - \lambda_\diamond} \left(\begin{array}{cc} 0 & 0 \\ \phi\psi & 0\end{array}\right)\T_s(2\check{L}_0,0,\lambda)\right]\\
&= \det(\U_a(\lambda,\gamma)) + \frac{1}{\lambda - \lambda_\diamond} \sum_{i = 1}^{2m} \left( \psi u_i(2\check{L}_0,\lambda_\diamond)\right) \left( \phi^\mathrm{T} \left[\Ce_{ji}(\lambda_\diamond,\gamma)\right]_{j = m + 1}^{2m} \right).
\end{align*}
This concludes the proof of the third property.
\end{proof}
\begin{rem}[\textit{Appearance of $\gamma$ in the singular part}] In the case $m = 1$, Proposition \ref{singm1} shows that $\gamma$ appears as a factor in the singular part of the Laurent expansion of $\E_{s,0}(\cdot,\gamma)$ at a zero $\lambda_\diamond$ of $\E_{f,0}$. Therefore, $\E_{s,0}(\cdot,\gamma)$ has a pole at $\lambda_\diamond$ for some $\gamma \in S^1$ if and only if $\E_{s,0}(\cdot,\gamma)$ has a pole at $\lambda_\diamond$ for all $\gamma \in S^1$. As a consequence, it is sufficient to check condition 2 in Proposition \ref{stability} only for some $\gamma \in S^1$.\\
\\
However, in the general setting of Proposition \ref{slowEvans}-3 the principal part of the Laurent expansion of $\E_{s,0}(\cdot,\gamma)$ is polynomial in $\gamma$. So, it could happen that $\E_{s,0}(\cdot,\gamma)$ has a pole at $\lambda_\diamond$ for all but a discrete set of $\gamma \in S^1$. We expect that such a (non-generic) situation occurs precisely when $\lambda_\diamond$ is a limit point of the slow spectrum $\bigcup_{\gamma \in S^1} \N(\E_{s,0}(\cdot,\gamma))$. In this case, it is again sufficient to check condition 2 in Proposition \ref{stability} only for some $\gamma \in S^1$. $\hfill \blacksquare$
\end{rem}
It is possible to obtain expressions for the singular part of the Laurent series of $\E_{s,0}(\cdot,\gamma)$ at a zero of $\E_{f,0}$ of higher multiplicity. Again, for simplicity of exposition, we treat only the case where $\partial_v G(u_0,v,0)$ is symmetric for all $v \in V$.
\begin{prop}
Suppose $\partial_v G(u_0,v,0)$ is symmetric for all $v \in V$. Let $\lambda_\diamond$ be a zero of $\E_{f,0}$ of multiplicity $k_{\lambda_\diamond}$. Let $\phi_{i,\lambda_\diamond} = (v_{i,\lambda_\diamond},q_{i,\lambda_\diamond}), i = 1,\ldots,k_{\lambda_\diamond}, B_{\lambda_\diamond}$ and $\X_{\lambda_\diamond,in}$ as in Proposition \ref{Sturminhomspecial}. Define for $i = 1,\ldots,k_{\lambda_\diamond}$ and $\lambda \in B_{\lambda_\diamond}$
\begin{align*} \phi_i &:= \int_{-\infty}^\infty \partial_v H_2(u_0,v_{\ho}(z)) v_{i,\lambda_\diamond}(z)dz \in \C^m, \ \ \ \psi_i := \int_{-\infty}^\infty v_{i,\lambda_\diamond}(z)^*\partial_u G(u_0,v_{\ho}(z),0) dz \in \mathrm{Mat}_{1 \times m}(\C),\\
\end{align*}
and
\begin{align*}
\G_a(\lambda) &:= \int_{-\infty}^\infty \left[\partial_u H_2 (u_0,v_{\ho}(z)) + \partial_v H_2(u_0,v_{\ho}(z))\V_{\lambda_\diamond,in}(z,\lambda)\right]dz \in \mathrm{Mat}_{m \times m}(\C),
\end{align*}
where $\V_{\lambda_\diamond,in}$ denotes the upper-left $(n \times m)$-block of the analytic $(2n \times 2m)$-matrix $\X_{\lambda_\diamond,in}$. Let $(u_i(x,\lambda),p_i(x,\lambda))$, $i = 1,\ldots,2m$ be a fundamental set of solutions to the slow limit system (\ref{slowintr}). For each subset $\tau \subset \{1,\ldots,k_{\lambda_\diamond}\}$ and each injective map $\sigma \colon \tau \to \{1,\ldots,2m\}$ denote by $\Ce_{\tau,\sigma}(\lambda,\gamma)$ the matrix obtained by replacing each $\sigma(j)$-th column of the matrix $\U_a(\lambda,\gamma)$ in (\ref{Laurentmatrix}) by $(0,\phi_j)$ for all $j \in \tau$. For all $\gamma \in S^1$, the principle part of the Laurent series of $\E_{s,0}(\cdot,\gamma)$ at $\lambda_\diamond$ is given by
\begin{align*} \sum_{i = 1}^{k_{\lambda_\diamond}} \frac{1}{(\lambda - \lambda_\diamond)^i} \sum_{\begin{smallmatrix}\tau \subset \{1,\ldots,k_{\lambda_\diamond}\}, \\ \#\tau = i\end{smallmatrix}} \sum_{\begin{smallmatrix}\sigma \colon \tau \to \{1,\ldots,2m\} \\ \mathrm{injective}\end{smallmatrix}} \left[\prod_{j = 1}^i \psi_j u_{\sigma(j)}(2\check{L}_0,\lambda_\diamond)\right] \det(\Ce_{\tau,\sigma}(\lambda_\diamond,\gamma)). \end{align*}
\end{prop}
\begin{proof}
This is analogous to the proof of Proposition \ref{slowEvans}-3 using Proposition \ref{Sturminhomspecial} instead.
\end{proof}
\subsection{An exponential dichotomy capturing the fast dynamics} \label{4.3}
As mentioned in Section \ref{Approach}, our goal is to apply the Riccati transformation on the linearized equations (\ref{fullstab}) in order to factorize $\E_\epsilon$ into a fast and a slow part as in \eqref{factorintr}. However, according to Theorem \ref{ric} the Riccati transformation is only legitimate, when system (\ref{faststab}) has an exponential dichotomy on $\R$. We can show that this is the case, whenever $\lambda$ is away from the zero set $\N(\E_{f,0})$. Therefore, we introduce the following notation, before stating the result.
\begin{nota}
Let $\Lambda \in (-\Lambda_0,0)$ with $\Lambda_0 > 0$ as in Lemma \ref{lemlambda0}. For $\delta > 0$ we denote
\begin{align*}\Sigma_{\Lambda,\delta} := \Sigma_{\Lambda,0} \!\setminus\! \bigcup_{\lambda \in \N(\E_{f,0})} B(\lambda,\delta).\end{align*}
\end{nota}
\begin{theo} \label{fastexpdi}
Let $\Lambda \in (-\Lambda_0,0)$ with $\Lambda_0 > 0$ as in Lemma \ref{lemlambda0}. Take $\delta > 0$. For $\epsilon > 0$ sufficiently small, system (\ref{faststab}) has for all $\lambda \in \Sigma_{\Lambda,\delta}$ an exponential dichotomy on $\R$ with constants $K_f,\mu_f > 0$, which are independent of $\epsilon$ and $\lambda$.
\end{theo}
\begin{proof} Our approach is as follows. First, we show that system (\ref{faststab}) has exponential dichotomies on the intervals $[0,2L_\epsilon]$ and $[-2L_\epsilon,0]$. To establish an exponential dichotomy of (\ref{faststab}) on an interval $[a,2L_\epsilon-a]$ for $a > 0$, we regard (\ref{faststab}) as a perturbation of
\begin{align}
\left\{
\begin{array}{rcl} D_2 \partial_x v &=& q \\ \partial_x q &=& \left(\partial_v G (u_{\p,\epsilon}(x),0,\epsilon) + \lambda\right)v \end{array}\right., \ \ \ v,q \in \C^n.\label{redfast}
\end{align}
and apply Proposition \ref{roughnessintervals}. By choosing $a > 0$ independent of $\epsilon$, we can extend the exponential dichotomy to $[0,2L_\epsilon]$ by the extension Lemma \ref{extensionexpdi}. Subsequently, we calculate the minimal opening between the stable and unstable subspaces at $0$ of the two exponential dichotomies of (\ref{faststab}) by comparing system (\ref{faststab}) with the fast homogeneous limit system (\ref{redstab}). We will show that, whenever $\lambda$ is contained in $\Sigma_{\Lambda,\delta}$, this minimal opening is substantial. Therefore, an application of the pasting Lemma \ref{pastingexpdi} gives an exponential dichotomy of (\ref{faststab}) on $[-2L_\epsilon,2L_\epsilon]$. Finally, we apply the periodic extension Lemma \ref{periodextensionexpdi} to yield the result.\\
\\
\textit{An exponential dichotomy for system (\ref{redfast}) on $\R$}\\
We show that the requirements in Proposition \ref{suffcritexpodi} are satisfied. First, by Lemma \ref{lemlambda0} the coefficient matrix $\A_{f,\epsilon}(x,\lambda)$ of system (\ref{redfast}) is hyperbolic on $\R \times \Ce_\Lambda$, with eigenvalues bounded away from the imaginary axis by some constant, which is independent of $\epsilon,x$ and $\lambda$. Second, Remark \ref{S3.4} implies that $\A_{f,\epsilon}$ is uniformly bounded on $\R \times \Sigma_{\Lambda,0}$ and that we have $\frac{\partial }{\partial x}\A_{f,\epsilon}(x,\lambda) = \ord(\sqrt{\epsilon})$ uniformly on $\R \times \Sigma_{\Lambda,0}$. So, by Proposition \ref{suffcritexpodi} system (\ref{redfast}) admits, provided $\epsilon > 0$ is sufficiently small, for $\lambda \in \Sigma_{\Lambda,0}$ an exponential dichotomy on $\R$ with constants $K_{fr},\mu_{fr} > 0$, independent of $\epsilon$ and $\lambda$.\\
\\
\textit{Exponential dichotomies for system (\ref{faststab}) on $[0,2L_\epsilon]$ and on $[-2L_\epsilon,0]$}\\
We want to apply Proposition \ref{roughnessintervals} to system (\ref{redfast}). By assumption \ref{assS3}-3 there exists $K_1,\mu_1 > 0$, independent of $\epsilon$, such that
\begin{align} \|\partial_v G (u_{\p,\epsilon}(x),v_{\p,\epsilon}(x),\epsilon) - \partial_v G (u_{\p,\epsilon}(x),0,\epsilon)\| \leq K_1e^{-\mu_1 \min\{ x,2L_\epsilon - x\}},\label{estimatefast} \end{align}
for $x \in [0,2L_\epsilon]$. Now take $x_0 > 0$, independent of $\epsilon$, such that
\begin{align*}K_1e^{-\mu_1 x_0} < \frac{\mu_{fr}}{36K_{fr}^5}.\end{align*}
Estimate (\ref{estimatefast}) yields by roughness (Proposition \ref{roughnessintervals}) that system (\ref{faststab}) has for every $\lambda \in \Sigma_{\Lambda,0}$ an exponential dichotomy on $[x_0,2L_\epsilon-x_0]$ with constants independent of $\epsilon$ and $\lambda$. By the extension Lemma \ref{extensionexpdi} we can extend the exponential dichotomy of (\ref{faststab}) on $[x_0,2L_\epsilon - x_0]$ to $[0,2L_\epsilon]$ with constants independent of $\epsilon$ and $\lambda$. An analogous treatment provides an exponential dichotomy of (\ref{faststab}) on $[-2L_\epsilon,0]$ with constants independent of $\epsilon$ and $\lambda$. We conclude that (\ref{faststab}) has exponential dichotomies on both $[0,2L_\epsilon]$ and $[-2L_\epsilon,0]$ for every $\lambda \in \Sigma_{\Lambda,0}$ with constants $K_{f0}, \mu_{f0} > 0$, independent of $\epsilon$ and $\lambda$.\\
\\
\textit{Comparison of system (\ref{faststab}) with the homogeneous fast limit system (\ref{redstab})}\\
Take $\lambda \in \Sigma_{\Lambda,0}$. Following Remark \ref{introbetaf} we have for $x \in [-\epsilon^{-\rho},\epsilon^{-\rho}]$
\begin{align} \|\partial_v G (u_{\p,\epsilon}(x),v_{\p,\epsilon}(x),\epsilon) - \partial_v G(u_0,v_{\ho}(x),0)\| = \ord(\epsilon^{\beta_f}). \label{fastest3} \end{align}
Denote by $\T_r(x,y,\lambda)$ and $\T_{f,\epsilon}(x,y,\lambda)$ the evolution operators of (\ref{redstab}) and (\ref{faststab}), respectively. By Remark \ref{S3.4} the coefficient matrix of (\ref{faststab}) can be bounded on $\R$ by some constant $M > 0$, independent of $\epsilon$ and $\lambda$. So, system (\ref{faststab}) has bounded growth with constants $K_3 = 1$ and $\mu_3 = M$. Now choose $\chi = \beta_f/(4M) > 0$. By Lemma \ref{Palmer} and (\ref{fastest3}) we estimate, provided $\epsilon > 0$ is sufficiently small, for all $x,y \in [\chi \log(\epsilon),-\chi\log(\epsilon)]$
\begin{align} \|\T_r(x,y,\lambda) - \T_{f,\epsilon}(x,y,\lambda)\| < 1. \label{fastest4}\end{align}\\
\textit{The minimal opening for system (\ref{redstab})}\\
We recall some facts from Theorem \ref{Sturm} and Corollary \ref{Sturmcor}. First, system (\ref{redstab}) admits for $\lambda \in \Ce_\Lambda$ exponential dichotomies on both half-lines with constants $K_r,\mu_r > 0$, independent of $\lambda$. Second, the corresponding projections $\Proj_{r\pm}(x,\lambda)$ are analytic in $\lambda$. Third, the stable and unstable subspaces $E^s_{r+}(0,\lambda) := \Proj_{r+}(0,\lambda)[\C^{2n}]$ and $E^u_{r-}(0,\lambda) := \ker(\Proj_{r-}(0,\lambda))$ are complementary on $\Sigma_{\Lambda,\delta}$. Therefore, Proposition \ref{opening} implies that the continuous map $\eta_r \colon \Ce_\Lambda \to [0,\infty)$ given by the minimal opening $\eta_r(\lambda) = \eta(E^s_{r+}(0,\lambda),E^u_{r-}(0,\lambda))$ is bounded away from $0$ on $\Sigma_{\Lambda,\delta}$ by some constant $c_{\Lambda,\delta} > 0$.\\
\\
\textit{The minimal opening for system (\ref{faststab})}\\
Take $\lambda \in \Sigma_{\Lambda,0}$. Denote by $\Proj_{f\pm,\epsilon}(x,\lambda)$ the projections corresponding to the exponential dichotomies of (\ref{faststab}) on $[0,2L_\epsilon]$ and on $[-2L_\epsilon,0]$, respectively. By combining estimate (\ref{fastest4}) with Lemma \ref{compdich}, there exists for each $w \in E_{f+,\epsilon}^s(0,\lambda) := \Proj_{f+,\epsilon}(0,\lambda)[\C^{2n}]$ an element $v \in E^s_{r+}(0,\lambda)$ such that
\begin{align} \|v - w\| \leq  (1+K_r\epsilon^{\chi \mu_r})K_{f0}\epsilon^{\chi \mu_{f0}}\|w\|. \label{faststableest}\end{align}
Similarly, there exists for each $w \in E_{f-,\epsilon}^u(0,\lambda) := \ker(\Proj_{f-,\epsilon}(0,\lambda))$ a vector $v \in E^u_{r-}(0,\lambda)$ such that (\ref{faststableest}) holds true. Therefore, Proposition \ref{opening}-3 yields the estimate for $\lambda \in \Sigma_{\Lambda,\delta}$
\begin{align} 0 < c_{\Lambda,\delta} \leq \eta_r(\lambda) \leq \eta(E^s_{f+,\epsilon}(0,\lambda),E^u_{f-,\epsilon}(0,\lambda)) + 4(1+K_r\epsilon^{\chi \mu_r})K_{f0}\epsilon^{\chi \mu_{f0}}. \label{openingest}\end{align}\\
\textit{Application of the Pasting Lemma}\\
By estimate (\ref{openingest}) and Proposition \ref{opening}-2 one deduces that, for $\epsilon > 0$ sufficiently small, $E^s_{f+,\epsilon}(0,\lambda)$ and $E^u_{f-,\epsilon}(0,\lambda)$ are complementary on $\Sigma_{\Lambda,\delta}$. So, the projection $\Proj_{f,\epsilon}(\lambda)$ onto $E_{f+,\epsilon}^s(0,\lambda)$ along $E_{f-,\epsilon}^u(0,\lambda)$ is well-defined on $\Sigma_{\Lambda,\delta}$. Moreover, by Proposition \ref{opening}-1 and (\ref{openingest}) we obtain, provided $\epsilon > 0$ is sufficiently small, the bound,
\begin{align*} \|\Proj_{f,\epsilon}(\lambda)\| \leq \frac{1}{\eta(E^s_{f+,\epsilon}(0,\lambda),E^u_{f-,\epsilon}(0,\lambda))} \leq \frac{2}{c_{\Lambda,\delta}},\end{align*}
for $\lambda \in \Sigma_{\Lambda,\delta}$. Now, by the Pasting Lemma \ref{pastingexpdi} equation (\ref{faststab}) has an exponential dichotomy on $[-2L_\epsilon,2L_\epsilon]$ with constants depending only on $c_{\Lambda,\delta},\mu_{f0}$ and $K_{f0}$. Subsequently, by the periodic extension Lemma \ref{periodextensionexpdi} system (\ref{faststab}) admits an exponential dichotomy on $\R$ with constants $K_f,\mu_f > 0$ depending only on $c_{\Lambda,\delta},\mu_{f0},K_{f0}$ and $M$. Recall that $M > 0$ is a constant, independent of $\epsilon$ and $\lambda$, which bounds the coefficient matrix of (\ref{faststab}) on $\R \times \Sigma_{\Lambda,0}$. We conclude that $\mu_f$ and $K_f$ are independent of $\epsilon$ and $\lambda$.
\end{proof}
For later purposes we treat the following corollary of Theorem \ref{fastexpdi}.
\begin{cor} \label{corinhom}
Let $\Lambda \in (-\Lambda_0,0)$ with $\Lambda_0 > 0$ as in Lemma \ref{lemlambda0}. Take $\delta > 0$. For $\lambda \in \Sigma_{\Lambda,\delta}$, there exists a unique bounded solution $\Psi_{in,\epsilon}(x,\lambda)$ to the inhomogeneous matrix equation,
\begin{align} \partial_x \Psi = \A_{22,\epsilon}(x,\lambda)\Psi + \A_{21,\epsilon}(x), \ \ \ \Psi \in \mathrm{Mat}_{2n \times 2m}(\C), \label{fastinhom2}\end{align}
satisfying for all $x \in [-L_\epsilon,L_\epsilon]$
\begin{align*} \|\Psi_{in,\epsilon}(x,\lambda) - \X_{in}(x,\lambda)\| = \ord(\epsilon^{\beta_f}),\end{align*}
uniformly in $\lambda \in \Sigma_{\Lambda,\delta}$. Here, $\beta_f > 0$ is as in Remark \ref{introbetaf} and $\X_{in}(x,\lambda)$ is the unique solution to (\ref{fastinhom}) established in Theorem \ref{Sturminhom}.
\end{cor}
\begin{proof}
Let $\lambda \in \Sigma_{\Lambda,\delta}$. First, system (\ref{faststab}) has by Theorem \ref{fastexpdi} an exponential dichotomy on $\R$ with constants $K_f,\mu_f > 0$, independent of $\epsilon$ and $\lambda$. Second, by Remark \ref{introbetaf} we have for $x \in [-\epsilon^{-\rho},\epsilon^{-\rho}]$
\begin{align*} \|\partial_v G (u_{\p,\epsilon}(x),v_{\p,\epsilon}(x),\epsilon) - \partial_v G(u_0,v_{\ho}(x),0)\|, \|\partial_u G (u_{\p,\epsilon}(x),v_{\p,\epsilon}(x),\epsilon) - \partial_u G(u_0,v_{\ho}(x),0)\| = \ord(\epsilon^{\beta_f}).\end{align*}
Now we apply Proposition \ref{inhomexpdi} to the inhomogeneous equations (\ref{fastinhom}) and (\ref{fastinhom2}): there exists a unique bounded solution $\Psi_{in,\epsilon}(x,\lambda)$ of equation (\ref{fastinhom2}) satisfying for all $x \in [-\tfrac{1}{2}\epsilon^{-\rho},\tfrac{1}{2}\epsilon^{-\rho}]$
\begin{align}
\|\Psi_{in,\epsilon}(x,\lambda) - \X_{in}(x,\lambda)\| = \ord(\epsilon^{\beta_f}), \label{inhomest3}
\end{align}
uniformly in $\lambda \in \Sigma_{\Lambda,\delta}$. Here, we have used $K_f$ and $\mu_f$ do not depend on $\epsilon$ and $\lambda$ and $\X_{in}(\cdot,\lambda)$ is bounded on $\R$ uniformly in $\lambda \in \Sigma_{\Lambda,\delta}$ by Theorem \ref{Sturminhom}-3. Note that $G$ vanishes at $v = 0$ by \ref{assS1}. Hence, by \ref{assS3}-3 there exists $K_1,\mu_1 > 0$, independent of $\epsilon$, such that for $x \in [-L_\epsilon,L_\epsilon]$ it holds
\begin{align} \|\partial_u G(u_{\p,\epsilon}(x),v_{\p,\epsilon}(x),\epsilon)\| \leq K_1e^{-\mu_1|x|}. \label{inhomest5}\end{align}
Combing estimate (\ref{inhomest5}) with Proposition \ref{inhomexpdi} implies that there exists $K_2,\mu_2 > 0$, independent of $\epsilon$ and $\lambda$, such that \begin{align} \|\Psi_{in,\epsilon}(x,\lambda)\| \leq K_2e^{-\mu_2|x|}, \label{inhomest4}\end{align}
for all $x \in [-L_\epsilon,L_\epsilon]$. Estimate (\ref{inhomest4}) together with the estimate in Theorem \ref{Sturminhom}-3 show that (\ref{inhomest3}) actually holds for all $x \in [-L_\epsilon,L_\epsilon]$ uniformly in $\lambda \in \Sigma_{\Lambda,\delta}$.
\end{proof}
\subsection{Factorization of the Evans function via the Riccati transform}\label{4.4}
As mentioned in Section \ref{Approach}, an application of the Riccati transform enables us to reduce our linear stability problem (\ref{fullstab}) into two subproblems. This yields the factorization of the Evans function $\E_\epsilon$ into two factors $\E_{s,\epsilon}$ and $\E_{f,\epsilon}$, which can be linked to $\E_{f,0}$ and $\E_{s,0}$, respectively. The Riccati transformation is employed in the following theorem.
\begin{theo} \label{ricfac}
Let $\Lambda \in (-\Lambda_0,0)$ with $\Lambda_0 > 0$ as in Lemma \ref{lemlambda0}. Take $\delta > 0$. For $\epsilon > 0$ sufficiently small, there exists a function $U_\epsilon \colon \R \times \Sigma_{\Lambda,\delta} \to \mathrm{Mat}_{2m \times 2n}(\C)$ such that for all $\lambda \in \Sigma_{\Lambda,\delta}$ and $\gamma \in \C$ we have the factorization,
\begin{align*}
\E_\epsilon(\lambda,\gamma) = \E_{s,\epsilon}(\lambda,\gamma)\E_{f,\epsilon}(\lambda,\gamma),
\end{align*}
with $\E_{s,\epsilon}, \E_{f,\epsilon} \colon \Sigma_{\Lambda,\delta} \times \C \to \C$ given by
\begin{align*}
\E_{s,\epsilon}(\lambda,\gamma) &:= \det(\T_{sd,\epsilon}(0,-L_\epsilon,\lambda) - \gamma \T_{sd,\epsilon}(0,L_\epsilon,\lambda)),\\
\E_{f,\epsilon}(\lambda,\gamma) &:= \det(\T_{fd,\epsilon}(0,-L_\epsilon,\lambda) - \gamma \T_{fd,\epsilon}(0,L_\epsilon,\lambda)).
\end{align*}
Here, $\T_{sd,\epsilon}(x,y,\lambda)$ is the evolution of system,
\begin{align} \partial_x \chi = \sqrt{\epsilon}(\A_{11,\epsilon}(x,\lambda) + \A_{12,\epsilon}(x)U_\epsilon(x,\lambda))\chi,  \ \ \ \chi \in \C^{2m}, \label{slowdiag}\end{align}
and $\T_{fd,\epsilon}(x,y,\lambda)$ is the evolution of system,
\begin{align}
\partial_x \omega = (\A_{22,\epsilon}(x,\lambda) - \sqrt{\epsilon} U_\epsilon(x,\lambda)\A_{12,\epsilon}(x))\omega, \ \ \ \omega \in \C^{2n}. \label{fastdiag}
\end{align}
Let $\beta_f,\rho > 0$ as in Remark \ref{introbetaf}. $U_\epsilon$ has the following properties:
\begin{enumerate}
 \item $U_\epsilon(\cdot,\lambda)$ is bounded uniformly in $\epsilon$ and $\lambda$ on $\R$;
 \item $U_\epsilon(\cdot,\lambda)$ is $2L_\epsilon$-periodic for each $\lambda \in \Sigma_{\Lambda,\delta}$;
 \item For all $x \in [-L_\epsilon,L_\epsilon]$ it holds uniformly in $\lambda \in \Sigma_{\Lambda,\delta}$
 \begin{align*} \|U_\epsilon(x,\lambda) - \X_{in}(x,\lambda)\| = \ord(\epsilon^{\beta_f}).\end{align*}
 Here, $\X_{in}(x,\lambda)$ is the unique solution to (\ref{fastinhom}) obtained in Theorem \ref{Sturminhom};
 \item For all $x \in [-L_\epsilon,L_\epsilon] \!\setminus\! I_\epsilon$ with $I_\epsilon := [-\epsilon^{-\rho},\epsilon^{-\rho}]$ we have $\|U_\epsilon(x,\lambda)\| = \ord(\epsilon^3)$, uniformly in $\lambda \in \Sigma_{\Lambda,\delta}$.
\end{enumerate}
\end{theo}
\begin{proof} Fix $\lambda \in \Sigma_{\Lambda,\delta}$. System (\ref{fullstab}) is clearly of the form (\ref{ric1}) with coefficient matrices that are uniformly bounded in $\epsilon$ on $\R$ by Remark \ref{S3.4}. Furthermore, by Theorem \ref{fastexpdi} system (\ref{faststab}) has an exponential dichotomy on $\R$ with constants $K_f,\mu_f > 0$, independent of $\epsilon$ and $\lambda$. Hence, we can apply the Riccati transform from Theorem \ref{ric} to (\ref{fullstab}). Let $H_\epsilon(x,\lambda) \in \mathrm{Mat}_{2(m+n) \times 2(m+n)}(\C)$ and $U_\epsilon(x,\lambda) \in \mathrm{Mat}_{2m \times 2n}(\C)$ be as in Theorem \ref{ric}. By Theorem \ref{ric}-2 the change of variables $\phi(x) = H_\epsilon(x,\lambda)\psi(x)$ transforms (\ref{fullstab}) into the diagonal system,
\begin{align} \partial_x \psi = \left(\begin{array}{cc}  \sqrt{\epsilon}(\A_{11,\epsilon}(x,\lambda) + \A_{12,\epsilon}(x)U_\epsilon(x,\lambda)) & 0 \\ 0 & \A_{22,\epsilon}(x,\lambda) - \sqrt{\epsilon} U_\epsilon(x,\lambda)\A_{12,\epsilon}(x)\end{array}\right)\psi, \label{diagonal}\end{align}
with $\psi \in \C^{2(m+n)}$. The evolution $\T_{d,\epsilon}(x,y,\lambda)$ of system (\ref{diagonal}) is a block diagonal matrix with consecutively $\T_{sd,\epsilon}(x,y,\lambda)$ and $\T_{fd,\epsilon}(x,y,\lambda)$
on the diagonal. Furthermore, $H_\epsilon(\cdot,\lambda)$ and $U_\epsilon(\cdot,\lambda)$ are $2L_\epsilon$-periodic by Theorem \ref{ric}-5. Finally, as a product of two triangular matrices with only ones on the diagonal, the determinant of $H_\epsilon(x,\lambda)$ equals $1$ for every $x \in \R$. This enables us to factorize the Evans functions $\E_\epsilon$ as follows
\begin{align*} \E_\epsilon(\lambda,\gamma) &= \det\left(H_\epsilon(0,\lambda)\left[\T_{d,\epsilon}(0,-L_\epsilon,\lambda) - \gamma \T_{d,\epsilon}(0,L_\epsilon,\lambda)\right]H_\epsilon(L_\epsilon,\lambda)^{-1}\right) = \E_{s,\epsilon}(\lambda,\gamma)\E_{f,\epsilon}(\lambda,\gamma).
\end{align*}\\
\textit{Properties of $U_\epsilon$}\\
The first two properties are immediate by Theorem \ref{ric}-1 and Theorem \ref{ric}-5. Let $\Psi_{in,\epsilon}(x,\lambda)$ be the unique solution to (\ref{fastinhom2}). By Theorem \ref{ric}-3 and Corollary \ref{corinhom} it holds for $x \in [-L_\epsilon,L_\epsilon]$
\begin{align*}\begin{split}\|U_\epsilon(x,\lambda) - \X_{in}(x,\lambda)\| &\leq \|U_\epsilon(x,\lambda) - \Psi_{in,\epsilon}(x,\lambda)\| + \|\Psi_{in,\epsilon}(x,\lambda) - \X_{in}(x,\lambda)\| \\
&= \ord(\epsilon^{\min\{1/4,\beta_f\}}),\end{split}\end{align*}
uniformly in $\lambda \in \Sigma_{\Lambda,\delta}$. Here, we have used that $K_f$ and $\mu_f$ are independent of $\epsilon$ and $\lambda$. This settles the third property, since we may without loss of generality assume $\beta_f \leq 1/4$. For the fourth property we use the method of successive approximation. Note that $G$ vanishes at $v = 0$ by \ref{assS1}. Hence, by \ref{assS3}-3 there exists $K_1,\mu_1 > 0$, independent of $\epsilon$, such that $\|\A_{21,\epsilon}(x)\| \leq K_1e^{-\mu_1 |x|}$ for $x \in [-L_\epsilon,L_\epsilon]$. Using the latter two lines, we approximate $U_\epsilon(x,\lambda)$ successively for three times with Theorem \ref{ric}-4. This yields $\|U_\epsilon(x,\lambda)\| = \ord(\epsilon^3)$ for $x \in [-L_\epsilon,L_\epsilon] \!\setminus\! I_\epsilon$ uniformly in $\lambda \in \Sigma_{\Lambda,\delta}$. Here, we have used again that $K_f$ and $\mu_f$ are independent of $\epsilon$ and $\lambda$.
\end{proof}
\subsection{The factors of the Evans function and the three singular limit problems} \label{4.rela}
We derive from Theorem \ref{ricfac} that our linear stability problem (\ref{fullstab}) diagonalizes into two subproblems (\ref{slowdiag}) and (\ref{fastdiag}). This diagonalization yields the splitting of the Evans function $\E_\epsilon$ into two factors $\E_{s,\epsilon}$ and $\E_{f,\epsilon}$. By relating (\ref{slowdiag}) to the slow limit problem (\ref{slowintr}) and (\ref{fastdiag}) to the homogeneous fast limit problem (\ref{redstab}), we can link $\E_{s,\epsilon}$ to the slow reduced Evans function $\E_{s,0}$ and $\E_{f,\epsilon}$ to the fast reduced Evans function $\E_{f,0}$.\\
\\
We start with the relation between (\ref{slowintr}) and (\ref{slowdiag}). Note that the transformation matrix $U_\epsilon$ is related to the fast inhomogeneous limit problem (\ref{fastinhom}) by Theorem \ref{ricfac}-3. This is how problem (\ref{fastinhom}) merges into the slow reduced Evans function $\E_{s,0}$ via system (\ref{slowdiag}). This is reflected in the proof of the following result, which links $\E_{s,\epsilon}$ to the slow reduced Evans function $\E_{s,0}$.
\begin{lem} \label{Evansslowestimate}
Let $\Lambda \in (-\Lambda_0,0)$ with $\Lambda_0 > 0$ as in Lemma \ref{lemlambda0}. Take $\delta > 0$. Let $\E_{s,\epsilon}$ be as in Theorem \ref{ricfac}. For $\epsilon > 0$ sufficiently small, there exists $\mu_s > 0$, independent of $\epsilon$ and $\lambda$, such that we have the following estimate,
\begin{align*} \E_{s,\epsilon}(\lambda,\gamma) = \E_{s,0}(\lambda,\gamma) + \ord(\epsilon^{\mu_s}),\end{align*}
uniformly in $\lambda \in \Sigma_{\Lambda,\delta}$ and $\gamma \in S^1$.
\end{lem}
\begin{proof} Let $\lambda \in \Sigma_{\Lambda,\delta}$ and $\gamma \in S^1$. Our approach is as follows. We split the coefficient matrix in system (\ref{slowdiag}) corresponding to their decay behavior outside the pulse region, i.e. we write
\begin{align}
\A_{11,\epsilon}(x,\lambda) + \A_{12,\epsilon}(x)U_\epsilon(x,\lambda) = \B_{1,\epsilon}(x,\lambda) + \B_{2,\epsilon}(x,\lambda), \label{slowest3}
\end{align}
with
\begin{align*}
\B_{1,\epsilon}(x,\lambda) &:=  \left(\begin{array}{cc} 0 & D_1^{-1} \\ \epsilon \left(\partial_u H_1 (u_{\p,\epsilon}(x),0,\epsilon) + \lambda\right) & 0\end{array}\right),\\
\B_{2,\epsilon}(x,\lambda) &:=  \left(\begin{array}{cc} 0 & 0 \\ B_{2,\epsilon}(x) & 0\end{array}\right) + \A_{12,\epsilon}(x)U_\epsilon(x,\lambda),\\
B_{2,\epsilon}(x) &:= \partial_u H_2 (u_{\p,\epsilon}(x),v_{\p,\epsilon}(x)) + \epsilon\left(\partial_u H_1 (u_{\p,\epsilon}(x),v_{\p,\epsilon}(x),\epsilon) - \partial_u H_1 (u_{\p,\epsilon}(x),0,\epsilon)\right).
\end{align*}
Note that $\B_{1,\epsilon}(\cdot,\lambda)$ and $\B_{2,\epsilon}(\cdot,\lambda)$ are bounded on $\R$ uniformly in $\epsilon$ and $\lambda \in \Sigma_{\Lambda,\delta}$ by Theorem \ref{ricfac}-1 and Remark \ref{S3.4}. The splitting gives rise to an intermediate system,
\begin{align}
\partial_x \chi = \sqrt{\epsilon} \B_{1,\epsilon}(x,\lambda)\chi,  \ \ \ \chi \in \C^{2m}. \label{redslow}\end{align}
The intermediate system (\ref{redslow}) helps us to draw the connection between system (\ref{slowdiag}) and the slow limit system (\ref{slowintr}). On the one hand, we will show that (\ref{slowdiag}) and (\ref{redslow}) are closely related by using variation of constants. On the other hand, by performing a suitable coordinate change the evolutions of systems (\ref{slowintr}) and (\ref{redslow}) are close to each other. This enables us to approximate $\E_{s,\epsilon}$ with $\E_{s,0}$ on $\Sigma_{\Lambda,\delta}$.\\
\\
\textit{Some estimates}\\
Denote by $\T_{is,\epsilon}(x,y,\lambda)$ the evolution of system (\ref{redslow}). Let $\rho > 0$ be as in Remark \ref{introbetaf}. By Lemma \ref{Palmer} we have for each $x, y \in I_\epsilon = [-\epsilon^{-\rho},\epsilon^{-\rho}]$
\begin{align}\|\T_{sd,\epsilon}(x,y,\lambda) - I\|, \|\T_{is,\epsilon}(x,y,\lambda) - I\| = \ord(\epsilon^{1/2-\rho}), \label{slowest1}\end{align}
uniformly in $\lambda \in \Sigma_{\Lambda,\delta}$. Next, we will pay attention to the term $\B_{2,\epsilon}$ in the expansion (\ref{slowest3}). First, note that $H_2$ vanishes at $v = 0$ by \ref{assS1}. Hence, assumption \ref{assS3}-3 implies that there exists $K_1,\mu_1 > 0$, independent of $\epsilon$, such that $\|B_{2,\epsilon}(x)\| \leq K_1e^{-\mu_1 |x|}$ for $x \in [-L_\epsilon,L_\epsilon]$. Putting the latter and Theorem \ref{ricfac}-4 together, we approximate for $x \in [-L_\epsilon,L_\epsilon] \!\setminus\! I_\epsilon$
\begin{align}
\|\B_{2,\epsilon}(x,\lambda)\| = \ord(\epsilon^3), \label{slowest7}
\end{align}
uniformly in $\lambda \in \Sigma_{\Lambda,\delta}$. On the other hand, one readily establishes that (\ref{redslow}) has bounded growth on $[-L_\epsilon,L_\epsilon]$ with constants $K_1\epsilon^{-1/2},\epsilon\mu_1 > 0$, where $K_1,\mu_1 > 0$ are independent of $\lambda$ and $\epsilon$. Therefore, another application of Lemma \ref{Palmer} gives for $x,y \in [0,L_\epsilon] \!\setminus\! I_\epsilon$
\begin{align}\begin{split} \|\T_{sd,\epsilon}(x,y,\lambda) - \T_{is,\epsilon}(x,y,\lambda)\| &\leq K_1\epsilon^{-1/2}L_\epsilon \sup_{x \in [0,L_\epsilon] \!\setminus\! I_\epsilon} \sqrt{\epsilon}\|\B_{2,\epsilon}(x,\lambda)\|e^{\mu_1\epsilon L_\epsilon} =  \ord(\epsilon^2), \end{split}\label{slowest2}\end{align}
uniformly in $\lambda \in \Sigma_{\Lambda,\delta}$. \\
\\
\textit{Connection between (\ref{slowdiag}) and (\ref{redslow}): a variation of constants approach}\\
The variation of constants formula gives
\begin{align*} \T_{sd,\epsilon}(0,L_\epsilon,\lambda) = \T_{is,\epsilon}(0,L_\epsilon,\lambda) - \sqrt{\epsilon} \int_0^{L_\epsilon} \T_{is,\epsilon}(0,z,\lambda)\B_{2,\epsilon}(z,\lambda)\T_{sd,\epsilon}(z,L_\epsilon,\lambda)dz.\end{align*}
Using the bounded growth of (\ref{redslow}), we have by estimates (\ref{slowest7}) and (\ref{slowest2})
\begin{align} \begin{split}\T_{sd,\epsilon}(0,L_\epsilon,\lambda) &= \T_{is,\epsilon}(0,L_\epsilon,\lambda) - \sqrt{\epsilon} \int_0^{\epsilon^{-\rho}} \T_{is,\epsilon}(0,z,\lambda)\B_{2,\epsilon}(z,\lambda)\T_{sd,\epsilon}(z,L_\epsilon,\lambda)dz + \ord(\epsilon) \\ &= \Ef_{+,\epsilon}(\lambda)\T_{is,\epsilon}(0,L_\epsilon,\lambda) + \ord(\epsilon),\end{split} \label{slowest10}
\end{align}
uniformly in $\lambda \in \Sigma_{\Lambda,\delta}$, with
\begin{align*} \Ef_{+,\epsilon}(\lambda) &:= I - \sqrt{\epsilon} \int_0^{\epsilon^{-\rho}}\T_{is,\epsilon}(0,z,\lambda) \B_{2,\epsilon}(z,\lambda)\T_{sd,\epsilon}(z,\epsilon^{-\rho},\lambda)dz\T_{is,\epsilon}(\epsilon^{-\rho},0,\lambda).
\end{align*}
Using (\ref{slowest1}), we derive
\begin{align} \Ef_{+,\epsilon}(\lambda) &=  I - \sqrt{\epsilon} \int_0^{\epsilon^{-\rho}} \B_{2,\epsilon}(z,\lambda)dz + \ord(\epsilon^{1 - 2\rho}) \label{slowest5}
\end{align}
uniformly in $\lambda \in \Sigma_{\Lambda,\delta}$. Write the $(2n\times 2m)$-matrix $\X_{in}(x,\lambda)$ as a composition of four block matrices, where $\V_{in}(x,\lambda)$ is the upper-left $n \times m$-block. Using Remark \ref{introbetaf}, Theorem \ref{Sturminhom}-3 and Theorem \ref{ricfac}-3, we approximate (\ref{slowest5}) as
\begin{align} \Ef_{+,\epsilon}(\lambda) &= I + \sqrt{\epsilon} \int_0^{\epsilon^{-\rho}} \left[\left(\begin{array}{cc} 0 & 0 \\ \partial_u H_2 (u_0,v_{\ho}(x)) & 0\end{array}\right) + \right.\nonumber \\
& \ \ \ \ \ \ \ \ \ \ \ \ \ \ \ \ \ \ \ \ \ \ \ \ \ \ \ \ \ \ \ \ \ \ \ \left.\left(\begin{array}{cc} 0 & 0 \\ \partial_v H_2 (u_0,v_{\ho}(x)) & 0\end{array}\right)\left(\begin{array}{cc} \V_{in}(x,\lambda) & 0 \\ * & 0\end{array}\right)\right]dx + \ord(\epsilon^{1/2+\mu_s}) \label{slowest9} \\
&=  \left(\begin{array}{cc} I & 0 \\ -\sqrt{\epsilon} \int_0^\infty \left[\partial_u H_2 (u_0,v_{\ho}(x)) + \partial_v H_2 (u_0,v_{\ho}(x)) \V_{in}(x,\lambda)\right]dx & I\end{array}\right) + \ord(\epsilon^{1/2+\mu_s}), \nonumber \end{align}
uniformly in $\lambda \in \Sigma_{\Lambda,\delta}$, where $\mu_s = \beta_f-\rho$. Without loss of generality we may assume $0 < \rho < \beta_f$ so that $\mu_s > 0$. \\
\\
\textit{Connection between system (\ref{slowintr}) and (\ref{redslow})}\\
We apply two operations on system (\ref{redslow}). First, we perform the coordinate change $\chi = C_\epsilon \tilde{\chi}$, where $C_\epsilon := \left(\begin{smallmatrix} I & 0 \\ 0 & \sqrt{\epsilon} \end{smallmatrix}\right) \in \mathrm{Mat}_{2m \times 2m}(\C)$. Second, we switch to the large spatial scale $\xx = \epsilon x$. The evolution of the resulting system from $0$ to $\check{L}_\epsilon$ can be approximated by $\T_s(0,\check{L}_0,\lambda)$ by combining Lemma \ref{Palmer} with \ref{assS3}-1,3. Thus, we obtain the following estimate,
\begin{align} \|C_\epsilon^{-1} \T_{is,\epsilon}(0,L_\epsilon,\lambda)C_\epsilon - \T_s(0,\check{L}_0,\lambda)\| = \ord(\sqrt{\epsilon}), \label{slowest8}\end{align}
uniformly in $\lambda \in \Sigma_{\Lambda,\delta}$. \\
\\
\textit{Comparing $\E_{s,\epsilon}$ to the slow reduced Evans function $\E_{s,0}$}\\
Plugging (\ref{slowest9}) and (\ref{slowest8}) into (\ref{slowest10}) yields
\begin{align}
\T_{sd,\epsilon}(0,L_\epsilon,\lambda) = C_\epsilon\left[\Upsilon_+(\lambda)\T_{s}(0,\check{L}_\epsilon,\lambda) + \ord(\epsilon^{\mu_s})\right] C_\epsilon^{-1},\label{slowest11}
\end{align}
uniformly in $\lambda \in \Sigma_{\Lambda,\delta}$, with
\begin{align*}
\Upsilon_+(\lambda) &:= \left(\begin{array}{cc} I & 0 \\ -\int_0^\infty \left[\partial_u H_2 (u_0,v_{\ho}(x)) + \partial_v H_2 (u_0,v_{\ho}(x)) \V_{in}(x,\lambda)\right]dx & I\end{array}\right).
\end{align*}
Similarly, we derive
\begin{align}
\T_{sd,\epsilon}(0,-L_\epsilon,\lambda) = C_\epsilon\left[\Upsilon_-(\lambda)\T_{s}(0,-\check{L}_\epsilon,\lambda) + \ord(\epsilon^{\mu_s})\right] C_\epsilon^{-1},\label{slowest12}
\end{align}
uniformly in $\lambda \in \Sigma_{\Lambda,\delta}$, with
\begin{align*}
\Upsilon_-(\lambda) &:= \left(\begin{array}{cc} I & 0 \\ \int_{-\infty}^0 \left[\partial_u H_2 (u_0,v_{\ho}(x)) + \partial_v H_2 (u_0,v_{\ho}(x)) \V_{in}(x,\lambda)\right]dx & I\end{array}\right).
\end{align*}
First, by Liouville's Theorem we have $\det(\T_{s}(0,\check{L}_0,\lambda)) = 1$. Second, the identity $\Upsilon_+(\lambda)^{-1}\Upsilon_-(\lambda) = \Upsilon(\lambda)$ holds true, where $\Upsilon(\lambda)$ is defined in (\ref{defupsilon}). Putting these two items together, we estimate $\E_{s,\epsilon}$ using (\ref{slowest11}) and (\ref{slowest12})
\begin{align*}
\begin{split}\E_{s,\epsilon}(\lambda,\gamma) &= \det(\T_{sd,\epsilon}(0,-L_\epsilon,\lambda) - \gamma \T_{sd,\epsilon}(0,L_\epsilon,\lambda))\\
&= \det\left(\Upsilon_-(\lambda)\T_{s}(2\check{L}_0,\check{L}_0,\lambda) - \gamma \Upsilon_+(\lambda)\T_{s}(0,\check{L}_0,\lambda)\right) + \ord(\epsilon^{\mu_s})\\
&= \E_{s,0}(\lambda,\gamma) + \ord(\epsilon^{\mu_s}),\end{split}
\end{align*}
uniformly in $\lambda \in \Sigma_{\Lambda,\delta}$ and $\gamma \in S^1$. This estimate concludes the proof. \end{proof}
It remains to link $\E_{f,\epsilon}$ to the fast reduced Evans function $\E_{f,0}$.
\begin{lem} \label{Evansfastestimate}
Let $\Lambda \in (-\Lambda_0,0)$ with $\Lambda_0 > 0$ as in Lemma \ref{lemlambda0}. Take $\delta > 0$. There exists $\mu_p > 0$ such that, for $\epsilon > 0$ sufficiently small, there is a map $h_{\epsilon} \colon \Sigma_{\Lambda,\delta} \to \C$ satisfying
\begin{align*} 0 < |h_{\epsilon}(\lambda)| &= \ord(e^{-\mu_p L_\epsilon}),\\
 \E_{f,\epsilon}(\lambda,\gamma)h_{\epsilon}(\lambda) &= (-\gamma)^n\E_{f,0}(\lambda) + \ord(\epsilon^{\mu_p}),\end{align*}
uniformly in $\lambda \in \Sigma_{\Lambda,\delta}$ and $\gamma \in S^1$.
\end{lem}
\begin{proof} Our approach is as follows. First, we show that system (\ref{fastdiag}) has an exponential dichotomy on $\R$, if $\lambda$ is in $\Sigma_{\Lambda,\delta}$. Recall that, the fast reduced Evans function $\E_{f,0}$ is defined in terms of bases $B_r^{u,s}(\lambda)$ of the stable and unstable subspaces at $0$ of the homogeneous fast limit system (\ref{redstab}). By comparing system (\ref{fastdiag}) to (\ref{redstab}), we are able to construct bases $B_\epsilon^{u,s}(\lambda)$ of the (un)stable subspaces at $0$ of (\ref{fastdiag}), which are close to $B_r^{u,s}(\lambda)$. By tracking the bases $B_\epsilon^{u,s}(\lambda)$ either forward or backward, we construct bases of the (un)stable subspaces at $\pm L_\epsilon$ of (\ref{fastdiag}). These bases at $\pm L_\epsilon$ will form the column vectors of a matrix $\h_\epsilon(\lambda)$, which connects $\E_{f,\epsilon}$ to $\E_{f,0}$.\\
\\
\textit{An exponential dichotomy on $\R$ of system (\ref{fastdiag})}\\
Let $\lambda \in \Sigma_{\Lambda,\delta}$ and $\gamma \in S^1$. First, system (\ref{faststab}) has by Theorem \ref{fastexpdi} an exponential dichotomy on $\R$ with constants $K_f,\mu_f > 0$, independent of $\epsilon$ and $\lambda$. Second, $U_\epsilon(x,\lambda)$ is bounded on $\R$ uniformly in $\epsilon$ and $\lambda \in \Sigma_{\Lambda,\delta}$ by Theorem \ref{ricfac}-1. Therefore, Proposition \ref{RoughnessR} yields that (\ref{fastdiag}) has an exponential dichotomy on $\R$ with constants $K_{fd},\mu_{fd} > 0$, independent of $\epsilon$ and $\lambda$. Denote the corresponding projections by $\Proj_{fd,\epsilon}(x,\lambda), x \in \R$. Since the coefficient matrix of (\ref{fastdiag}) is $2L_\epsilon$-periodic by Theorem \ref{ricfac}-2, the projections $\Proj_{fd,\epsilon}(\cdot,\lambda)$ are also $2L_\epsilon$-periodic by \cite[Proposition 8.4]{COP}.\\
\\
\textit{Comparing system (\ref{fastdiag}) to the homogeneous fast limit system (\ref{redstab})}\\
Let $M > 0$ be a bound of the coefficient matrix $\A_{22,0}(x,\lambda)$ of equation (\ref{redstab}) on $\R \times \Sigma_{\Lambda,\delta}$. By Remark \ref{introbetaf} and Theorem \ref{ricfac}-1 we have for $x \in I_\epsilon = [-\epsilon^{-\rho},\epsilon^{-\rho}]$
\begin{align} \|\A_{22,\epsilon}(x,\lambda) - \sqrt{\epsilon} U_\epsilon(x,\lambda)\A_{12,\epsilon}(x) - \A_{22,0}(x,\lambda)\| = \ord(\epsilon^{\mu_0}).\label{fullestimate1}\end{align}
with $\mu_0 := \min\{\beta_f,1/2\}$. Denote by $\T_r(x,y,\lambda)$ the evolution of (\ref{redstab}) and take $\chi := \mu_0/(4M) > 0$. By Lemma \ref{Palmer} and estimate (\ref{fullestimate1}) it follows, provided that $\epsilon > 0$ is sufficiently small, for all $x,y \in [\chi \log(\epsilon),-\chi\log(\epsilon)]$
\begin{align} \|\T_r(x,y,\lambda) - \T_{fd,\epsilon}(x,y,\lambda)\| < 1. \label{fullestimate2}\end{align}
By Corollary \ref{Sturmcor} system (\ref{redstab}) has for $\lambda \in \Sigma_{\Lambda,0}$ exponential dichotomies on both half-lines with constants $K_r,\mu_r > 0$, independent of $\lambda$. Let $B_r^{u,s}(\lambda)$ be as in Theorem \ref{Sturm}. Since $\Sigma_{\Lambda,0}$ is bounded and $B_r^{u,s}(\cdot)$ is continuous, there exists a constant $C_{\Lambda} > 0$ such that $\|B_r^{u,s}(\lambda)\| < C_\Lambda$ for $\lambda \in \Sigma_{\Lambda,0}$. Now, combine estimate (\ref{fullestimate2}) and Lemma \ref{compdich}: there exists, for $\epsilon > 0$ sufficiently small, bases $B_{\epsilon}^{u,s} \colon \Sigma_{\Lambda,\delta} \to \mathrm{Mat}_{2n \times n}(\C)$ of $\Proj_{fd,\epsilon}(0,\lambda)[\C^{2n}] = B_\epsilon^s[\C^n]$ and $\ker(\Proj_{fd,\epsilon}(0,\lambda)) = B_\epsilon^u[\C^n]$, such that
\begin{align} \|B_{\epsilon}^{u,s}(\lambda) - B_r^{u,s}(\lambda)\| = \ord(\epsilon^{\chi \mu_r}),\label{faststableest2}\end{align}
uniformly in $\lambda \in \Sigma_{\Lambda,\delta}$. Via (\ref{faststableest2}) we establish, provided that $\epsilon > 0$ is sufficiently small, the bound $\|B_\epsilon^{u,s}(\lambda)\| < C_\Lambda$ for $\lambda \in \Sigma_{\Lambda,\delta}$.\\
\\
\textit{Comparing $\E_{f,\epsilon}$ with the fast reduced Evans function $\E_{f,0}$}\\
Define for $\lambda \in \Sigma_{\Lambda,\delta}$
\begin{align*}\h_{\epsilon}(\lambda) := \left(\T_{fd,\epsilon}(-L_\epsilon,0,\lambda)B_{\epsilon}^{u}(\lambda), \T_{fd,\epsilon}(L_\epsilon,0,\lambda)B_{\epsilon}^{s}(\lambda)\right).\end{align*}
Since $\Proj_{fd,\epsilon}(\cdot,\lambda)$ is $2L_\epsilon$-periodic, the first $n$ column vectors of $\h_\epsilon(\lambda)$ form a basis of $\ker(\Proj_{fd,\epsilon}(L_\epsilon,\lambda))$ and the last $n$ column vectors form a basis of $\Proj_{fd,\epsilon}(L_\epsilon,\lambda)[\C^{2n}]$. Thus, $\h_\epsilon(\lambda)$ is invertible. By Hadamard's inequality we have $\det(\h_\epsilon(\lambda)) = \ord(e^{-2n\mu_{fd}L_\epsilon})$, uniformly in $\lambda \in \Sigma_{\Lambda,\delta}$. Moreover, since $\Proj_{fd,\epsilon}(\cdot,\lambda)$ is $2L_\epsilon$-periodic, we estimate
\begin{align} \begin{split} \|\T_{fd,\epsilon}(0,L_\epsilon,\lambda)\T_{fd,\epsilon}(-L_\epsilon,0,\lambda)B_{\epsilon}^{u}(\lambda)\| &\leq  K_{fd}^2C_\Lambda e^{-2\mu_{fd}L_\epsilon}, \\ \|\T_{fd,\epsilon}(0,-L_\epsilon,\lambda)\T_{fd,\epsilon}(L_\epsilon,0,\lambda)B_{\epsilon}^{s}(\lambda)\| &\leq K_{fd}^2C_\Lambda e^{-2\mu_{fd}L_\epsilon},\end{split} \label{basesest}
\end{align}
for $\lambda \in \Sigma_{\Lambda,\delta}$. We combine estimates (\ref{faststableest2}) and (\ref{basesest}) and derive
\begin{align*} &\left(\T_{fd,\epsilon}(0,-L_\epsilon,\lambda) - \gamma \T_{fd,\epsilon}(0,L_\epsilon,\lambda)\right)\h_{\epsilon}(\lambda) = \left(B_r^{u}(\lambda),\gamma B_r^s(\lambda)\right) + \ord(\epsilon^{\chi\mu_r}), \end{align*}
 uniformly in $\lambda \in \Sigma_{\Lambda,\delta}$ and $\gamma \in S^1$. Taking determinants in the latter matrix equation yields
\begin{align*}  \E_{f,\epsilon}(\lambda,\gamma)\det(\h_{\epsilon}(\lambda)) &= (-\gamma)^n\E_{f,0}(\lambda) + \ord(\epsilon^{\chi\mu_r}),\end{align*}
uniformly in $\lambda \in \Sigma_{\Lambda,\delta}$ and $\gamma \in S^1$. Thus, defining $h_\epsilon(\lambda) := \det(\h_{\epsilon}(\lambda))$ concludes the proof.
\end{proof}
\begin{rem}[\textit{The connection between $\E_{f,\epsilon}$ and $\E_{f,0}$}]
In the proof of  Lemma \ref{Evansfastestimate} the connection between $\E_{f,\epsilon}$ and $\E_{f,0}$ is given by the matrix $\h_\epsilon$. This idea is taken from the proof of \cite[Theorem 2]{SAN}. However, the context in \cite{SAN} is different. Here, one shows that the eigenvalues of a periodic boundary value problem are exponentially close to the eigenvalues of the corresponding unbounded problem. $\hfill \blacksquare$
\end{rem}
\subsection{Application of Rouch\'e's Theorem} \label{4.5}
In contrast to the estimate achieved in Lemma \ref{Evansslowestimate}, we need to rescale $\E_{f,\epsilon}$ in Lemma \ref{Evansfastestimate} by an exponentially small quantity $h_{\epsilon}$ in order to relate it to the $\epsilon$-independent fast reduced Evans function $\E_{f,0}$. This quantity prevents us from directly estimating the Evans function $\E_\epsilon$ by the reduced Evans function $\E_0(\lambda,\gamma) = (-\gamma)^n\E_{s,0}(\lambda,\gamma)\E_{f,0}(\lambda)$, using the estimates in Lemmas \ref{Evansslowestimate} and \ref{Evansfastestimate}. Nevertheless, it is still possible to compare the zero sets of $\E_\epsilon$ and $\E_0$ using the classical symmetric version of Rouch\'e's Theorem due to Estermann. This will conclude the proof of our main results.
\begin{proof}[Proof of Theorem \ref{mainresult}]
Let $\Lambda \in (-\Lambda_0,0)$ with $\Lambda_0 > 0$ as in Lemma \ref{lemlambda0}. Take $\delta > 0$ sufficiently small such that $\Gamma \subset \Sigma_{\Lambda,\delta}$. By Proposition \ref{slowEvans}-1 the slow reduced Evans function $\E_{s,0}(\cdot,\gamma)$ is meromorphic in the interior of $\Gamma$ and analytic and non-zero on $\Gamma$. Moreover, by Theorem \ref{Sturm} the fast reduced Evans function $\E_{f,0}$ is analytic on $\Gamma$ and its interior and is non-zero on $\Gamma$. Hence, for $\epsilon > 0$ sufficiently small, we have for all $\lambda \in \Gamma$ and $\gamma \in S^1$
\begin{align} \begin{split}|\E_\epsilon&(\lambda,\gamma) - \E_0(\lambda,\gamma)| \\&\leq (1- h_{\epsilon}(\lambda))|\E_\epsilon(\lambda,\gamma)| + |\E_{s,\epsilon}(\lambda,\gamma)\E_{f,\epsilon}(\lambda,\gamma)h_{\epsilon}(\lambda) - (-\gamma)^n\E_{s,0}(\lambda,\gamma)\E_{f,0}(\lambda)|
\\& < |\E_\epsilon(\lambda,\gamma)| + |\E_0(\lambda,\gamma)|,\end{split} \label{Roucheineq}\end{align}
where we have used Theorem \ref{ricfac} and Lemmas \ref{Evansslowestimate} and \ref{Evansfastestimate}. The result follows by an application of the symmetric version of Rouch\'e's Theorem.
\end{proof}
\begin{proof}[Proof of Theorem \ref{fixgamma}]
Let $\Lambda \in (-\Lambda_0,0)$ with $\Lambda_0 > 0$ as in Lemma \ref{lemlambda0}. Let $\gamma \in S^1$. By Proposition \ref{slowEvans}-1,2 and Theorem \ref{Sturm} the map $\E_\gamma \colon \Sigma_{\Lambda,0} \to \C$ given by $\E_\gamma(\lambda) = (-\gamma)^n\E_{s,0}(\lambda,\gamma)\E_{f,0}(\lambda)$ is non-trivial and analytic. The zeros of $\E_\gamma$ are given by the discrete set $\N_\gamma := \N(\E_{f,0}) \cup \N(\E_{s,0}(\cdot,\gamma))$. Now, take $\delta_\gamma > 0$ sufficiently small such that for every $\lambda_\diamond \in \N_\gamma$ the disc $\overline{B}(\lambda_\diamond,\delta_\gamma)$ contains no $\lambda \in \N_\gamma$ with $\lambda \neq \lambda_\diamond$. Take $0 < \delta < \delta_\gamma$ and $\lambda_\diamond \in \N_\gamma$. For all $\lambda \in \partial B(\lambda_\diamond,\delta)$ we have by construction that $\E_\gamma$ is non-zero. Hence, for $\epsilon > 0$ sufficiently small, inequality (\ref{Roucheineq}) holds true for all $\lambda \in \partial B(\lambda_\diamond,\delta)$. The result follows by an application of Rouch\'e's Theorem.
\end{proof}
\begin{proof}[Proof of Proposition \ref{zeropole}]
Let $\Lambda \in (-\Lambda_0,0)$ with $\Lambda_0 > 0$ as in Lemma \ref{lemlambda0}. Let $\lambda_\diamond$ be a simple zero of $\E_{f,0}$. By hypothesis, there exists $\delta_1 > 0$ such that
\begin{align*}\overline{B}(\lambda_\diamond,\delta_1) \subset \{\lambda_\diamond\} \cup \left[\Sigma_{\Lambda,0} \!\setminus\! \left(\bigcup_{\gamma \in S^1} \N(\E_{s,0}(\cdot,\gamma)) \cup \N(\E_{f,0})\right)\right].\end{align*}
Take $0 < \delta < \delta_1$. The disc $\overline{B}(\lambda_\diamond,\delta)$ contains no zeros of $\E_{s,0}(\cdot,\gamma)$ for every $\gamma \in S^1$ and precisely one simple zero $\lambda_\diamond$ of $\E_{f,0}$. Hence, for $\epsilon > 0$ sufficiently small, inequality (\ref{Roucheineq}) holds true for all $\lambda \in \partial B(\lambda_\diamond,\delta)$ and $\gamma \in S^1$. The result follows by combining Rouch\'e's Theorem and Proposition \ref{slowEvans}-3.
\end{proof}
\begin{rem}[\textit{Rate of convergence}]
The technical results Lemmas \ref{Evansslowestimate} and \ref{Evansfastestimate} seem to provide a rate at which the spectrum $\sigma(\El_\epsilon)$ converges to the spectrum in the singular limit. However, the approximations in these lemmas are only valid away from the zeros of the fast reduced Evans function $\E_{f,0}$! So, one can only deduce that spectrum converging to
\begin{align*} \left[\bigcup_{\gamma \in S^1} \N(\E_{s,0}(\cdot,\gamma))\right] \setminus \N(\E_{f,0}),\end{align*}
does this at an algebraic rate of order $\ord(\epsilon^{\mu_s})$. By making the parameter $\delta$ appearing in the proof of Lemma \ref{Evansfastestimate} dependent on $\epsilon$, it may be possible to derive an overall rate at which the spectrum $\sigma(\El_\epsilon)$ converges to its singular limit spectrum. However, this is beyond the scope of this paper.$\hfill \blacksquare$
\end{rem}
\section{Concluding remarks} \label{sec7}
\subsection{Discussion}
As mentioned in the introduction, our factorization method via the Riccati transformation of the Evans function offers one unified analytic alternative to both the elephant trunk procedure developed by Alexander, Gardner and Jones \cite{AGJ,GJO} and the NLEP approach of \cite{DGK,DGK2} -- that both have a geometric nature. It is worthwhile to compare and discuss the links between these methods.\\
\\
Moreover, the present work can be seen as the natural generalization of the spectral analysis \cite{PLO} of periodic pulse solutions in the Gierer-Meinhardt equations to periodic pulse patterns in the general class of `slowly nonlinear', $(m+n)$-component, singularly perturbed reaction-diffusion equations \eqref{diff}. Recently, a similar generalization to slowly nonlinear $2$-component systems for homoclinic pulses has been developed in \cite{VEA}. In that sense, the present paper stands in the tradition of \cite{DGK2,VEA,PLO}. However, our spectral analysis differs fundamentally from the analyses in these works, which rely eventually on the geometric approaches developed in \cite{AGJ,DGK,GJO}.
\subsubsection{Relation to the elephant trunk procedure}
Consider a localized pulse solution to a $2$-component, singularly perturbed reaction-diffusion equation. When the linear stability problem (LSP) has a slow-fast structure, it is a general phenomenon that it decouples outside the pulse region due to exponential decay of the solution to the asymptotic background state. This yields a decomposition of the solution space into three subspaces $V_{s\pm} \oplus V_{c\pm} \oplus V_{u\pm}$ at both sides ($\pm$) of the pulse region. Here, $V_{s\pm}$ consists of fast exponentially decaying solutions. Similarly, $V_{u\pm}$ consists of fast exponentially increasing solutions. Lastly, $V_{c\pm}$ consists of solutions that evolve slowly. In the sense of \cite{PALE}, one could say (LSP) admits exponential separations with respect to the decompositions $V_{s\pm} \oplus V_{c\pm} \oplus V_{u\pm}$. The difficulty is to `glue' the subspaces $V_{\cdot+}$ and $V_{\cdot-}$ for $\cdot = u,s,c$ together, yielding an exponential separation of (LSP) on the whole line. Eventually, this
induces a factorization of the Evans function in a fast and slow component.\\
\\
Gardner and Jones achieved this in \cite{GJO} by considering (LSP) in projective space. When (LSP) is asymptotically of constant coefficients type, one can first obtain stable and unstable bundles. Subsequently, these bundles are split into fast and slow (un)stable subbundles. The elephant trunk lemma is used to track the fast (un)stable bundle through the pulse region. By the control on the fast subbundle, it is possible to approximate the dynamics of the slow (un)stable subbundles. Eventually, this yields a $(1,2,1)$-exponential separation of (LSP) on $\R$. Note that the $2$-dimensional center direction corresponds to the slow stable and unstable subbundles. In our stability analysis, the Riccati transformation plays the role of the elephant trunk lemma -- see Section \ref{4.4}. This transformation yields an $(n,2m,n)$-exponential separation on $\R$ of (LSP) as long as we are not close to the eigenvalues of the fast singular limit problem. \\
\\
Although the proof of the elephant trunk lemma has been worked out in full detail for some specific $2$-component models \cite{DGK,ESZ,GJO,RUB} only, it is widely accepted that the method can be followed for a larger class of systems. However, there are some limitations. For instance, the elephant trunk lemma is only suitable for linear stability problems that have an asymptotically constant coefficient matrix. This is neither a restriction for slowly linear systems as the classical Gray-Scott and Gierer-Meinhardt models nor for homoclinic pulses on $\R$. However, the linear stability problem associated with spatially periodic patterns in slowly nonlinear systems exhibits non-autonomous behavior in the background state on its domain of periodicity -- and thus does not approach a constant coefficient matrix. This prohibits the application of the elephant trunk procedure. Moreover, the elephant trunk lemma is only capable of tracking the `most unstable' fast solution, which corresponds to the (simple) eigenvalue of largest real part of the asymptotic coefficient matrix. Therefore, it is unclear how to obtain the exponential separation with the elephant trunk method in the multi-dimensional setting $n > 1$.\\
\\
Furthermore, there is a major difference in the mathematical framework used in \cite{AGJ,GJO} and our work. The framework in \cite{AGJ, GJO} has a highly geometrical character, whereas our method is of a more analytical nature. Alexander, Gardner and Jones track solutions via vector bundles formed from the projectivized (LSP). This has the advantage that the generated bundles have a clean and natural characterization as $\epsilon$ tends to zero, whereas the actual solutions of (LSP) become singular. On the other hand, one could argue that exponential dichotomies provide a natural framework to capture the dynamics of (LSP) being a non-autonomous linear system, which depends analytically on the parameter $\lambda$. The Riccati transformation is naturally formulated in terms of exponential dichotomies and is explicit in terms of the coefficient matrix of (LSP). Therefore, the exponential separation of the solution space is much more explicit than in \cite{AGJ, GJO}, which shortens proofs. Finally, it is interesting to remark that in both the approach initiated by Alexander, Gardner and Jones and our method we need an a-priori $\epsilon$-independent estimate on the sector containing the spectrum. Our proof of this fact in Section \ref{4.1} forms an analytical counterpart to the geometrical proof provided in \cite[Proposition 2.2]{AGJ} and \cite[Lemma 3.3]{GJO}.
\subsubsection{Relation to the NLEP approach}
Based on the geometric methods of Alexander, Gardner and Jones \cite{AGJ,GJO}, the NLEP approach was developed in the context of the stability of homoclinic $N$-pulse patterns in the Gray-Scott equation \cite{DGK} and Gierer-Meinhardt-type models \cite{DGK2}. This method established the approximation of the Evans function by the product \eqref{factorintrred} of an analytic fast reduced Evans function and a meromorphic slow reduced Evans function and provided explicit analytic expressions for both factors. The NLEP approach was extended to the spectral analysis of spatially periodic pulse patterns (in semi-strong interaction) in the generalized Gierer-Meinhardt equations in \cite{PLO} and to the stability of heteroclinic and homoclinic multi-front patterns in $2$- and $3$-component bistable systems of FitzHugh-Nagumo-type \cite{DIN, HEI}. Moreover, the method has recently been generalized to the stability of homoclinic pulses in slowly nonlinear systems in \cite{VEA,VEE}. In each of these works, the fast and slow reduced Evans functions are interpreted geometrically in terms of fast and slow transmission functions that encode the passage of specially selected fast and slow basis functions over the fast pulse regions. The expressions for the slow transmission functions include Melnikov-type components. The meromorphic character of the slow reduced Evans function generates the zero-pole cancelation mechanism -- also called NLEP paradox -- in each of these models. The spectral analysis for periodic pulse solutions developed here shows that these phenomena occur in a broad class of multi-component singularly perturbed reaction-diffusion systems. \\
\\
Although the present work stands in the tradition of \cite{DGK,DGK2,DIN,VEA,HEI,PLO}, the methods differ fundamentally. Unlike these works, our analysis is based on an intrinsically analytic reduction method. This has the advantage that our spectral analysis allows for non-autonomous behavior of the linear stability problem outside the pulse region -- a crucial extension in the case of spatially periodic patterns in slowly nonlinear systems. This extended applicability of the present method also plays a role in the spectral analysis of homoclinic patterns: for instance a $uv$-term in the $v$-component of the $n=m=1$ version of \eqref{reac} cannot be allowed in \cite{VEA}, whereas our stability analysis can indeed handle such terms. Moreover, unlike in the present work, the singular limit problems appearing in \cite{DGK,DGK2,DIN,VEA,HEI,PLO} are scalar, which significantly simplifies the analysis of these problems. In \cite{DGK,DGK2,PLO} the slow and fast reduced Evans functions can be explicitly computed in terms of hypergeometric functions, while in \cite{DIN,HEI} the stability of the (multi-)fronts is determined by spectrum near the origin, so that the relevant reductions can be determined in a relatively straightforward manner. An extensive analysis of the multi-dimensional singular limit problems, as we did in Section \ref{4.2}, is thus not necessary in these cases. Finally, there also is an important similarity between the geometric methods in the literature and the analytic approach developed here. Although the methods are most often applied to patterns that exhibit a reversibility symmetry, this symmetry is not essential for the application of the methods -- see \cite{HEI} for an example. Nevertheless, such a symmetry in general simplifies the analysis -- as mentioned in Proposition \ref{slowm1reversible}.
\subsection{Future directions} \label{future}
Perhaps the most pressing question is how the fine structure of the `small' spectrum around zero can be determined (Section \ref{secsmallspectrum}). This fine structure is crucial to decide upon spectral stability (Corollary \ref{stability}) and, eventually, nonlinear diffusive stability (Remark \ref{nonlinearstab}). In \cite{SAS} the fine structure is presented for nearly homoclinic wave trains and in \cite{PLO} for periodic pulse solutions of semi-strong interaction type in the setting of the slowly linear Gierer-Meinhardt equation. However, the methods in these two papers are not directly applicable to our situation -- see Remark \ref{finestructure}. Nevertheless, preliminary investigations indicate that the analytic methods of \cite{SAS} may be extended to the present semi-strong interaction regime by separating fast from slow dynamics via Riccati transformations. This is the subject of work in progress.\\
\\
Another direction is the nature of destabilization of spatially periodic pulse solutions. Recent research \cite{STE,STE2} shows that destabilization mechanisms can be rather complex when periodic patterns approach a homoclinic limit. While increasing the wavelength, the character of the destabilization alternates between two kinds of Hopf bifurcations. This phenomenon is called the `Hopf dance' \cite{STE,STE2}. It has been analytically established in (slowly linear) Gierer-Meinhardt models in \cite{STE} and recovered by numerical methods in the generalized Klausmeier-Gray-Scott model \cite{STE,STE2}. The latter observations suggest that it is a persistent mechanism -- at least in slowly linear models -- especially since the systems considered in \cite{STE2} include nonlinear diffusion terms. However, the analysis in \cite{STE} suggests that slowly nonlinear dynamics of the models may a priori have a decisive impact on the appearance of the Hopf dance. Both the Hopf dance as well as the `belly dance' \cite{STE} -- an associated higher order phenomenon -- can be analyzed in the general slowly nonlinear setting of \eqref{reac} by the methods developed here -- see \cite{RIJV}.\\
\\
Finally, we note that the analysis developed in this work may be used as the foundation for an analytic study of the bifurcations exhibited by spatially periodic patterns in semi-strong interaction -- patterns that thus are `far from equilibrium'. In other words, the developed explicit insights in the linear stability is a key to understanding the weakly nonlinear dynamics of patterns as they have become spectrally unstable. A first -- and fundamental -- step in this direction has been taken in \cite{VEE2}, in which a normal form approach associated with a Hopf destabilization of homoclinic pulses in the $n=m=1$ version of \eqref{reac} is developed. Unlike known classical slowly linear examples such as the Gray-Scott and Gierer-Meinhardt models, the Hopf bifurcation for homoclinic pulses can be supercritical. It can even be the first step in a sequence of further bifurcations that leads to complex (amplitude) dynamics of a standing solitary pulse -- as observed in the simulations of \cite{VEE}. It should be noted that the subcriticality of the Hopf bifurcation in Gierer-Meinhardt-type models was a `conjecture' based on numerical evidence until the work in \cite{VEE2}. Both the above described Hopf dance and the fact that the pulses that together form the spatially periodic patterns are in semi-strong interaction, indicate that the weakly nonlinear dynamics of these patterns beyond their destabilization may be very rich.
\appendix
\section{Proof of existence of stationary, spatially periodic pulse solutions} \label{A0}
\begin{proof}[Proof of Theorem \ref{maintheorem}] We adopt the notation of assumptions \ref{assE1} and \ref{assE2}. Recall that (\ref{ODE}) is $R$-reversible by Remark \ref{reversiblerem}. Our proof is based on the fact that every orbit that crosses the space $\ker(I-R)$ twice, must be a closed loop. Therefore, we start with a `good' set of initial conditions $\Ze \subset \ker(I-R)$ and track these conditions under the forward flow of (\ref{ODE}) with the aid of an appropriate Exchange Lemma. We will show that the tracked trajectories remain close to the singular orbit consisting of the segments $\psi_{\s}(\xx), \xx \in [0,2\check{L}_0]$ and $\pih_{\ho}(x,u_0), x \in \R$, where $\pih_\ho$ is defined in Remark \ref{TOTD}. In particular, we establish that the union of trajectories starting in $\Ze$ intersects $\ker(I-R)$ transversally in some point $P_{\p,\epsilon}$, which lies close to $\pih_{\ho}(0,u_0)$. Finally, the desired periodic solution is the one that starts in $P_{\p,\epsilon}$. \\
\\
\textit{A good set of initial conditions}\\
Denote by $e_i, i \in 1,\ldots,m$ the unit basis of $\R^m$. Let $\U$ be the $(m \times (m-1))$-matrix with column vectors $e_1,\ldots,e_{i_*-1},e_{i_*+1},\ldots,e_m$, where $i^*$ is as in \ref{assE2}. Consider the $(m+n-1)$-dimensional manifold,
\begin{align*} \Ze := \{(u_1 + u,0,v,0) : u \in \U[\R^{m-1}], v \in \R^n\} \subset \ker(I-R).\end{align*}
The intersection of $\Ze$ and $\M$ corresponds to,
\begin{align*}\Proj_0 := \{(u_1 + u,0) : u \in \U[\R^{m-1}]\} \subset \ker(I - R_s).\end{align*}
By assumption \ref{assE2} $\Proj_0$ becomes under the forward flow of the slow reduced system (\ref{slowp}) an $m$-dimensional manifold $\Proj_0^*$, which intersects $\T_-$ transversely at $(u_0,-\J(u_0))$. Indeed, we have
\begin{align*}
0 &\neq \det\left(\begin{array}{c|c} \Phi_s(0,\check{L}_0)\left[\begin{array}{c} I \\ -\partial_u \J(u_0)\end{array}\right] & \begin{array}{cc} 0 & \U \\ H_1(u_1,0) & 0 \end{array}\!\end{array}\right)\\
 &=\det\left(\begin{array}{c|c} \!\begin{array}{cc} I & -\J(u_0) \\ -\partial_u \J(u_0) & H_1(u_0,0)\end{array} & \Phi_\s(\check{L}_0,0) \left[\begin{array}{c} \U \\ 0 \end{array}\right]\end{array}\right).
\end{align*}
We have used here that $\Phi_s(\check{L}_0,0)$ induces an isomorphism between the tangent spaces of $\Proj_0^*$ at $(u_1,0)$ and at $(u_0,-\J(u_0))$ and that the determinant of $\Phi_s(\check{L}_0,0)$ equals $1$ by Liouville's Theorem.\\
\\
\textit{Putting system (\ref{ODE}) in Fenichel normal form}\\
Let $\M_0$ be a compact $2m$-dimensional submanifold of $\M$ so large that $\M_0$ serves as a neighborhood of $\psi_{\s}(\xx), \xx \in [0,\check{L}_0]$ and of the projection $(u_0,\int_0^x H_2(u_0,v_\ho(z,u_0))dz), x \in \R$ of $\pih_{\ho}(x,u_0)$ on $\M$. \\
\\
By assumption \ref{assS2} $\M_0$ is normally hyperbolic. So, according to Fenichel's theory \cite[Theorem 9.1]{FEN2}, $\M_0$ perturbs, for $\epsilon > 0$ sufficiently small, to a manifold $\M_\epsilon$, which is diffeomorphic to $\M_0$ and locally invariant for the dynamics of (\ref{ODE}). Since $\M_0$ is itself locally invariant for the dynamics of (\ref{ODE}), one readily establishes that there exists a constant $C_0 > 0$ such that $\M_\epsilon$ has Hausdorff distance $\ord(e^{-C_0/\epsilon})$ from $\M_0$ -- see also \cite[Theorem 2.1]{DES}.\\
\\
By \cite[Proposition 1]{JOT} there exists a $C^1$ coordinate change $(u,p,v,q) \to (a,b,c)$, which brings system (\ref{ODE}) into `Fenichel normal form',
\begin{align}
\left\{\!
\begin{array}{rcl}
\partial_x a &=& A(a,b,c,\epsilon)a\\
\partial_x b &=& B(a,b,c,\epsilon)b\\
\partial_x c &=& \epsilon K(c,\epsilon) + C(a,b,c,\epsilon)(a \otimes b)\end{array}\right., \ \ \ a,b \in \R^n, c \in \R^{2m}, \label{Fenichelnorm}
\end{align}
in an $\epsilon$-independent neighborhood $\D$ of $\M_0$, where $A,B,K$ and $C$ are continuous, $K$ is a vector in $\R^{2m}$, $A,B$ are square matrices of order $n$ and $C$ is a tensor of appropriate rank. In particular, $C(a \otimes b)$ is bilinear in $a$ and $b$. In `Fenichel coordinates' $\M_\epsilon$ is given by $a = b = 0$ and the local stable and unstable manifolds $W_\epsilon^{u,s}(\M_\epsilon) \cap \D$ of $\M_\epsilon$ are the spaces $b = 0$ and $a = 0$, respectively. Since $R$ maps $W_\epsilon^u(\M_\epsilon)$ to $W_\epsilon^s(\M_\epsilon)$, $\ker(I-R) \cap \D$ is contained in the space $a = b$. Finally, system $\partial_\xx c = K(c,0)$ corresponds to the slow reduced system (\ref{slowp}).\\
\\
\textit{Application of the Exchange Lemma}\\
By the latter paragraph $\Ze \subset \ker(I-R)$ intersects the local stable manifold $W_0^s(\M_0) \cap \D$ of the fast reduced system (\ref{relp}) transversally at $(u_1,0,0,0)$. Moreover, the reduced slow flow (\ref{slowp}) on $\M$ is not tangent to $\Proj_0$ at $(u_1,0)$. We conclude that the conditions for the Exchange Lemma \cite[Section 2.5]{SEC} are satisfied.\\
\\
Denote by $\Proj_\epsilon \subset \M_\epsilon$ the $(m-1)$-dimensional manifold, where $\Ze$ and the local stable manifold $W_{\epsilon}^s(\M_\epsilon) \cap \D$ meet transversally. Moreover, let $\Ze_\epsilon^*$ and $\Proj_\epsilon^* \subset \M_\epsilon$ be the $(m+n)$- and $m$-dimensional manifolds obtained by flowing initial conditions on $\Ze$ and $\Proj_\epsilon$ forward in (\ref{ODE}). Finally, we denote by
\begin{align*} \Y_\epsilon := \bigcup_{\phi \in \Proj_\epsilon^*} W_\epsilon^u(\phi) \subset W_\epsilon^u(\M_\epsilon),\end{align*}
the union of unstable fibers in (\ref{ODE}) with base points in $\Proj_\epsilon^*$. Note that $\Y_\epsilon$ is locally invariant by Fenichel's theory \cite[Theorem 9.1]{FEN2}. By combining \ref{assE2} with the Exchange Lemma \cite[Theorem 2.3]{SEC}, there exists an $(m+n)$-dimensional submanifold $\Ze_{1,\epsilon}$ of $\Ze_\epsilon^*$ and an $\epsilon$-independent neighborhood $\Delta \subset \D$ of $(u_0,-\J(u_0),0,0)$ such that the Hausdorff distance between $\Delta \cap \Y_\epsilon$ and $\Ze_{1,\epsilon}$ is $\ord(e^{-C_0/\epsilon})$. Moreover, trajectories crossing $\Ze_{1,\epsilon}$ remain in $\D$ during the excursion from $\Ze$ to $\Ze_{1,\epsilon}$.\\
\\
\textit{The singular limit of $\Y_\epsilon$}\\
First, recall that $\Proj_0^*$ intersects $\T_-$ transversely at $(u_0,-\J(u_0))$. Second, the unstable manifold $W_0^u(\M_0)$ in (\ref{relp}) intersects $\ker(I-R)$ transversely in an $m$-dimensional manifold $\Es_0 := \{\pih_\ho(0,u) : u \in U_\ho\}$ by assumption \ref{assE1}. The $\alpha$-limit set of $\Es_0$ equals the touch-down manifold $\T_-$ in $\M$ by Remark \ref{TOTD}. We now put these two items together and conclude that the $(m+n)$-dimensional union,
\begin{align*} \Y_0 := \bigcup_{\phi \in \Proj_0^*} W_0^u(\phi) \subset W_0^u(\M_0),\end{align*}
of unstable fibers in (\ref{relp}) with base points in $\Proj_0^*$ intersects the $(m+n)$-dimensional manifold $\ker(I-R)$ transversally in the point $\pih_\ho(0,u_0)$. \\
\\
\textit{Obtaining the periodic orbit}\\
By Fenichel theory \cite[Theorem 9.1]{FEN2} the manifolds $\Y_\epsilon$ and $\Y_0$ have Hausdorff distance $\ord(\epsilon)$ in a neighborhood of the intersection point $\pih_\ho(0,u_0)$. Therefore, provided $\epsilon > 0$ is sufficiently small, $\Y_\epsilon$ intersects $\ker(I-R)$ transversally in some point $P_{\ho,\epsilon}$, which lies $\ord(\epsilon)$-close to $\pih_\ho(0,u_0)$. Denote by $\pih_{\ho,\epsilon}(x)$ the solution to (\ref{ODE}) with initial condition $\pih_{\ho,\epsilon}(0) = P_{\ho,\epsilon}$.\\
 \\
Since $\pih_\ho(0,u_0)$ is homoclinic to $(u_0,-\J(u_0),0,0) \in \M$, there exists $x_0 < 0$ such that $\pih_\ho(x_0,u_0)$ is contained in the neighborhood $\Delta$ of $(u_0,-\J(u_0),0,0)$. Hence, since $\pih_{\ho,\epsilon}(0)$ is $\ord(\epsilon)$-close to $\pih_\ho(0,u_0)$ and $x_0$ is independent of $\epsilon$, one derives via Gr\"onwall type estimates that $\pih_{\ho,\epsilon}(x_0)$ is contained in $\Delta \cap \Y_\epsilon$. Recall that the outcome of the Exchange Lemma is that $\Y_\epsilon$ has Hausdorff distance $\ord(e^{-C_0/\epsilon})$ from $\Ze_{1,\epsilon}$ in the neighborhood $\Delta$ of $\pih_{\ho,\epsilon}(x_0)$.\\
 \\
Denote by $\Ze_{1,\epsilon}^*$ the $(m+n)$-dimensional manifold obtained by flowing $\Ze_{1,\epsilon}$ forward in (\ref{ODE}). Since $x_0$ is $\epsilon$-independent, we infer, again via Gr\"onwall type estimates, that the Hausdorff distance between $\Y_\epsilon$ and $\Ze_{1,\epsilon}^*$ is $\ord(e^{-C_0/\epsilon})$ in a neighborhood of $\pih_{\ho,\epsilon}(0)$. Therefore, $\Ze_{1,\epsilon}^*$ intersects $\ker(I-R)$ transversally in some point $P_{\p,\epsilon}$, which is $\ord(\epsilon)$-close to $\pih_\ho(0,u_0)$. The solution $\pih_{\p,\epsilon}(x)$ with initial condition $\pih_{\p,\epsilon}(0) = P_{\p,\epsilon}$ is the desired periodic orbit. Indeed, $\pih_{\p,\epsilon}(x)$ crosses $\ker(I-R)$ at $x = 0$ and at some point $x = -L_\epsilon < 0$, since $\pih_{\p,\epsilon}$ is contained in $\Ze_\epsilon^*$. In particular, we have the desired symmetry properties: $\pih_{\p,\epsilon}(x) = R\pih_{\p,\epsilon}(-x)$ and $\pih_{\p,\epsilon}(L_\epsilon - x) = R\pih_{\p,\epsilon}(L_\epsilon + x)$ for $x \in \R$. \\
\\
\textit{Checking conditions 1-3 of \ref{assS3}}\\
First, assertion \ref{assS3}-2 is immediate by the fact that $\pih_{\ho}(0,u_0)$ lies $\ord(\epsilon)$-close to $\pih_{\p,\epsilon}(0)$. Assertion \ref{assS3}-3 requires more work. By a Gr\"onwall type estimate, we establish
\begin{align} \|\pih_{\ho}(x,u_0) - \pih_{\p,\epsilon}(x)\| = \ord(\epsilon), \label{exest1}\end{align}
for $x \in [x_0,0]$. By the construction in the latter paragraph, we may without loss of generality assume that $\pih_{\p,\epsilon}(x_0)$ is contained in $\Ze_{1,\epsilon}$. So, $\pih_{\p,\epsilon}(x)$ is in $\D$ for $x \in [-L_\epsilon,x_0]$. Denote by $a_{\p,\epsilon}(x),b_{\p,\epsilon}(x)$ and $c_{\p,\epsilon}(x)$ the components of the solution $\pih_{\p,\epsilon}(x)$ in Fenichel coordinates (\ref{Fenichelnorm}). By \cite[Corollary 1]{JOT} there exists $\epsilon$-independent constants $C_a,C_b,\mu_0 > 0$ such that
\begin{align} \|b_{\p,\epsilon}(x)\| \leq C_be^{-\mu_0L_\epsilon}, \ \ \ \|a_{\p,\epsilon}(x)\| \leq C_ae^{\mu_0(x-x_0)}, \label{exest2}\end{align}
for $x \in [-L_\epsilon,x_0]$. Recall that $\M_\epsilon$ has Hausdorff distance $\ord(e^{-C_0/\epsilon})$ to $\M_0 \subset \M$. Hence, using estimates (\ref{exest1}) and (\ref{exest2}) and the symmetry of $\pih_{\p,\epsilon}$, estimate (\ref{slowestS31}) in assertion \ref{assS3}-3 follows.\\
\\
By (\ref{exest2}) the norm of the tensor $a_{\p,\epsilon}(x) \otimes b_{\p,\epsilon}(x)$ is $\ord(e^{-\mu_0L_\epsilon})$ on $[-L_\epsilon,x_0]$. Therefore, using Gr\"onwall type estimates once again, there exists a solution $(0,0,c_{\s,\epsilon}(x))$ in the invariant manifold $\M_\epsilon \subset \{a = b = 0\}$ satisfying $\partial_x c = \epsilon K(c,\epsilon)$, which is $\ord(e^{-\mu_0L_\epsilon})$-close to $c_{\p,\epsilon}(x)$ for $x \in [-L_\epsilon,x_0]$. Moreover, there exists a solution $c_{\s,0}(\xx)$ to $\partial_\xx c = K(c,0)$, which satisfies $\|c_{\p,\epsilon}(x) - c_{\s,0}(\epsilon^{-1} x)\| = \ord(\epsilon)$ for $x \in [-L_\epsilon,x_0]$. Recall that system $\partial_\xx c = K(c,0)$ corresponds to the slow reduced system (\ref{slowp}). Thus, we have established the last estimate (\ref{slowestS32}) in \ref{assS3}-3 and assertion \ref{assS3}-1. Note that the solution $c_{\s,0}(u)$ corresponds to the solution $\psi_\s(\xx)$ to (\ref{slowp}) in \ref{assE2}.
\end{proof}
\section{Prerequisites for the proofs of the main results}\label{A1}
In this section we treat some prerequisites needed for the proofs of our main results. We will provide the proofs of statements that we could not find in literature.
\subsection{Bounded growth estimate}
We start with the notion of bounded growth, which is due to Coppel \cite{COP}.
\begin{defi}
Let $n \in \Z_{> 0}$, $J \subset \R$ an interval and $A \in C(J,\mathrm{Mat}_{n \times n}(\C))$. Denote by $T(x,y)$ the evolution operator of system,
\begin{align}\partial_x \phi = A(x)\phi,\ \ \ \phi \in \C^n,\label{linsys}\end{align}
Equation (\ref{linsys}) has \emph{bounded growth on $J$ with constants $\mu,K > 0$}, if for all $x,y \in J$ it holds
\begin{align*} \|T(x,y)\| \leq Ke^{\mu|x-y|}.\end{align*}
\end{defi}
Clearly, if $A(x)$ is bounded on $\R$, system (\ref{linsys}) has bounded growth on $\R$ with constants $\mu = \|A\|$ and $K = 1$ due to Gr\"onwall's inequality. In our spectral analysis, we often want to compare a linear system with its perturbation. This requires the following consequence of  Gr\"onwall's inequality for linear systems.
\begin{lem} \label{Palmer}
Let $n \in \Z_{> 0}$, $a,b \in \R, a < b$ and $A, B \in C([a,b],\mathrm{Mat}_{n \times n}(\C))$. Denote by $T(x,y)$ the evolution operator of system (\ref{linsys}). Similarly, denote by $T'(x,y)$ the evolution of system,
\begin{align}\partial_x \phi = B(x)\phi,\ \ \ \phi \in \C^n.\label{linsys2}\end{align}
Suppose (\ref{linsys}) has bounded growth on $[a,b]$ with constants $\mu,K > 0$. It holds
\begin{align*}\|T(x,y) - T'(x,y)\|\leq K \int_a^b\|A(z) - B(z)\|dz \exp\left(\mu(b-a) + K\int_a^b\|A(z)-B(z)\|dz\right),\end{align*}
for all $x,y \in [a,b]$.
\end{lem}
\begin{proof}
This follows from the proof of \cite[Lemma 1]{PALP}.
\end{proof}
\subsection{Exponential dichotomies}
Exponential dichotomies are an important tool in studying spectral properties of differential equations. They enable us to track solutions in the linear stability problem by separating the solution space in solutions that either decay exponentially in forward time or else in backward time.
\begin{defi}
Let $n \in \Z_{> 0}$, $J \subset \R$ an interval and $A \in C(J,\mathrm{Mat}_{n \times n}(\C))$. Denote by $T(x,y)$ the evolution operator of (\ref{linsys}). Equation (\ref{linsys}) has \emph{an exponential dichotomy on $J$ with constants $K,\mu > 0$ and projections $P(x) \colon \C^n \to \C^n, x \in J$} if for all $x,y \in J$ it holds
\begin{itemize}
 \item $P(x)T(x,y) = T(x,y) P(y)$;
 \item $\|T(x,y)P(y)\| \leq Ke^{-\mu(x-y)}$ for $x \geq y$;
 \item $\|T(x,y)(I-P(y))\| \leq Ke^{-\mu(y-x)}$ for $y \geq x$.
\end{itemize}
\end{defi}
Let $P(x),x \in J$ be the family of projections corresponding to an exponential dichotomy on $J$. For each $x \in J$, the range $P(x)[\C^n]$ and the kernel $\ker(P(x))$ are called the \emph{stable and unstable subspaces at $x$}, respectively. These spaces are often denoted by $E^s(x)$ and $E^u(x)$, leaving the projection $P(x)$ implicit. Similarly, we abbreviate $T^s(x,y) = T(x,y)P(y)$ and $T^u(x,y) = T(x,y)(I-P(y))$.\\
\\
Below we give a short overview of the properties of exponential dichotomies that we need for our spectral analysis. For an extensive introduction on dichotomies the reader is referred to the work of Coppel \cite{COP}. A generalization of the concept of exponential dichotomies is the notion of exponential separation, which is treated in a paper of Palmer \cite{PALE}.
\subsubsection{Sufficient criteria}
Clearly, an autonomous linear system with hyperbolic coefficient matrix has an exponential dichotomy on $\R$. This result can be extended to non-autonomous linear systems under the hypothesis that the coefficient matrix changes sufficiently slow and remains bounded.
\begin{prop} \label{suffcritexpodi}
Let $n \in \Z_{> 0}$, $J \subset \R$ an interval and $A \in C^1(J,\mathrm{Mat}_{n \times n}(\C))$ such that
\begin{enumerate}
\item There exists $\alpha > 0$ such that for each $x \in J$ the matrix $A(x)$ has (counted with algebraic multiplicity) $k$ eigenvalues with real part $\leq -\alpha$ and $n - k$ eigenvalues with real part $\geq \alpha$.
\item There exists $M > 0$, which bounds $A$ on $J$.
\end{enumerate}
There exists $\delta > 0$, depending only on $\alpha$ and $M$, such that, if for all $x \in J$ we have $\|A'(x)\| \leq \delta$, then (\ref{linsys}) has an exponential dichotomy with constants $K,\mu > 0$ and projections $P(x), x \in J$. Here, we have $\mu = \frac{1}{2}\alpha$ and $K$ depends only on $M$ and $\alpha$. Moreover, there exists a fundamental matrix solution $X(x)$ of (\ref{linsys}) such that
\begin{align*}P(x) = X(x)\left(\begin{array}{cc} I_k & 0 \\ 0 & 0\end{array}\right)X^{-1}(x), \ \ \ x \in J.\end{align*}
\end{prop}
\begin{proof}
This is the content of \cite[Proposition 6.1]{COP}.
\end{proof}
\subsubsection{Extending and pasting}
Once one puts a linear stability problem in the framework of exponential dichotomies, a great technical toolbox becomes available. Namely, on the one hand, there is a scala of constructions available to extend the interval of the dichotomy. On the other hand, exponential dichotomies persist under small perturbations of the equation. Therefore, these techniques enable us to decompose our linear stability problem into simpler subproblems. In this section, we treat extension and `pasting' of exponential dichotomies. In the next section \ref{roughness} we consider the persistence of exponential dichotomies under small perturbations.
\begin{lem}[Extension Lemma] \label{extensionexpdi}
Let $n \in \Z_{> 0}$, $J_2 \subset J_1 \subset \R$ intervals and $A \in C(J_1,\mathrm{Mat}_{n \times n}(\C))$. Suppose equation (\ref{linsys}) has an exponential dichotomy on $J_2$ with constants $K_2,\mu > 0$ and projections $P_2(x), x \in J_2$. Then, if the length of $J_1 \!\setminus\! J_2$ is finite, system (\ref{linsys}) has an exponential dichotomy on $J_1$ with constants $K_1,\mu > 0$ and projections $P_1(x), x \in J_1$. The constant $K_1$ satisfies
\begin{align*} K_1 = K_2\exp\left(\int_{J_1 \!\setminus\! J_2}(2 \|A(x)\| + \mu)dx\right).\end{align*}
Moreover, we have $P_1(x) = P_2(x)$ for all $x \in J_2$.
\end{lem}
\begin{proof}
This is proven in \cite[p. 13]{COP}.
\end{proof}
The next result shows that, if a periodic equation admits an exponential dichotomy on a sufficiently large interval, then it has an exponential dichotomy on the whole line.
\begin{lem}[Periodic Extension Lemma] \label{periodextensionexpdi}
Let $n \in \Z_{> 0}$, $T > 0$ and $A \in C(\R,\mathrm{Mat}_{n \times n}(\C))$. Suppose that $A$ is $T$-periodic and that equation (\ref{linsys}) has an exponential dichotomy on an interval $J$ of length $2T$ with constants $K,\mu > 0$. Let $M \geq \sup_{x \in \R} \|A(x)\|$ and $h = \mu^{-1}(\sinh^{-1}(4) + \log(K))$.\\
\\
If $T > 0$ is so large that $T \geq 2h$, then equation (\ref{linsys}) has an exponential dichotomy on $\R$ with constants $K_1,\mu_1 > 0$. We have $\mu_1 = h^{-1}\log 3$ and $K_1$ depends only on $M, K$ and $\mu$.
\end{lem}
\begin{proof}
This is the content of \cite[Theorem 1]{PALP}.
\end{proof}
Exponential dichotomies on two connected intervals can be pasted together as long as their stable and unstable spaces at the connection point are complementary.
\begin{lem}[Pasting Lemma] \label{pastingexpdi}
Let $n \in \Z_{> 0}$, $J \subset \R$ an interval and $A \in C(J,\mathrm{Mat}_{n \times n}(\C))$. Let $J_1,J_2$ be two intervals such that their union equals $J$ and $\max J_1 = b = \min J_2$ for some $b \in \R$. Suppose equation (\ref{linsys}) has exponential dichotomies on both $J_1$ and $J_2$ with constants $K,\mu > 0$ and projections $P_1(x), x \in J_1$ and $P_2(x), x \in J_2$, respectively.\\
\\
If $E_1^u(b) = \ker(P_1(b))$ and $E_2^s(b) = P_2(b)[\C^n]$ are complementary, then (\ref{linsys}) has an exponential dichotomy on $J$ with constants $K_1,\mu > 0$. Here, $K_1$ depends only on $K$ and $\|P\|$, where $P$ is the projection on $E_2^s(b)$ along $E_1^u(b)$.
\end{lem}
\begin{proof}
Let $X(x)$ be the fundamental matrix of (\ref{linsys}) satisfying $X(b) = I$. Define $P(x) = X(x)PX^{-1}(x)$ for $x \in J$, where $P$ is the projection on $E_2^s(b)$ along $E_1^u(b)$. Observe that $P = P(b)$ has the same range as $P_2(b)$ and the same kernel as $P_1(b)$. Now, the exposition in \cite[pp. 16-17]{COP} shows that (\ref{linsys}) has exponential dichotomies on $J_1$ and on $J_2$ with constants $K_1,\mu > 0$ and projections $P(x)$ for $x \in J_1$ and $x \in J_2$, respectively. We have $K_1 = K + K^2\|P\| + K^3$. To conclude the proof we need to show that the dichotomy estimates remain true on the union $J = J_1 \cup J_2$. Indeed, take $x \in J_2$ and $y \in J_1$. We estimate
\begin{align*} \|T(x,y)P(y)\| &\leq \|T(x,b)P_2(b)\| \|P\| \|P_1(b)T(b,y)\| \leq K^2\|P\|e^{-\mu(x-y)},\end{align*}
where we have used $P_2(b)P = P$ and $PP_1(b)=P$. Similarly, one estimates $\|T(y,x)(I-P(x))\| \leq K^2\|P\|e^{-\mu(x-y)}$ for $x \in J_2$ and $y \in J_1$.
\end{proof}
\subsubsection{Roughness} \label{roughness}
Exponential dichotomies are in particular useful to study the spectral properties of perturbed differential equations, since they persist under small perturbations of the equation. This property is often referred to as roughness or robustness. We start with a general roughness result for exponential dichotomies on arbitrary intervals.
\begin{prop}[Roughness on arbitrary intervals] \label{roughnessintervals} Let $n \in \Z_{> 0}$,$\delta > 0$, $J \subset \R$ an interval and $A, B \in C(J,\mathrm{Mat}_{n \times n}(\C))$ such that
\begin{align*} \sup_{x \in J} \|A(x) - B(x)\| \leq \delta.\end{align*}
Suppose (\ref{linsys}) has an exponential dichotomy on $J$ with constants $K,\mu > 0$ and projections $P(x) := X(x)PX^{-1}(x), x \in J$, where $X(x)$ is a fundamental matrix of (\ref{linsys}) and $P$ an orthogonal projection. If we have
\begin{align*}\delta \leq \frac{\mu}{36K^5},\end{align*}
then equation (\ref{linsys2}) has an exponential dichotomy on $J$ with constants $K_1,\mu_1 > 0$ and projections $P_1(x),x \in J$
satisfying $\mu_1 = \mu - 6K^3\delta, K_1 = 12K^3$ and for all $x \in J$ \begin{align*} \|P(x) - P_1(x)\| \leq \frac{144K^6\delta}{\mu}.\end{align*}
\end{prop}
\begin{proof}
This is the content of \cite[Proposition 5.1]{COP}.
\end{proof}
The latter persistence result can be simplified significantly in the case $J = \R$.
\begin{prop}[Roughness on $\R$] \label{RoughnessR}
Let $n \in \Z_{> 0}$,$\delta > 0$ and $A, B \in C(\R,\mathrm{Mat}_{n \times n}(\C))$ such that
\begin{align*} \sup_{x \in \R} \|A(x) - B(x)\| \leq \delta.\end{align*}
Suppose equation (\ref{linsys}) has an exponential dichotomy on $\R$ with constants $K,\mu > 0$. If we have
\begin{align*}\delta \leq \frac{\mu}{4K^2},\end{align*}
then equation (\ref{linsys2}) has an exponential dichotomy on $\R$ with constants $K_1,\mu_1 > 0$. Here, we have $\mu_1 = \mu - 2K\delta$ and $K_1$ depends on $K$ only.
\end{prop}
\begin{proof}
This can be found in \cite[pp. 34-35]{COP}.
\end{proof}
If an equation has an exponential dichotomy on $\R$, then it admits no non-trivial bounded solutions. It is possible to achieve persistence of the latter fact under milder conditions than those stated in Proposition \ref{RoughnessR}.
\begin{prop}[Persistence of no non-trivial bounded solutions] \label{roughnessPALI}
Let $n \in \Z_{> 0}$ and $A, B \in C(\R,\mathrm{Mat}_{n \times n}(\C))$. Suppose (\ref{linsys}) has both an exponential dichotomy on $\R$ with constants $K,\mu > 0$ and bounded growth on $\R$. Denote by $T_{1,2}(x,y)$ the evolution operators of systems (\ref{linsys}) and (\ref{linsys2}), respectively. If there exists $\tau \geq \mu^{-1}(\sinh^{-1}(4) + \log(K))$ such that for all $x,y \in \R$ with $|x-y| \leq 2\tau$ we have
\begin{align*} \|T_1(x,y) - T_2(x,y)\| < 1,\end{align*}
then (\ref{linsys2}) admits no non-trivial bounded solutions.
\end{prop}
\begin{proof}
This is the content of \cite[Theorem 1]{PALI}.
\end{proof}
\subsubsection{Uniqueness}
We emphasize that exponential dichotomies on an interval $J \subset \R$ are in general not unique. For instance, if $J = \R_{\geq 0}$, then the stable subspace $E^s(0)$ is uniquely determined, whereas the unstable subspace $E^u(0)$ can be chosen to be any complement of $E^s(0)$. However, given two exponential dichotomies it is possible to estimate the `gap' between the stable subspaces and unstable subspaces.
\begin{lem} \label{compdich}
Let $n \in \Z_{> 0}$, $a,b \in \R$ with $a < b$ and $A,B \in C([a,b],\mathrm{Mat}_{n \times n}(\C))$. Suppose equations (\ref{linsys}) and (\ref{linsys2}) have exponential dichotomies on $[a,b]$ with constants $K_{1,2},\mu_{1,2} > 0$ and projections $P_{1,2}(x), x \in J$. Denote by $T_{1,2}(x,y)$ the evolution operators of systems (\ref{linsys}) and (\ref{linsys2}). Let $\delta \geq 0$ such that
\begin{align*}\|T_1(a,b) - T_2(a,b)\| \leq \delta.\end{align*}
Then, for every $v \in E_1^s(a) = P_1(a)[\C^n]$, there exists $w \in E_2^s(a) = P_2(a)[\C^n]$ such that
\begin{align}\|v-w\| \leq (\delta + K_2e^{-\mu_2(b-a)})K_1e^{-\mu_1 (b-a)}\|v\|. \label{compdicest}\end{align}
Similarly, for every $v \in E_1^u(b) = \ker(P_1(b))$, there exists $w \in E_2^u(b) = \ker(P_2(b))$ such that (\ref{compdicest}) holds true.
\end{lem}
\begin{proof}
Let $v \in E_1^s(a)$ and consider $w = T_2(a,b)P_2(b)T_1(b,a)v \in E_2^s(a)$. We estimate
\begin{align*}
\|T_2(a,b)P_2(b)T_1(b,a)v - v\| &\leq \left[\|T_2(a,b) - T_1(a,b)\| + \|T_2(a,b)(I-P_2(b))\|\right]\|T_1(b,a)v\|\\
 &\leq (\delta + K_2e^{-\mu_2(b-a)})K_1e^{-\mu_1 (b-a)}\|v\|.
\end{align*}
The other statement is proven in an analogous way.
\end{proof}
\subsubsection{Inhomogeneous problems}
When comparing solutions to inhomogeneous problems using Gr\"onwall's inequality, one often obtains sharp estimates on finite intervals only, since the solutions to the corresponding homogeneous problems will grow in general exponentially. Exponential dichotomies prove to be important tools to compare bounded solutions to inhomogeneous problems on the whole real line. This is the content of the following technical result.
\begin{prop} \label{inhomexpdi}
Let $n \in \Z_{> 0}, f,g \in C(\R,\C^n)$ bounded and $A, B \in C(\R,\mathrm{Mat}_{n \times n}(\C))$. Suppose equation (\ref{linsys}) has an exponential dichotomy on $\R$ with constants $K,\mu > 0$. Then the inhomogeneous problem,
\begin{align} \partial_x \phi = A(x)\phi + f(x), \ \ \ \phi \in \C^n, \label{lininhom}\end{align}
has a unique bounded solution $\phi(x)$. Furthermore, suppose in addition that $A$ and $B$ are bounded. Let $a,b \in \R$ with $a < b$. Then, for any bounded solution $\psi(x)$ to the inhomogeneous problem,
\begin{align} \partial_x \omega = B(x)\omega + g(x), \ \ \ \omega \in \C^n, \label{lininhom2}\end{align}
we estimate for $x \in [a,b]$
\begin{align}\begin{split}
\|\phi(x) - \psi(x)\| &\leq \frac{K}{\mu} \left(e^{-\mu(x-a)} + e^{-\mu(b-x)}\right)\left(\|\psi\| \|A-B\| + \|f-g\|\right)\\
 &\ \ \ \ \ \ \ + \frac{2K}{\mu}\left(\|\psi\| \sup_{z \in [a,b]} \|A(z)-B(z)\| + \sup_{z \in [a,b]}\|f(z) - g(z)\|\right).\end{split} \label{inhomest}
\end{align}
\end{prop}
\begin{proof}
Denote by $T(x,y)$ the evolution of system (\ref{linsys}). By \cite[Proposition 8.2]{COP} system (\ref{lininhom}) has a unique bounded solution given by
\begin{align*} \phi(x) = \int_{-\infty}^x T^s(x,z)f(z)dz - \int_x^\infty T^u(x,z)f(z)dz, \ \ \ x \in \R.\end{align*}
Now, let $A$ and $B$ be bounded and $\psi$ a bounded solution to (\ref{lininhom2}). Note that $w \colon \R \to \C^n$ defined by $w(x) = \phi(x) - \psi(x)$ is a bounded solution to the inhomogeneous equation,
\begin{align*} \partial_x w = A(x)w + h(x),\end{align*}
where the inhomogeneity $h \colon \R \to \C^n$ given by $h(x) = (A(x) - B(x))\psi(x) + f(x) - g(x)$ is bounded on $\R$ by hypothesis. By applying \cite[Proposition 8.2]{COP} once again we deduce that $w(x)$ is given by
\begin{align} w(x) =& \int_{-\infty}^x T^s(x,z)h(z)dz - \int_x^\infty T^u(x,z)h(z)dz, \ \ \ x \in \R.\label{inhomest2}\end{align}
Now, let $a,b \in \R$ with $a < b$. Estimate (\ref{inhomest}) for $x \in [a,b]$ is achieved by splitting both integrals in expression (\ref{inhomest2}) into two parts. The first integral is split in integrals over $(-\infty,a)$ and over $(a,x)$. Similarly, the second integral is split in integrals over $(x,b)$ and over $(b,\infty)$. This yields four integrals, which can be estimated separately in order to obtain estimate (\ref{inhomest}).
\end{proof}
\subsection{The minimal opening between subspaces}
The minimal opening \cite[Section 13.3]{GOH} is a quantity measuring the `gap' between two subspaces.
\begin{defi}
Let $n \in \Z_{> 0}$. The \emph{minimal opening} between two non-trivial subspaces $\M$ and $\N$ of $\C^n$ is given by
\begin{align*} \eta(\M,\N) = \inf\{\|x-y\| : x \in \M, y \in \N, \max(\|x\|,\|y\|) = 1\}.\end{align*}
\end{defi}
The minimal opening has the useful property that the norm of the projection on $\M$ along $\N$ can be bounded in terms of $\eta(\M,\N)$. This norm estimate is essential for the application of the pasting Lemma \ref{pastingexpdi} in our spectral analysis.
\begin{prop} \label{opening}
Let $n \in \Z_{> 0}$. The following assertions hold true.
\begin{enumerate}
 \item If $P$ is a non-trivial projection on $\C^n$, then it holds
 \begin{align*}\|P\| \leq \frac{1}{\eta(P[\C^n],\ker(P))}.\end{align*}
 \item For non-trivial subspaces $\M$ and $\N$ of $\C^n$ it holds $\eta(\M,\N) \neq 0$ if and only if $\M \cap \N = \{0\}$.
 \item Let $\M_{1,2}$ and $\N_{1,2}$ be non-trivial subspaces of $\C^n$. Suppose that there exists $0 < \delta < 1$ such that for each $v \in \M_i$ there exists a $w \in \N_i$ such that $\|v-w\| \leq \delta \|v\|$ for $i = 1,2$. Then, we have the estimate
 \begin{align*}\eta(\N_1,\N_2) \leq \eta(\M_1,\M_2) + 4\delta.\end{align*}
 \item Let $\Omega \subset \C$ be open and connected. Suppose $\M(\lambda)$ and $\N(\lambda)$ are continuous families of subspaces on $\Omega$, i.e. there exist continuous families of projections $P_\M, P_\N \colon \Omega \to \mathrm{Mat}_{n \times n}(\C)$ such that $P_\M(\lambda)[\C^n] = \M(\lambda)$ and $P_\N(\lambda)[\C^n] = \N(\lambda)$ for $\lambda \in \Omega$. Then, the map $\lambda \mapsto \eta(\M(\lambda),\N(\lambda))$ is also continuous on $\Omega$.
\end{enumerate}
\end{prop}
\begin{proof} The first two assertions are derived in \cite[p. 396]{GOH} and \cite[Proposition 13.2.1]{GOH}, respectively. For the third assertion take $\eps > 0$. There exists $v_1 \in \M_1$ and $v_2 \in \M_2$ with $\max(\|v_1\|,\|v_2\|) = 1$ such that $\|v_1-v_2\| \leq \eta(\M_1,\M_2) + \eps$. Without loss of generality we may assume $\|v_1\| = 1$. By hypothesis there exists $w_1 \in \N_1$ such that $\|v_1 - w_1\| \leq \delta$. Because we have $\delta < 1$, we can normalize $w_1$ and define $z_1 := \frac{w_1}{\|w_1\|}$. One readily estimates $\|v_1 - z_1\| \leq 2\delta$. Similarly, there exists $w_2 \in \N_2$ such that $\|v_2 - w_2\| \leq \delta$. In the case $\|w_2\| > 1$, take $z_2 := \frac{w_2}{\|w_2\|}$. One easily verifies $\|v_2 - z_2\| \leq 2\delta$. In the case $\|w_2\| \leq 1$, we just take $z_2 := w_2$. Finally, we estimate
\begin{align*} \eta(\N_1,\N_2) \leq \|z_1 - z_2\| \leq \|v_1 - v_2\| + \|v_1 - z_1\| + \|v_2 - z_2\| \leq \eta(\M_1,\M_2) + 4\delta + \eps.\end{align*}
Since $\eps$ is arbitrarily chosen, the second assertion follows. Finally, for the fourth assertion let $P_\M(\lambda)$ and $P_\N(\lambda)$ be continuous families of projections on $\Omega$ with ranges $\M(\lambda)$ and $\N(\lambda)$, respectively. With the aid of identities (13.1.4), (13.2.5) and (13.2.7) in \cite{GOH} we derive for $\lambda_0 \in \Omega$
\begin{align*} |\eta(\M(\lambda),\N(\lambda)) - \eta(\M(\lambda_0),\N(\lambda_0))| \leq \sqrt{2}\left(\|P_\M(\lambda) - P_\M(\lambda_0)\| + \|P_\N(\lambda) - P_\N(\lambda_0)\|\right).\end{align*}
This shows that $\lambda \to \eta(\M(\lambda),\N(\lambda))$ is continuous on $\Omega$.
\end{proof}
\section{Proofs of technicalities} \label{A2}
\begin{lem} \label{boundG}
The map $\G$, defined in (\ref{defupsilon}), is bounded at $\infty$.
\end{lem}
\begin{proof}
The coordinate change $(v,q) \mapsto (v,\sqrt{\lambda}w)$ puts system (\ref{redstab}) into the form,
\begin{align}
\left\{\!
\begin{array}{rcl} D_2 \partial_x v &=& \sqrt{\lambda}w \\ \partial_x w &=& \left(\frac{\partial_v G(u_0,v_{\ho}(x),0)}{\sqrt{\lambda}} + \sqrt{\lambda}\right)v \end{array}\right., \ \ \ v,w \in \C^n.  \label{linsys8}
\end{align}
If $\lambda > 0$ is sufficiently large, system (\ref{linsys8}) has by Proposition \ref{RoughnessR} an exponential dichotomy on $\R$ with constants $K_1,\mu(\lambda) > 0$, where $K_1$ is independent of $\lambda$ and $\mu(\lambda) = \tfrac{1}{2}\|\sqrt{D_2}\|^{-1} \sqrt{\lambda}$. Therefore, system (\ref{redstab}) has an exponential dichotomy on $\R$ with constants $K_2(\lambda),\mu(\lambda) > 0$, where $K_2 = \sqrt{\lambda}K_2'$ and $K_2'$ is independent of $\lambda$. Note that $\frac{K_2(\lambda)}{\mu(\lambda)}$ is $\lambda$-independent. Moreover, since $G$ vanishes at $v = 0$ by \ref{assS1}, $\A_{21,0}$ is exponentially localized by Remark \ref{homoclin}. Combining these fact with Proposition \ref{inhomexpdi} yields that, for $\lambda > 0$ sufficiently large, there exists $K_3,\mu_3 > 0$, independent of $\lambda$, such that $\|\X_{in}(x,\lambda)\| \leq K_3 e^{-\mu_3|x|}$. This concludes the proof.
\end{proof}
\begin{lem} \label{C2}
Let $m = 1$. The trace $t(\lambda) = \mathrm{Tr}(\Upsilon(\lambda)\T_s(2\check{L}_0,0,\lambda))$ diverges to $\infty$ as $\lambda \to \infty$.
\end{lem}
\begin{proof} Take $\lambda > 0$. Consider system,
\begin{align}
\left\{\!
\begin{array}{rcl} \partial_\xx u &=& \sqrt{\lambda} p \\ \partial_\xx p &=& \left(\frac{1}{\sqrt{\lambda}} \ \frac{\partial H_1}{\partial u} (u_\s(\xx),0,0) + \sqrt{\lambda}\right)u\end{array}\right., \ \ \ u,p \in \C, \label{linsys13}
\end{align}
with evolution $\T_{s1}(\xx,\yy,\lambda)$. Denote by $\T_{s2}(\xx,\yy,\lambda)$ the evolution of the autonomous system,
\begin{align}
\left\{\!
\begin{array}{rcl} \partial_\xx u &=& \sqrt{\lambda} p \\ \partial_\xx p &=& \sqrt{\lambda} u \end{array}\right., \ \ \ u,p \in \C. \label{linsys12}
\end{align}
Let $M > 0$ be a bound of $\partial_u H_1 (u_\s(\cdot),0,0)$ on $[0,2\check{L}_0]$. System (\ref{linsys12}) has bounded growth with constants $K = 1$ and $\mu = \sqrt{\lambda}$. Therefore, it holds by Proposition \ref{Palmer} for $\lambda > 0$ sufficiently large
\begin{align} \|\T_{s1}(2\check{L}_0,0,\lambda) - \T_{s2}(2\check{L}_0,0,\lambda)\| \leq \frac{2M\check{L}_0e^{M}}{\sqrt{\lambda}}e^{2\sqrt{\lambda}\check{L}_0}. \label{slowestim2}\end{align}
On the other hand, system (\ref{slowintr}) is equivalent to system (\ref{linsys13}) upon performing a coordinate change. Indeed, take $C_\lambda := \left(\begin{smallmatrix} I & 0 \\ 0 & \sqrt{\lambda} \end{smallmatrix}\right)$. It holds
\begin{align} C_\lambda \T_{s1}(2\check{L}_0,0,\lambda) C_\lambda^{-1} = \T_s(2\check{L}_0,0,\lambda), \label{slowestim1}\end{align}
We approximate $t(\lambda)$ for $\lambda > 0$ sufficiently large with the aid of Lemma \ref{boundG}, (\ref{slowestim2}) and (\ref{slowestim1})
\begin{align*} \begin{split}t(\lambda) &= \mathrm{tr}\left(\left(\begin{array}{cc} 1 & 0 \\ \frac{\G(\lambda)}{\sqrt{\lambda}} & 1\end{array}\right)T_{s1}(2\check{L}_0,0,\lambda)\right)
= \mathrm{tr}\left(\left(\begin{array}{cc} 1 & 0 \\ \frac{\G(\lambda)}{\sqrt{\lambda}} & 1\end{array}\right)\T_{s2}(2\check{L}_0,0,\lambda)\right) + \ord\left(\tfrac{1}{\sqrt{\lambda}}e^{2\sqrt{\lambda}\check{L}_0}\right) \\ &= e^{-2\sqrt{\lambda}\check{L}_0} + e^{2\sqrt{\lambda}\check{L}_0} + \tfrac{1}{2\sqrt{\lambda}}\G(\lambda)(e^{2\sqrt{\lambda}\check{L}_0} - e^{-2\sqrt{\lambda}\check{L}_0}) + \ord\left(\tfrac{1}{\sqrt{\lambda}}e^{2\sqrt{\lambda}\check{L}_0}\right) \\&= e^{2\sqrt{\lambda}\check{L}_0}\left(1 +\ord\left(\tfrac{1}{\sqrt{\lambda}}\right)\right),\end{split}\end{align*}
In the latter approximation we have used explicit expressions of the evolution $\T_{s2}(\xx,\yy,\lambda)$ of the linear autonomous system (\ref{linsys12}). We conclude $t(\lambda) \to \infty$ as $\lambda \to \infty$.
\end{proof}
\begin{proof}[Proof of Proposition \ref{slowEvans}-2] Let $\lambda > 0$. Putting $\yy = \sqrt{\lambda} \xx$ and $p = \sqrt{\lambda D_1} r$ rescales system (\ref{slowintr}) into
\begin{align}
\left\{\!\begin{array}{rcl} \sqrt{D_1} \partial_\yy u &=& r \\ \sqrt{D_1} \partial_\yy r &=& \left(\frac{\partial_u H_1 \left(u_\s\left(\lambda^{-1/2}\yy\right),0,0\right)}{\lambda} + I\right)u \end{array}\right., \ \ \ u,r \in \C^{m}. \label{linsys10}
\end{align}
Denote by $\T_{s1}(\yy,\check{z},\lambda)$ the evolution of system (\ref{linsys10}). It holds
\begin{align} C_\lambda \Upsilon_1(\lambda) \T_{s1}(2\sqrt{\lambda}\check{L}_0,0,\lambda) C_\lambda^{-1} = \Upsilon(\lambda)\T_s(2\check{L}_0,0,\lambda), \label{slowesti5}\end{align}
with $C_\lambda := \left(\begin{smallmatrix} I & 0 \\ 0 & \sqrt{\lambda D_1}\end{smallmatrix}\right)$ and $\Upsilon_1(\lambda) := \left(\begin{smallmatrix} I & 0 \\ \left[\sqrt{\lambda D_1}\right]^{-1}G(\lambda) & I\end{smallmatrix}\right)$. System (\ref{linsys10}) is close to,
\begin{align}
\left\{\!\begin{array}{rcl} \sqrt{D_1} \partial_\yy u &=& r \\ \sqrt{D_1} \partial_\yy r &=& u \end{array}\right., \ \ \ u,r \in \C^{m}. \label{linsys11}
\end{align}
Clearly, (\ref{linsys11}) has an exponential dichotomy on $\R$ with constants $K_1 = 1$ and $\mu_1 = \|D_1\|^{-1/2}$. The corresponding rank $m$ projections $P_1 = \frac{1}{2}\left(\begin{smallmatrix} I & -I \\ -I & I\end{smallmatrix}\right)$ are independent of $x$, since (\ref{linsys11}) is autonomous. Let $M_1 > 0$ be a bound of  $\partial_u H_1(u_\s(\cdot),0,0)$ on $[0,2\check{L}_0]$. So, by roughness (Proposition \ref{roughnessintervals}) system (\ref{linsys10}) has, for $\lambda > 0$ sufficiently large, an exponential dichotomy on $[0,2\sqrt{\lambda}\check{L}_0]$ with constants $K_2,\mu_2 > 0$, independent of $\lambda$. The corresponding projections $P_2(x,\lambda), x \in [0,2\sqrt{\lambda}\check{L}_0]$ satisfy
\begin{align}\|P_2(x,\lambda) - P_1\| \leq \frac{144M_1}{\mu_1\lambda}. \label{slowesti1}\end{align}
Now, choose bases $B_1^{u,s} \in \mathrm{Mat}_{2m \times m}(\C)$ of $P_1[\C^{2m}] = B_1^s[\C^m]$ and $\ker(P_1) = B_1^u[\C^m]$. Define $B_2^s(\lambda) = P_2(0,\lambda)B_1^s$ and $B_2^u(\lambda) = (I-P_2(2\sqrt{\lambda}\check{L}_0,\lambda))B_1^u$. By estimate (\ref{slowesti1}) it holds
\begin{align}  \|B_2^{u,s}(\lambda) - B_1^{u,s}\| = \ord\left(\frac{1}{\lambda}\right). \label{slowesti2}\end{align}
Consider the invertible matrix,
\begin{align*}\h(\lambda) := \left(\T_{s1}(0,2\sqrt{\lambda}\check{L}_0,\lambda)B_2^u(\lambda), B_2^s(\lambda)\right).\end{align*}
By Lemma \ref{boundG} and (\ref{slowesti2}) we have
\begin{align*} \left(\Upsilon_1(\lambda) \T_{s1}(2\sqrt{\lambda}\check{L}_0,0,\lambda) - \gamma\right)\h(\lambda) = \left(B_1^u,-\gamma B_1^s\right) +  \ord\left(\frac{1}{\sqrt{\lambda}}\right).\end{align*}
Taking determinants in the previous expression gives by (\ref{slowesti5})
\begin{align*} \E_{s,0}(\lambda,\gamma)\det(\h(\lambda)) &= \det\left(\Upsilon_1(\lambda) \T_{s1}(2\sqrt{\lambda}\check{L}_0,0,\lambda) - \gamma\right)\det(\h(\lambda)) = (-\gamma)^m \det\left(B_1^u,B_1^s\right) + \ord\left(\frac{1}{\sqrt{\lambda}}\right).\end{align*}
By construction $\det(B_1^u,B_1^s)$ is non-zero and independent of $\lambda$. Combining this with $\det(\h(\lambda)) \neq 0$ yields that $\E_{s,0}(\lambda,\gamma)$ must be non-trivial for $\lambda > 0$ sufficiently large.
\end{proof}

\paragraph*{Acknowledgement.} Bj\"orn de Rijk would like to thank Frits Veerman for taking the time to introduce him into the subject and for all the helpful discussions.

\bibliographystyle{plain}
\bibliography{mybib}

\end{document}